\documentclass[11 pt,reqno]{amsart}

\usepackage[latin1]{inputenc}
\usepackage{amsmath,amsthm,amssymb,amscd,mathrsfs,mathtools,xcolor,accents}
\usepackage[shortlabels]{enumitem}
\usepackage[perpage]{footmisc}
\usepackage[normalem]{ulem}
\usepackage{graphicx}
\definecolor{darkgreen}{rgb}{0.0, 0.6, 0.13}

\usepackage{bbm} 
\usepackage{pifont} 
\usepackage{hyperref}

\newtheorem{thm}{Theorem}[section]
\newtheorem{cor}[thm]{Corollary}

\newtheorem{lem}[thm]{Lemma}
\newtheorem{prop}[thm]{Proposition}
\theoremstyle{definition}

\theoremstyle{remark}
\newtheorem{rem}[thm]{Remark}

\numberwithin{equation}{section}
\def\sqw{\hbox{\rlap{\leavevmode\raise.3ex\hbox{$\sqcap$}}$
\sqcup$}}
\def\findem{\ifmmode\sqw\else{\ifhmode\unskip\fi\nobreak\hfil
\penalty50\hskip1em\null\nobreak\hfil\sqw
\parfillskip=0pt\finalhyphendemerits=0\endgraf}\fi}

\allowdisplaybreaks[4]

\begin{document}
\title[Asymptotic Stability  of Domain Walls for 1D SM]{Asymptotic Stability  of Domain Walls for One Dimensional Schr\"odinger Map with Dzyaloshinskii-Moriya Interaction}

\author[Z. Li]{Ze Li}
\address{Ze Li
\newline\indent
School of Mathematics and Statistics, Ningbo University
\newline\indent
Ningbo, 315211, Zhejiang, P.R. China
}
\email{rikudosennin@163.com}

\author[Z. Wang]{Zifan Wang}
\address{School of Mathematical Sciences, University of Science and Technology of China,
\newline\indent
Hefei 230026, Anhui, China}
\email{math2013@mail.ustc.edu.cn}
\author[L. Zhao]{Lifeng Zhao}
\address{School of Mathematical Sciences, University of Science and Technology of China,
\newline\indent
Hefei 230026, Anhui, China}
\email{zhaolf@ustc.edu.cn}


\dedicatory{}

\begin{abstract}
In this paper, we study the long time dynamics of small perturbations of static domain walls in weighted Sobolev spaces for the one dimensional Schr\"odinger map with constant Dzyaloshinskii-Moriya interaction. The perturbation equation contains derivative nonlinearities, and its linearized operator has a threshold resonance. The results are twofold. (i) We prove the asymptotic stability of the domain wall modulo the translation and rotation symmetries. More precisely, the corresponding modulation parameters converge as time tends to infinity, while the radiation decays at the sharp rate $t^{-1/2}$. (ii) We prove that the radiation exhibits modified scattering, with an explicit logarithmic phase correction generated by the long range cubic interaction. The proof develops geometric reduction to overcome the loss of derivatives and exploits structural cancellations in the nonlinear interactions.
\end{abstract}

\maketitle

\section{Introduction} 
\subsection{ Schr\"odinger Map with Dzyaloshinskii--Moriya interaction}

We consider the one dimensional Schr\"odinger map describing the normalized magnetization field $m:\mathbb R_t\times\mathbb R_x\to\mathbb S^2$ in a ferromagnetic nanowire with Dzyaloshinskii--Moriya interaction (DMI):
\begin{equation}\label{mLLM}
\begin{cases}
\partial_t m=m\wedge\left(\partial_x^2m+2\eta e_1\wedge\partial_xm+m_1e_1\right),\\
m(0,x)=m_0(x).
\end{cases}
\end{equation}
Here $\wedge$ denotes the cross product in $\mathbb R^3$ and $(e_1,e_2,e_3)$ is the standard orthonormal basis of $\mathbb R^3$. The term $m_1e_1$ describes the easy axis anisotropy, with $\pm e_1$ as the preferred magnetization directions. The DMI is represented by the derivative term $2\eta e_1\wedge\partial_xm$ where $\eta$ is the constant DMI coupling parameter, which arises from spin orbit coupling in magnetic systems with broken inversion symmetry and favors chiral magnetic textures \cite{D1958,Moriya1960,TRC2012,EmoriEtAl2013}.  Throughout the paper, we fix $|\eta|<1$ and set $\mu=\sqrt{1-\eta^2}$. The local and global well-posedness of \eqref{mLLM} are established in Appendix~\ref{applwp} and Appendix~\ref{appgwp}, respectively. 

The equation \eqref{mLLM} is invariant under translations and rotations about the easy axis. More precisely, if $m(t,x)$ is a solution, then $R_\phi m(t,x-y)$ is also a solution for every $y,\phi\in\mathbb R$, with initial data $R_\phi m_0(x-y)$, where $R_\phi$ is the rotation about $e_1$ given by
\begin{equation}\label{eRphi}
R_\phi=
\begin{pmatrix}
1&0&0\\
0&\cos\phi&-\sin\phi\\
0&\sin\phi&\cos\phi
\end{pmatrix}.
\end{equation}

Equation \eqref{mLLM} conserves the energy 
\begin{equation}\label{H2eta}
\mathcal H_\eta^1(m)=\int_{\mathbb R}\left(|\nabla_\eta m|^2+\mu^2(1-m_1^2)\right)\,dx,
\end{equation}
where $\nabla_\eta=\partial_x+\eta e_1\wedge$. In addition to the energy, the equation possesses an infinite sequence of higher order conservation laws, see
\cite{F1984,BR1981}. In particular, the second order energy
\begin{equation}\label{H4eta}
\begin{aligned}
\mathcal H_\eta^2(m)=&\int_{\mathbb R}\Bigl[|\nabla_\eta^2m|^2-\frac54|\nabla_\eta m|^4+\mu^2\left(\frac32m_1^2|\nabla_\eta m|^2-|\partial_xm_1|^2\right)+\frac{\mu^4}{4}(1-m_1^4)\Bigr]\,dx
\end{aligned}
\end{equation}
is also conserved. Both conservation laws are derived in Appendix~\ref{appcsv}.

The equation (\ref{mLLM}) admits finite energy stationary solutions  connecting the preferred magnetization states $-e_1$ and $e_1$, known as static magnetic domain walls. Magnetic domain walls are of interest for spintronic data storage: in racetrack memory, information is encoded in magnetic domain patterns and transported along nanowires through controlled domain wall motion \cite{PHT2008}. C\^ote and Ignat \cite{CoteIgnat} classified all such static domain walls up to translations and rotations about the easy axis and showed that they minimize the functional $\mathcal H_\eta^1$ subject to the boundary conditions $m(\pm\infty)=\pm e_1$. In this paper, we focus on the static domain wall given by
\begin{equation*}
\omega_*(x)=\left(T_\eta(x),S_\eta(x)\cos(\eta x),-S_\eta(x)\sin(\eta x)\right),
\end{equation*}
where $T_\eta(x)=\tanh(\mu x)$ and $S_\eta(x)=\operatorname{sech}(\mu x)$. 

The stability of domain walls has been extensively studied for Landau--Lifshitz--Gilbert equations with positive damping. In the absence of DMI and an applied magnetic field, Carbou and Labb\'e \cite{CarbouLabbe2006} proved orbital stability in $H^2$ and asymptotic stability in $H^1$ for Bloch walls, modulo translations and rotations about the easy axis. For a constant applied field, Jizzini \cite{Jizzini2011} identified the sharp field strength condition for $H^1$ stability. Goussev, Robbins, and Slastikov \cite{GRS2010} constructed explicit precessing domain wall solutions under longitudinal applied fields of arbitrary strength and time dependence. Gou, Goussev, Robbins, and Slastikov \cite{GouEtAl2011} subsequently established their linear asymptotic stability for sufficiently small applied fields. These results concern models without DMI. In the presence of DMI, C\^ote and Ignat \cite{CoteIgnat} proved nonlinear asymptotic stability in $H^1$ of precessing domain walls under a small time dependent longitudinal applied field, with exponential convergence modulo translations and rotations about the easy axis. C\^ote and Ferri\`ere \cite{CF2024} extended this framework to configurations of two well separated domain walls moving apart. C\^ote and Ignat \cite[Section~1.4]{CoteIgnat} pointed out: ``{\it In case of no damping, i.e., $\alpha=0$, the equation will have the flavour of a Schr\"odinger map equation, for which stability questions of domain walls are widely open}.'' 

In this paper, we establish the asymptotic stability of $\omega_*$ under \eqref{mLLM} for sufficiently small and localized perturbations with high Sobolev regularity, thereby answering this question in the perturbative regime considered here. No parity assumption is imposed on the initial perturbation. Both the position and rotation angle of the wall converge, and the original magnetization field admits an explicit asymptotic decomposition into a limiting wall and radiation with a logarithmic phase correction.

\subsection{Main result}

Let $\langle \xi\rangle_{\mu}=\sqrt{\mu^2+\xi^2}$ and $\langle \xi\rangle=\sqrt{1+\xi^2}$. We use the following tangent frame along $\omega_*$:
$$n_*(x)=R_{-\eta x}(-S_\eta(x),T_\eta(x),0),\qquad p_*(x)=R_{-\eta x}e_3,$$
so that $(\omega_*,n_*,p_*)$ is a positively oriented orthonormal frame along the reference domain wall. For a solution close to the domain wall, we write
\begin{align}\label{kcord}
 m(t,x)=R_{\phi(t)}\Big((1+\alpha)\omega_*+\beta n_*+\gamma p_*\Big)(t,x-y(t)),
\end{align}
where $\alpha$, $\beta$, and $\gamma$ are real valued functions of $(t,x)$ satisfying the sphere constraint $(1+\alpha)^2+\beta^2+\gamma^2=1$. Let $u=\beta+i\gamma$. 

\begin{thm}\label{thmres}
Let $|\eta|<1$ and let $N\ge6$ be an integer. There exists a sufficiently small constant $  \varepsilon_1>0$ such that, if the initial data $m_0:\mathbb R\to\mathbb S^2$ satisfies
\begin{equation}\label{inismall}
\|m_0-\omega_*\|_{H^N}+\|x(m_0-\omega_*)\|_{H^{N-1}} \ll \varepsilon_1,
\end{equation}
then the solution to \eqref{mLLM} is global and has the following properties:\\
{\rm (i) (Asymptotic stability)}
The solution admits the above decomposition for all $t\ge0$, with modulation parameters $y,\phi\in C^1([0,\infty), \mathbb R)$ satisfying the asymptotics
\begin{equation}\label{asphiy}
y(t)=y_\infty+O\!\left(\varepsilon_{1}\langle t\rangle^{-\delta_*}\right),\quad\phi(t)=\phi_\infty+O\!\left(\varepsilon_{1}\langle t\rangle^{-\delta_*}\right)
\end{equation}
for constants $y_\infty$,$\phi_\infty$ and some $\delta_*>0$.
Moreover, the magnetization converges uniformly to the modulated domain wall at the sharp dispersive rate:
\begin{equation}\label{Lindeca}
\bigl\|m(t)-R_{\phi(t)}\omega_*(\cdot-y(t))\bigr\|_{L^\infty_x}\lesssim\varepsilon_1\langle t\rangle^{-1/2}.
\end{equation}
{\rm (ii) (Modified scattering for the radiation)} For $t\ge1$. Let $\zeta=\frac{x}{2t}.$
There exists $W_\infty\in L^2(\mathbb R)\cap L^\infty(\mathbb R)$ such that, 
\begin{equation}\label{L2aspthm}
\Big\|m(t)-R_{\phi_\infty}\Big[\omega_*+\operatorname{Re}\mathbf W(t)n_*+\operatorname{Im}\mathbf W(t)p_*\Big](\cdot-y_\infty)\Big\|_{L^2_x}\longrightarrow0,\quad\text{as }t\to\infty,
\end{equation}
where
\begin{align*}
\mathbf W(t,x)=&\frac{i\zeta-\mu T_\eta(x)}{\sqrt{2t}\langle\zeta\rangle_\mu}W_\infty(\zeta)\exp\!\left(\frac{ix^2}{4t}-i\mu^2t-\frac{i\pi}{4}+\frac{i}{4}\langle\zeta\rangle_\mu^2|W_\infty(\zeta)|^2\log t\right).
\end{align*}
\end{thm}

\begin{rem}
\eqref{Lindeca} is equivalent to
$$\|u(t)\|_{L^\infty_x}\lesssim\varepsilon_1\langle t\rangle^{-1/2},\qquad\|\alpha(t)\|_{L^\infty_x}\lesssim\varepsilon_1^2\langle t\rangle^{-1},$$
where $u:=\beta+i\gamma$ and $\alpha$ are defined by (\ref{kcord}). 
Together with \eqref{asphiy}, \eqref{Lindeca}  yields
$$m(t)=R_{\phi_\infty}\omega_*(\cdot-y_\infty)+O_{L^\infty_x}\!\left(\varepsilon_1\langle t\rangle^{-\delta_*}\right),\quad\text{as }t\to\infty.$$
\end{rem}

\begin{rem} The equation \eqref{mLLM} (or equivalently (\ref{LLM}) below) also  admits  traveling waves $u(x-ct)$ for all $c\in (-1,1)$,  see \cite{LG2018,LW2010} for the corresponding existence results. For (\ref{LLM}), de Laire-Gravejat  \cite{LG2015} studied the Cauchy problem and the orbital stability of the travelling waves, and Bahri \cite{B2018} proved the asymptotic stability in the weak topology for $c\neq 0$. Since domain walls are static, Theorem \ref{thmres} indeed corresponds to $c=0$. 
 As mentioned above, different from the nonlinear stability results \cite{CoteIgnat,CF2024} in the damping case,  Theorem \ref{thmres} tries to understand  the global dynamics near the domain walls in the dispersive case.
\end{rem}

\subsection{Strategy of the proof}
Our proof uses two complementary frames to describe the dynamics near the domain wall. The Frenet frame adapted to the reference domain wall provides the setting for proving full asymptotic stability through distorted Fourier analysis, modulation, and normal form transformations. The perturbation equation combines derivative nonlinearities with a threshold resonance of the linearized operator, making it difficult to close the estimates required for the long time analysis directly. To overcome this difficulty, we introduce a geometric reduction based on the domain wall and the complex structure of \eqref{mLLM}, constructing a frame along the evolving map and thereby transforming the Landau--Lifshitz equation into a complex scalar equation. The resulting equation can be estimated without loss of derivatives, and these estimates can then be transferred back to the perturbation equation in the Frenet frame, allowing us to complete the proof of Theorem~\ref{thmres}.

Let us begin by deriving the perturbation equation for $\mathbf M =R_{\eta x}m(t,x)$ in the Frenet frame adapted to the reference domain wall. It is easy to see that the equation (\ref{mLLM}) becomes the equation (\ref{LLM}) of $\mathbf M$.  We write
\begin{equation*}
\mathbf M(t,x)=R_{\phi(t)+\eta y(t)}
\Bigl[(1+\alpha)\Omega_*+\beta N_*+\gamma P_*\Bigr](x-y(t)).
\end{equation*}
Here $\Omega_*$ is defined in (\ref{ab}), $N_*$ and $P_*$ are defined in \eqref{fram}. Substituting this decomposition into \eqref{LLM} and letting $u=\beta+i\gamma$, we arrive at
\begin{equation}
\begin{aligned}
i\partial_tu-\mathcal L_\eta u
=&\mu S_\eta(u-\overline u)\partial_xu+u\partial_x^2\alpha-\alpha\partial_x^2u+\cdots,
\end{aligned}
\label{pereq}
\end{equation}
where $\mathcal L_\eta=-\partial_x^2+\mu^2(1-2S_\eta^2)$. To remove the quasilinear second order terms,  we introduce a complex variable $z$ that incorporates the sphere constraint into the perturbation equation:
\begin{equation*}
u=\frac{z}{1+\frac14|z|^2}.
\end{equation*}
Substituting these relations into \eqref{pereq}, we obtain
\begin{equation}
i\partial_tz-\mathcal L_\eta z = \mathcal N_2+\mathcal N_3+\mathcal N_{\ge4}+\mathcal Mod,
\label{eqz6}
\end{equation}
where
\begin{align*}
\mathcal N_2=-\frac12V_2z(z+\overline z)+\frac12V_3(z-\overline z)\partial_xz,\quad
\mathcal N_3=\frac12\overline z(\partial_xz)^2-\frac12V_1|z|^2z,  
\end{align*}
while the higher order term $\mathcal N_{\ge4}$ and the modulation term $\mathcal Mod$ are defined in~\eqref{zevol}. Note that the resulting equation still contains derivative nonlinearities, notably $\frac12\overline z(\partial_xz)^2$, which obstruct a direct closure of the weighted profile estimates.

\noindent\emph{Geometric reduction and decay estimates.}
The derivative nonlinearities in \eqref{eqz6} make it difficult to close the weighted profile estimates directly. To address this difficulty, we exploit the first order identity satisfied by the reference wall,
$$\partial_x\Omega_*=\mu\bigl(e_1-(\Omega_*)_1\Omega_*\bigr),$$
and introduce the tangent vector field
$$\mathbf B=\partial_x\mathbf M-\mu\bigl(e_1-\mathbf M_1\mathbf M\bigr).$$
This field measures the deviation of the magnetization from the reference wall and vanishes when $\mathbf M=\Omega_*$. Then $\mathbf B$ satisfies
\begin{equation*}
\nabla_t\mathbf B=\mathbf M\wedge\left(\nabla_x^2\mathbf B-(\mu^2+\mu\mathbf B_1)\mathbf B\right),
\end{equation*}
The operation $\mathbf M\wedge(\,\cdot\,)$ defines a complex structure on the tangent space $T_{\mathbf M}\mathbb S^2$, so that the covariant evolution of $\mathbf B$ can be viewed as an equation of Schr\"odinger type. To exploit this structure and avoid the loss of derivatives in \eqref{eqz6}, we choose an orthonormal frame $(\mathbf E_1,\mathbf E_2)$ along $\mathbf M$ that is parallel in $x$ direction. The key point in the choice of the frame is the identity
\begin{align*}
\partial_x\left(\nabla_t\mathbf E_1\cdot \mathbf E_2+\mu \mathbf B_1\right)
=-\frac12\partial_x|\mathbf B|^2.
\end{align*}
Indeed, after integration in $x$, this identity expresses the connection coefficient in the time direction in terms of $|\mathbf B|^2$, which closes the equation for the complex coordinate of $\mathbf B$. More precisely, defining
$$
\mathbf Q=\frac12\left(\mathbf B\cdot\mathbf E_1+i\mathbf B\cdot\mathbf E_2\right),
$$
we have $|\mathbf B|^2=4|\mathbf Q|^2$, and $\mathbf Q$ satisfies
\begin{equation*}
i\partial_t\mathbf Q+(\partial_x^2-\mu^2)\mathbf Q+2|\mathbf Q|^2\mathbf Q=0.
\end{equation*}
 Here, by exploiting the equation satisfied by the domain wall, the complex structure of the sphere, and an orthonormal parallel frame constructed to eliminate the loss of derivatives, we geometrically reduce the Schr\"odinger map to the cubic  Schr\"odinger equation. The weighted estimates for the profile of the complex scalar field $\mathbf Q$ are derived by the spacetime resonance method in Subsection~\ref{4}. To transfer these estimates back to the perturbation variable $z$, we apply the geometric reduction developed above to \eqref{phileft} and define the corresponding complex scalar field $Q$. The relation \eqref{phidef} between $M$ and $\mathbf{M}$ gives
\begin{equation*}
Q(t,x)=e^{i(\phi(t)+\eta y(t))}\mathbf Q\bigl(t,x+y(t)\bigr).
\end{equation*}
It remains to relate $Q$ to $z$. Since $Q$ is defined using the spatially parallel frame $(E_1,E_2)$ along $M$, we express this frame in terms of $z$. We first construct an explicit orthonormal frame $(\mathcal E_1,\mathcal E_2)$ along $M$. The two frames are related by
\begin{equation*}
E_1=\cos\chi\,\mathcal E_1-\sin\chi\,\mathcal E_2,\qquad
E_2=\sin\chi\,\mathcal E_1+\cos\chi\,\mathcal E_2,
\end{equation*}
where
\begin{equation*}
\chi(t,x)=-\int_{-\infty}^x
\frac{\frac12\operatorname{Im}(\overline z\,\partial_sz)+\mu S_\eta\operatorname{Im}z}
{1+\frac14|z|^2}(t,s){\rm d}s.
\end{equation*}
Thus we obtain the explicit relation
\begin{equation*}
Q=-\frac12e^{i\chi}\frac{D_\eta^*z+\frac{\mu S_\eta}{2}z^2}{1+\frac14|z|^2}.
\end{equation*}
which is the starting point for the decay estimates in Section~\ref{sectransf}.
The explicit relation above allows us to transfer the estimates established for the profile of $\mathbf Q$ back to $z$, see Section~\ref{sectransf}.
The above geometric reduction enables us to prove  the desired decay estimate (\ref{Lindeca})  in  $L^{\infty}$.   To obtain the asymptotic behaviors of both parameters $y(t)$, $\phi(t)$ and solution $m(t)$ as  (\ref{L2aspthm}), we need to further study the equation  (\ref{eqz6}) of $z$. 

\noindent\emph{Quadratic null structure.} The linearized operator  $\mathcal L_\eta$ has a threshold resonance at the bottom of its continuous spectrum. Indeed,
\begin{equation*}
\mathcal L_\eta T_\eta=\mu^2T_\eta,\qquad
T_\eta\in L^\infty(\mathbb R)\setminus L^2(\mathbb R).
\end{equation*}
The slow local decay associated with the threshold resonance makes the quadratic terms in \eqref{eqz6} difficult to treat by direct estimates. To exploit their oscillatory structure, we introduce the profile
$$g(t)=e^{it\mathcal L_\eta}z(t).$$
On the distorted Fourier side, 
\begin{align*}
\partial_t\widetilde g(t,\xi)
={}&\frac{i}{2}\int_{\mathbb R^2}
e^{it(\langle\xi\rangle_\mu^2-\langle a\rangle_\mu^2-\langle b\rangle_\mu^2)}
q_{++}(\xi,a,b)\widetilde g(t,a)\widetilde g(t,b){\rm d}a{\rm d}b\\
&+\frac{i}{2}\int_{\mathbb R^2}
e^{it(\langle\xi\rangle_\mu^2-\langle a\rangle_\mu^2+\langle b\rangle_\mu^2)}
q_{+-}(\xi,a,b)\widetilde g(t,a)\overline{\widetilde g(t,b)}{\rm d}a{\rm d}b+\cdots,
\end{align*}
where 
\begin{align*}
q_{++}(\xi,a,b)
&=\frac{1}{4\sqrt{2\pi}}\bigl(\langle\xi\rangle_\mu^2-\langle a\rangle_\mu^2-\langle b\rangle_\mu^2\bigr)\frac{\langle\xi\rangle_\mu}{\langle a\rangle_\mu\langle b\rangle_\mu}\operatorname{sech}\left(\frac{\pi(\xi-a-b)}{2\mu}\right),\\
q_{+-}(\xi,a,b)&=-\frac{1}{2\sqrt{2\pi}}\bigl(\langle\xi\rangle_\mu^2-\langle a\rangle_\mu^2+\langle b\rangle_\mu^2\bigr)\frac{\langle a\rangle_\mu}{\langle\xi\rangle_\mu\langle b\rangle_\mu}\operatorname{sech}\left(\frac{\pi(\xi-a+b)}{2\mu}\right).
\end{align*}
The key observation is that $q_{++}$ and $q_{+-}$ vanish on the time resonant sets:
$$
\{(a,b)\in\mathbb{R}\times\mathbb{R}:\langle\xi\rangle_\mu^2-\langle a\rangle_\mu^2-\langle b\rangle_\mu^2=0\}
\quad\text{and}\quad
\{(a,b)\in\mathbb{R}\times\mathbb{R}:\langle\xi\rangle_\mu^2-\langle a\rangle_\mu^2+\langle b\rangle_\mu^2=0\}
$$
This null structure allows us to divide by the phase in the normal form transformation without introducing a singularity on the time resonant sets. The cubic terms generated by the normal form are retained in the subsequent modified scattering analysis, together with the original cubic nonlinearity,  see Section~\ref{ssca}.\\

\noindent\emph{Convergence of the modulation parameters.}
The dispersive decay of the radiation does not by itself imply convergence of the modulation parameters. Indeed, a direct estimate of the quadratic terms in the modulation equations gives only
$$
|\dot y(t)|+|\dot\phi(t)|\lesssim\varepsilon_1^2\langle t\rangle^{-1},
$$
which is not integrable in time. To obtain the additional decay, we replace the leading quadratic functional $R_2(z)$ by $R_2(z^*)$, where $z^*=-2D_\eta H_\eta^{-1}Q$ and note that
$$R_2(z^*)=-2\mu\int_{\mathbb R}S_\eta^2\overline{z^*}Q\,{\rm d}x.$$
Combining this identity with the stationary phase expansions \eqref{Qsta1} and \eqref{zsta7} of $Q$ and $z^*$ reveals two additional structure that improve the decay: the imaginary part contains an factor $\frac{x}{2t}$, while in the real part the zero frequency contribution vanishes because
$$\int_{\mathbb R}S_\eta^2T_\eta\,{\rm d}x=0.$$
This cancellation uses the parity of the fixed wall coefficients, not a parity assumption on the perturbation. And this yields
$$ |\dot y(t)|+|\dot\phi(t)|\lesssim \varepsilon_1^2\langle t\rangle^{-1-\delta_*}$$ 
for some $\delta_*>0$. The resulting time integrability gives convergence of both modulation parameters, determining the limiting position and rotation angle of the domain wall.\\

\noindent\emph{Modified scattering.}
After the quadratic normal form transformation, stationary phase applied to the Dirac delta contribution yields the leading resonant term
$$\frac{i}{4t}\langle\xi\rangle_\mu^2|\widetilde g(t,\xi)|^2\widetilde g(t,\xi).$$
A crucial cancellation is found in the coefficient of the principal value contribution. More precisely, we have
\begin{align*}
1-\overline{t_\eta(\xi)}t_\eta(a)t_\eta(b)\overline{t_\eta(c)}
=&\partial_b\Phi_3\left(\frac{i\mu}{(\xi+i\mu)(a-i\mu)}-\frac{i\mu\overline{t_\eta(\xi)}t_\eta(a)}{(b-i\mu)(c+i\mu)}\right)-\frac{2i\mu p}{(\xi+i\mu)(a-i\mu)}.
\end{align*}
Here $c=a+b-\xi+p$, and $\partial_b$ is taken at fixed $(\xi,a,p)$. The factor $p$ cancels the singularity of
$$\operatorname{p.v.}\operatorname{csch}\left(\frac{\pi p}{2\mu}\right),$$
whereas the factor $\partial_b\Phi_3$ allows an integration by parts in $b$, producing the additional time decay needed to make the principal value contribution integrable.
All remaining contributions are integrable in time in $L^\infty_\xi$. Removing the logarithmic phase generated by the leading resonant term then yields modified scattering,  see Section~\ref{submodi}.

\subsection{Related literature}
Early results on the asymptotic stability of solitary waves for one dimensional nonlinear Schr\"odinger equations were obtained by Buslaev and Perelman \cite{BP1992}, \cite{BP1992b}, Buslaev and Sulem \cite{BS2003}, Cuccagna \cite{Cuccagna2001} and, for the integrable cubic NLS, Cuccagna and Pelinovsky \cite{CuccagnaPelinovsky2014}.  The analysis of the radiation is closely related to scattering for small solutions. Germain, Masmoudi and Shatah \cite{GMS2009,GMS2012} use the spacetime resonance method to study a class of quadratic Schr\"odinger equations in two and three space dimensions. By exploiting oscillations through integration by parts in time or frequency, they proved decay and scattering for small, localized solutions. See also \cite{GMS2012a}. For the one dimensional cubic nonlinear Schr\"odinger equation, Kato and Pusateri \cite{KP2011} gave a Fourier analytic proof of modified scattering, previously established by Hayashi and Naumkin \cite{HN1998}. Germain, Pusateri and Rousset \cite{GPR} used the distorted Fourier transform to prove asymptotic stability of the zero solution in $L^\infty_x$ and modified scattering for small, localized solutions in the presence of a generic potential. In the setting of solitary waves, Krieger and Schlag \cite{KriegerSchlag2006} constructed stable manifolds near unstable solitary waves for the one dimensional supercritical focusing NLS. Collot and Germain \cite{CG2025} combined modulation and renormalization with resonance analysis using the distorted Fourier transform. They proved asymptotic stability under spectral assumptions excluding threshold resonances, with modified scattering of the radiation when the nonlinearity is cubic at the origin.

The presence of a threshold resonance introduces an additional difficulty by slowing down the local decay of the radiation. L\"uhrmann and Schlag \cite{LS2023} exploited a factorization of the linearized operator to obtain sharp $L^\infty_x$ decay estimates and asymptotics in a long range scattering problem. For the cubic nonlinear Schr\"odinger equation with a potential, Chen and Pusateri \cite{ChenPusateri} weakened the assumptions on the potential and treated the resonant case under additional cancellation conditions at zero frequency. Their later work \cite{ChenPusateri2024} treated odd or even threshold resonance functions without imposing the earlier zero frequency cancellation conditions or any parity restriction on the initial data. For solitary waves of the one dimensional focusing cubic nonlinear Schr\"odinger equation, Li and L\"uhrmann \cite{LiLuhrmann2024} proved full asymptotic stability under small, localized even perturbations, and Yongming Li \cite{YongmingLi2025} subsequently removed the evenness assumption. Threshold resonances also act as a key obstruction in the stability of Klein--Gordon and sine-Gordon equations. For one dimensional Klein--Gordon equations with non-generic potentials and spatially localized quadratic nonlinearities, Lindblad, L\"uhrmann, Schlag and Soffer \cite{LLSS2023} showed that the interaction between the threshold resonance and the quadratic nonlinearity can produce a logarithmic slow down of the decay rate along certain rays. For the sine-Gordon kink, L\"uhrmann and Schlag \cite{LS2023} proved asymptotic stability under small, localized odd perturbations using the factorization of the linearized operator, a non-resonance property of the quadratic nonlinearity, and a variable coefficient quadratic normal form. Chen and L\"uhrmann \cite{ChenLuhrmann2024} subsequently proved full asymptotic stability of moving sine-Gordon kinks without any symmetry restriction, together with modified scattering of the radiation. Their proof combines distorted Fourier resonance analysis with modulation theory in a coupled bootstrap argument for the radiation and the modulation parameters, exploiting null structures in the quadratic terms of both equations.

For Schr\"odinger maps, the global dynamics for small data in critical spaces were clarified by Bejenaru, Ionescu, Kenig and Tataru  \cite{BIKT2011}. The stationary configurations of Schr\"odinger maps are harmonic maps. Gustafson, Kang and Tsai \cite{GKT2008} proved asymptotic stability of equivariant harmonic maps from $\mathbb R^2$ to $\mathbb S^2$ of degree at least four, and Gustafson, Nakanishi and Tsai \cite{GustafsonNakanishiTsai2010} extended this result to degree three, and recently Bejenaru, Pillai and Tataru \cite{BPT2024} solved the degree two case. On the hyperbolic plane, Lawrie, L\"uhrmann, Oh and Shahshahani \cite{LLS2023} and  Li \cite{Li2022} established asymptotic stability of harmonic maps under a suitable linearized stability condition and general perturbations. In one space dimension,  Li \cite{ZeLi2025} proved $t^{-1/2}$ decay toward constant maps into Riemannian surfaces under an intrinsic geometric vanishing condition at the reference point. The asymptotics are described by modified scattering when the curvature at that point is nonzero, and by ordinary scattering when it vanishes. The present work concerns the asymptotic stability of domain walls for the one dimensional Schr\"odinger map with Dzyaloshinskii--Moriya interaction.

\subsection{Organization of the paper}
The rest of the paper is organized as follows. The linear theory is developed in Section~\ref{2}, where the distorted Fourier transform is constructed and the linear estimates needed below are established. The perturbation equations are derived in Section~\ref{3}. The core of the proof consists of Section~\ref{georedu}, where we construct a parallel frame to reduce the Schr\"odinger map flow to a complex scalar equation and then establish weighted estimates for the corresponding profile, and Section~\ref{sectransf}, where these estimates are transferred back to the perturbation variable and combined with bounds for the modulation parameters to close the bootstrap. 
Section~\ref{ssca} is devoted to modified scattering. After removing the quadratic contribution by a normal form transformation, the resonant cubic interaction is isolated to identify the logarithmic phase correction, while the remaining contributions are shown to be integrable in time.

\subsection{Notations} The notation $A\lesssim B$ means there exists some $C>0$ such that $A\le CB$. Unless otherwise stated, the implicit constants throughout the paper may depend on the fixed parameters $\mu$ and $N$. The notation $A\ll B$ means that $A/B$ is sufficiently small.
In the following, we fix four small constants $\delta, \delta_*, \rho, \alpha$ to fulfill 
\begin{align}\label{Nota}
\delta,\rho\in(0,\frac{1}{2026}), \quad\delta_*\in(\frac{3}{2026},\frac{5}{2026}),\quad \alpha=\frac{505}{2026}\in(\frac15,\frac14).
\end{align}
Throughout the proof, the small constants are understood to satisfy the hierarchy
\begin{align}\label{Nota2}
0< \varepsilon_1\ll\min\{1,\varepsilon_*,r_*\},
\end{align}
where   $\varepsilon_1$ is an auxiliary bootstrap parameter, and $\varepsilon_*,r_*$ are the constants in Lemma~\ref{modula}.\\
We also use the weighted Sobolev space $$H^{0,1}(\mathbb R)=\left\{f\in L^2(\mathbb R):xf\in L^2(\mathbb R)\right\},$$
equipped with the norm $$\|f\|_{H^{0,1}}=\|f\|_{L^2}+\|xf\|_{L^2}.$$

\noindent{\bf Acknowledgments.} Z. Li was supported by the National Natural Science Foundation of China
under grants (No.1242211 and No.12371244). Z. Wang was supported by National Natural Science Foundation of China (No. 12271497).  L. Zhao was supported by National Natural Science Foundation of China (No. 12341102 and No. 12271497).

\section{Linear theory}\label{2}
Recall that $S_\eta(x)=\operatorname{sech}(\mu x)$, $T_\eta(x)=\tanh(\mu x)$, where $\mu=\sqrt{1-\eta^2}.$
We introduce the first order operators $D_\eta=\partial_x-\mu T_\eta$ and $D_\eta^*=-\partial_x-\mu T_\eta$ as densely defined operators on $L^2(\mathbb R)$ with domain $H^1(\mathbb R)$, and define \begin{align}\label{qa2}
\mathcal L_\eta=D_\eta D_\eta^*=-\partial_x^2+\mu^2\bigl(1-2S_\eta^2\bigr),\quad H_\eta=D_\eta^*D_\eta=-\partial_x^2+\mu^2  
\end{align}
as self-adjoint operators on $L^2(\mathbb R)$ with domain $H^2(\mathbb R)$. A direct computation gives $\mathcal L_\eta D_\eta=D_\eta H_\eta$, $D_\eta^*\mathcal L_\eta=H_\eta D_\eta^*$.

The equations $D_\eta^*f=0$ and $D_\eta f=0$ can be solved explicitly. In fact
$\ker D_\eta^*=\operatorname{span}\{S_\eta\}$, $\ker D_\eta=\{0\}$. 
Since $\mathcal L_\eta=D_\eta D_\eta^*$, we further have $\ker \mathcal L_\eta=\ker D_\eta^*=\operatorname{span}\{S_\eta\}$. 

Moreover, $\mathcal L_\eta$ is a rescaled P\"oschl--Teller Schr\"odinger operator, whose essential spectrum is $\sigma_{\mathrm e}(\mathcal L_\eta)=[\mu^2,\infty)$;  see, for instance, \cite[Lemma~4.14 and Remark~4.15]{CoteIgnat}.  Since $T_\eta(x)\to-1 \ \text{as }x\to-\infty, T_\eta(x)\to1 \ \text{as }x\to+\infty$, the corresponding reflection and transmission coefficients are $$r_\eta(\xi)=0, \quad t_\eta(\xi)=\frac{\xi+i\mu}{\xi-i\mu}.$$
For $f\in L^2(\mathbb R)$, define $$P_df=\frac{\mu}{2}S_\eta\int_{\mathbb R}S_\eta f{\rm d}x,\quad P_e=I-P_d.$$
Thus $P_d$ and $P_e$ are the orthogonal projections of $L^2(\mathbb R)$ onto $\operatorname{span}\{S_\eta\}$ and $S_\eta^\perp$, respectively.

\subsection{The distorted Fourier transform}

In this subsection, we construct the distorted Fourier transform and establish the basic properties that will be used in the subsequent analysis.
We use the flat Fourier transform
$$\widehat f(\xi)=(\mathcal Ff)(\xi)=\frac{1}{\sqrt{2\pi}}\int_{\mathbb R}e^{-ix\xi}f(x){\rm d}x.$$
We also write $\langle\xi\rangle_\mu=\sqrt{\xi^2+\mu^2}$. For every $\xi\in\mathbb R$, define
\begin{equation}
\label{gee}
\psi_\eta(x,\xi)=\frac{1}{\sqrt{2\pi}}\frac{i\xi-\mu T_\eta(x)}{\langle\xi\rangle_\mu}e^{ix\xi}.
\end{equation}
Then $\mathcal L_\eta\psi_\eta(x,\xi)=\langle\xi\rangle_\mu^2\psi_\eta(x,\xi).$ Accordingly, for $f,h\in\mathcal S(\mathbb R)$, we define the distorted Fourier transform associated with $\mathcal L_\eta$ and its inverse by
\begin{align}
(\widetilde{\mathcal F}f)(\xi)&=\int_{\mathbb R}\overline{\psi_\eta(x,\xi)}f(x){\rm d}x,\label{dF}\\
(\widetilde{\mathcal F}^{-1}h)(x)&=\int_{\mathbb R}\psi_\eta(x,\xi)h(\xi){\rm d}\xi.\label{indF}
\end{align}

In the following, we denote $\widetilde {\mathcal F} f$ by $\widetilde{f}$ for simplicity.

Direct computations give  the following four results, among which (\ref{freesch}) is well-known, see for instance \cite{HN1998}.

\begin{prop}\label{disfour}
Let $|\eta|<1$. For $\mathcal{L}_{\eta}$, $D^*_{\eta}$, $H_{\eta}$ and $\widetilde {\mathcal F}$ defined in (\ref{qa2}) and (\ref{dF}), 
there hold 
\begin{itemize}
\item For every $f\in L^1(\mathbb R)$, we have
\begin{align}\label{001}
\|\widetilde{\mathcal F}f\|_{L^\infty_\xi}\lesssim\|f\|_{L^1_x}.
\end{align}
\item   For \(f\in\mathcal S(\mathbb R)\), we have
$(\widetilde{\mathcal F}\mathcal L_\eta f)(\xi)=\langle\xi\rangle_\mu^2(\widetilde{\mathcal F}f)(\xi).$ 
\item 
For every $f\in H^1(\mathbb R)$, we have
\begin{align}\label{dftdbx}
\widetilde{f}(\xi)=\frac{\widehat{D_\eta^*f}(\xi)}{\langle\xi\rangle_\mu}.
\end{align}
\item Let \(f\in\mathcal S(\mathbb R)\) and \(t>0\). Then
\begin{align}\label{freesch}
 e^{-itH_\eta}f(x)=\frac{e^{-i\mu^2t-i\pi/4}}{\sqrt{2t}}e^{ix^2/(4t)}\mathcal F\left[e^{iy^2/(4t)}f(y)\right]\left(\frac{x}{2t}\right).
\end{align}
The identities extend to the corresponding $L^2$ based spaces by density.
\end{itemize} 
\end{prop}

Recall $D_{\eta}$ and $H_{\eta}$ defined in (\ref{qa2}).
\begin{lem}
\label{Detapro}
The operator $D_\eta H_\eta^{-1}D_\eta^*:H^1(\mathbb R)\to L^2(\mathbb R)$ admits a unique bounded extension to $L^2(\mathbb R)$, and this extension is the orthogonal projection onto $S_\eta^\perp$.
\end{lem}

\begin{proof}
Since $\widehat{H_\eta^{-1/2}f}(\xi)=\frac{\widehat f(\xi)}{\langle\xi\rangle_\mu}$, the operator $H_\eta^{-1/2}$ is an isomorphism from $L^2(\mathbb R)$ onto $H^1(\mathbb R)$. Set $A_\eta=D_\eta H_\eta^{-1/2}.$ For every $u\in H^1(\mathbb R)$, a direct integration by parts gives $$\|D_\eta u\|_{L^2}^2=\|u_x\|_{L^2}^2+\mu^2\|u\|_{L^2}^2.$$
Plancherel's theorem gives $\|A_\eta f\|_{L^2}=\|f\|_{L^2}, \ f\in L^2(\mathbb R).$ Thus $A_\eta$ is an isometry and $\operatorname{Ran}A_\eta$ is closed. Since $H_\eta^{-1/2}$ maps $L^2(\mathbb R)$ onto $H^1(\mathbb R)$,
$\operatorname{Ran}A_\eta=\operatorname{Ran}D_\eta.$ Together with $\overline{\operatorname{Ran}D_\eta}=(\ker D_\eta^*)^\perp=S_\eta^\perp,$ this yields $\operatorname{Ran}A_\eta=S_\eta^\perp.$ Therefore $A_\eta A_\eta^*=P_e.$
Moreover, on $H^1(\mathbb R)$, $A_\eta^*=H_\eta^{-1/2}D_\eta^*,$ and hence $A_\eta A_\eta^*=D_\eta H_\eta^{-1}D_\eta^*.$ Thus $D_\eta H_\eta^{-1}D_\eta^*$ extends uniquely to $P_e$ on $L^2(\mathbb R)$.
\end{proof}

\begin{rem}
Lemma~\ref{Detapro} and $\ker D_\eta=\{0\}$ yield $D_\eta:H^1(\mathbb R)\to S_\eta^\perp$ is bijective. We write $D_\eta^{-1}:S_\eta^\perp\to H^1(\mathbb R)$ for its inverse. 
Therefore $$D_\eta^{-1}P_e=H_\eta^{-1}D_\eta^*
\quad\text{on }H^1(\mathbb R),$$
and on $L^2(\mathbb R)$ this identity is understood by bounded continuous extension.
\end{rem}

Combining Lemma~\ref{Detapro} with \eqref{dftdbx}, we immediately obtain the following
representation of the distorted Fourier transform in terms of $D_\eta^{-1}$.

\begin{cor}
\label{cdfi}
For every $f\in H^1(\mathbb R)$, $(\widetilde{\mathcal F}f)(\xi) = \langle\xi\rangle_\mu\mathcal{F}(D_\eta^{-1}P_ef)(\xi).$
\end{cor}

\begin{prop}
The distorted Fourier transform satisfies $\widetilde{\mathcal F}\widetilde{\mathcal F}^{-1} = I \ \text{on }L^2(\mathbb R), \ \widetilde{\mathcal F}^{-1}\widetilde{\mathcal F} = P_e \ \text{on }L^2(\mathbb R).$ Moreover, $\widetilde{\mathcal F}$ is unitary from $S_\eta^\perp$ onto $L^2(\mathbb R)$.
\end{prop}

\begin{proof}
For $f\in H^1(\mathbb R)$, \eqref{dftdbx} gives $\widetilde{f} = \mathcal{F}(H_\eta^{-1/2}D_\eta^*f) = \mathcal F A_\eta^*f.$ A density argument shows that this identity extends to all $f\in L^2(\mathbb R)$, and hence $\widetilde{\mathcal F} = \mathcal F A_\eta^*.$ On the other hand, for $h\in\mathcal S(\mathbb R)$, using
\eqref{indF} and $A_\eta=D_\eta H_\eta^{-1/2}$, we obtain
\begin{align*}
(A_\eta\mathcal F^{-1}h)(x)
&=D_\eta H_\eta^{-1/2}\mathcal F^{-1}h(x)\\
&=\frac{1}{\sqrt{2\pi}}\int_{\mathbb R}\frac{i\xi-\mu T_\eta(x)}{\langle\xi\rangle_\mu}e^{ix\xi}h(\xi)\,\mathrm d\xi\\
&= \int_{\mathbb R}\psi_\eta(x,\xi)h(\xi)\,\mathrm d\xi\\
&= (\widetilde{\mathcal F}^{-1}h)(x).
\end{align*}
Therefore, by density, $\widetilde{\mathcal F}^{-1}=A_\eta\mathcal F^{-1}.$ Since $A_\eta^*A_\eta=I,\quad A_\eta A_\eta^*=P_e,$ and $\mathcal F$ is unitary on $L^2(\mathbb R)$, we obtain
$$\widetilde{\mathcal F}\widetilde{\mathcal F}^{-1}=\mathcal F A_\eta^*A_\eta\mathcal F^{-1}=I,\quad
\widetilde{\mathcal F}^{-1}\widetilde{\mathcal F}=A_\eta\mathcal F^{-1}\mathcal F A_\eta^*=A_\eta A_\eta^*=P_e.$$
Finally, $A_\eta^*$ is unitary from $S_\eta^\perp$ onto $L^2(\mathbb R)$, and $\mathcal F$ is unitary on $L^2(\mathbb R)$. Therefore $\widetilde{\mathcal F}=\mathcal F A_\eta^*$ is unitary from $S_\eta^\perp$ onto $L^2(\mathbb R)$.
\end{proof}

\subsection{Linear estimates}
In this subsection, we collect the linear estimates that will be used repeatedly in the nonlinear analysis below. 
The following stationary phase estimate follows from the argument in \cite[Lemma~2.2]{HN1998}. We include a proof for readers' convenience. 
\begin{lem}[\cite{HN1998}]
\label{Stapha}
 Let $f\in H^{0,1}(\mathbb R)$.
 \begin{itemize}
\item For $t\geq1$, there holds 
$$e^{it\partial_x^2}f(x)=\frac{e^{ix^2/(4t)-i\pi/4}}{\sqrt{2t}}\widehat f\left(\frac{x}{2t}\right)+R(t,x),\qquad
\|R(t)\|_{L^\infty_x}\lesssim t^{-3/4}\|xf\|_{L^2_x}.$$
\item For $t\geq1$, we have
\begin{equation}
\label{staphase}
\|e^{it\partial_x^2}f\|_{L^\infty_x}\lesssim t^{-1/2}\|\widehat f\|_{L^\infty_\xi}+t^{-\frac{3}{4}}\|xf\|_{L^2_x}.
\end{equation}
\end{itemize}
\end{lem}

\begin{proof}
A density argument reduces the proof to $f\in\mathcal S(\mathbb R)$. (\ref{freesch}) gives the following representation for the remainder:
$$R(t,x)=\frac{e^{ix^2/(4t)-i\pi/4}}{\sqrt{2t}}\mathcal F\left[\bigl(e^{iy^2/(4t)}-1\bigr)f(y)\right]\left(\frac{x}{2t}\right).$$
To estimate the remainder, we first note that
\begin{align*}
\int_{\mathbb R}\frac{\left|e^{iy^2/(4t)}-1\right|^2}{y^2}{\rm d}y
&=t^{-1/2}\int_{\mathbb R}\frac{\left|e^{iz^2/4}-1\right|^2}{z^2}{\rm d}z\\
&\lesssim t^{-1/2}\left(\int_{|z|\leq1}z^2\,{\rm d}z+\int_{|z|\geq1}\frac{1}{z^2}{\rm d}z\right)\\
&\lesssim t^{-1/2},
\end{align*}
where we used $|e^{i\theta}-1|\lesssim\min\{|\theta|,1\}.$ Therefore, by the $L^1_x\to L^\infty_\xi$ bound for the Fourier transform and the Cauchy--Schwarz inequality, we have 
\begin{align*}
\|R(t)\|_{L^\infty_x}
&\lesssim t^{-1/2}\left\|\mathcal F\left[\bigl(e^{iy^2/(4t)}-1\bigr)f(y)\right]\right\|_{L^\infty_\xi}\\
&\lesssim t^{-1/2}\int_{\mathbb R}\left|e^{iy^2/(4t)}-1\right||f(y)|{\rm d}y\\
&\leq t^{-1/2}\left(\int_{\mathbb R}\frac{\left|e^{iy^2/(4t)}-1\right|^2}{y^2}{\rm d}y\right)^{1/2}\|yf\|_{L^2_y}\\
&\lesssim t^{-3/4}\|xf\|_{L^2_x},
\end{align*}
from which the rest (\ref{staphase})   follows directly. 
\end{proof}

\begin{cor}
\label{restaph}
For $f\in H^{0,1}(\mathbb R)$ and $t\geq1$, we have
\begin{equation}
\label{reapp}
\|\widehat f\|_{L^\infty_\xi}\lesssim t^{1/2}\|e^{it\partial_x^2}f\|_{L^\infty_x}+t^{-1/4}\|xf\|_{L^2_x}.
\end{equation}
\end{cor}

\begin{proof}
Evaluating the expansion in Lemma~\ref{Stapha} at $x=2t\xi$ gives
$$\frac{1}{\sqrt{2t}}|\widehat f(\xi)|\lesssim \|e^{it\partial_x^2}f\|_{L^\infty_x}+t^{-3/4}\|xf\|_{L^2_x}.$$
Multiplying by $t^{1/2}$ and taking the supremum over $\xi$ gives \eqref{reapp}.
\end{proof}

\begin{lem}\label{Deliest}
Let $\eta\in (-1,1)$ and $m\geq0$. If $f\in H^m(\mathbb R)$, then 
\begin{equation}
\label{HetaHm}
\|D_\eta H_\eta^{-1}f\|_{H^{m+1}}\lesssim \|f\|_{H^m}.
\end{equation}
The corresponding estimate also holds in $W^{m,\infty}:$ if $f\in W^{m,\infty}$, then
\begin{equation}
\label{DeHe}
\|D_\eta H_\eta^{-1}f\|_{W^{m+1,\infty}}\lesssim\|f\|_{W^{m,\infty}}.
\end{equation}
\end{lem}

\begin{proof}
The Fourier multiplier formula gives $\mathcal {F}(\partial_xH_\eta^{-1}f)(\xi)=\frac{i\xi}{\xi^2+\mu^2}\widehat{f}(\xi).$ Therefore,
$$\|\partial_xH_\eta^{-1}f\|_{H^{m+1}}^2=\int_{\mathbb R}\langle\xi\rangle^{2m+2}\frac{\xi^2}{(\xi^2+\mu^2)^2}|\widehat f(\xi)|^2\,d\xi=\int_{\mathbb R}\left(\frac{\langle\xi\rangle|\xi|}{\xi^2+\mu^2}\right)^2\langle\xi\rangle^{2m}|\widehat f(\xi)|^2\,d\xi.$$
Thus we obtain
$$\|\partial_xH_\eta^{-1}f\|_{H^{m+1}}\le \sup_{\xi\in\mathbb R}\frac{\langle\xi\rangle|\xi|}{\xi^2+\mu^2}\,\|f\|_{H^m}\lesssim \|f\|_{H^m}.$$
Similarly, the Leibniz rule gives
$$\| T_\eta H_\eta^{-1}f\|_{H^{m+1}}\lesssim \|H_\eta^{-1}f\|_{H^{m+1}}\\
\lesssim \|f\|_{H^m}
$$
and \eqref{HetaHm} follows.

For the $W^{m,\infty}$ estimate, set $u=H_\eta^{-1}f$. Since $(-\partial_x^2+\mu^2)\bigl(e^{-\mu|\cdot|}/(2\mu)\bigr)=\delta_0$ in the distributional sense, we have
\begin{equation}\label{Green}
u=\frac{1}{2\mu}e^{-\mu|\cdot|}*f .
\end{equation}
It follows that $\|u\|_{L^\infty}\leq\mu^{-2}\|f\|_{L^\infty},$ $\|u_x\|_{L^\infty}\leq\mu^{-1}\|f\|_{L^\infty}.$

For $r\geq2$, the equation $-u_{xx}+\mu^2u=f$ gives $\partial_x^ru=\mu^2\partial_x^{r-2}u-\partial_x^{r-2}f$. Induction on $r$ therefore yields $\|\partial_x^ru\|_{L^\infty}\lesssim\|f\|_{W^{m,\infty}}.$ For $0\leq j\leq m+1$, the Leibniz rule gives
$$\partial_x^j(D_\eta u)=\partial_x^{j+1}u-\mu\sum_{a=0}^jC_j^a(\partial_x^aT_\eta)\partial_x^{j-a}u.$$
Combining these estimates yields \eqref{DeHe}.
\end{proof}

\begin{lem}\label{weiDeta}
Let $\eta\in (-1, 1)$. Assume that $S_\eta^{1/2}f\in L^\infty(\mathbb R)$. Then
$$\|S_\eta D_\eta H_\eta^{-1}f\|_{L^2}+\|S_\eta\partial_xD_\eta H_\eta^{-1}f\|_{L^2}\lesssim\|S_\eta^{1/2}f\|_{L^\infty}.$$
\end{lem}

\begin{proof}
The Green function representation \eqref{Green} gives
\begin{equation*}
H_\eta^{-1}f(x)=\frac1{2\mu}\int_{\mathbb R}e^{-\mu|x-s|}f(s)\,\mathrm ds .
\end{equation*}
Hence, 
\begin{equation*}
\partial_xH_\eta^{-1}f(x)=-\frac12\int_{\mathbb R}\operatorname{sgn}(x-s)e^{-\mu|x-s|}f(s)\,\mathrm ds .
\end{equation*}
and therefore
\begin{equation}\label{ewrp}
\begin{aligned}
S_\eta(x)|H_\eta^{-1}f(x)|
&\lesssim \|S_\eta^{1/2}f\|_{L^\infty}S_\eta^{1/2}(x)\int_{\mathbb R}e^{-\frac\mu2|x-s|}\,\mathrm ds
\lesssim S_\eta^{1/2}(x)\|S_\eta^{1/2}f\|_{L^\infty},\\
S_\eta(x)|\partial_xH_\eta^{-1}f(x)|
&\lesssim\|S_\eta^{1/2}f\|_{L^\infty}S_\eta^{1/2}(x)\int_{\mathbb R}e^{-\frac\mu2|x-s|}\,\mathrm ds
\lesssim S_\eta^{1/2}(x)\|S_\eta^{1/2}f\|_{L^\infty}.
\end{aligned}
\end{equation}
Hence
\begin{align*}
\|S_\eta D_\eta H_\eta^{-1}f\|_{L^2}\lesssim\|S_\eta^{1/2}f\|_{L^\infty}.
\end{align*}
Differentiating $D_\eta H_\eta^{-1}f$ and using \eqref{Tederi}, we obtain
\begin{align*}
\partial_xD_\eta H_\eta^{-1}f
&=\partial_x^2H_\eta^{-1}f-\mu(\partial_xT_\eta)H_\eta^{-1}f-\mu T_\eta\partial_xH_\eta^{-1}f\\
&=-f+\mu^2H_\eta^{-1}f-\mu^2S_\eta^2H_\eta^{-1}f-\mu T_\eta\partial_xH_\eta^{-1}f.
\end{align*}
Thus, by \eqref{ewrp}, \eqref{STide},
\begin{align*}
S_\eta|\partial_xD_\eta H_\eta^{-1}f|
&\leq S_\eta|f|+\mu^2S_\eta|H_\eta^{-1}f|+\mu^2S_\eta^3|H_\eta^{-1}f|+\mu S_\eta|\partial_xH_\eta^{-1}f|\\
&\lesssim S_\eta^{1/2}\|S_\eta^{1/2}f\|_{L^\infty}.
\end{align*}
Therefore,
\begin{align*}
\|S_\eta\partial_xD_\eta H_\eta^{-1}f\|_{L^2}
&\lesssim \|S_\eta^{1/2}\|_{L^2}\|S_\eta^{1/2}f\|_{L^\infty}
\lesssim \|S_\eta^{1/2}f\|_{L^\infty}.
\end{align*}
Then the lemma follows by combining the two estimates.
\end{proof}

\section{Writing the equations}\label{3}
\subsection{Perturbation equations}

To eliminate the Dzyaloshinskii--Moriya term, we set $\mathbf M(t,x)=R_{\eta x}m(t,x)$, Under this transformation, \eqref{mLLM} becomes
\begin{equation}\label{LLM}
\partial_t\mathbf M=\mathbf M\wedge\left(\partial_x^2\mathbf M+(1-\eta^2)\mathbf M_1e_1\right).
\end{equation}
  The corresponding reference domain wall is
\begin{align}\label{ab}
  \Omega_*(x)=\left(T_\eta(x),S_\eta(x),0\right).
\end{align}  
To describe its dynamics near of $\Omega_*$ modulo translations and rotations, we write
\begin{equation}\label{Mwallde}
\mathbf M(t,x) = R_{\phi(t)+\eta y(t)} \Bigl(\Omega_*(x-y(t))+\varepsilon(t,x-y(t))\Bigr).
\end{equation}
Here the perturbation $\varepsilon$ is decomposed as
\begin{equation}\label{adapdecom}
\varepsilon=\alpha\Omega_*+\beta N_*+\gamma P_*.
\end{equation}
The triple $(\Omega_*,N_*,P_*)$ is the orthonormal frame adapted to the
reference domain wall, where
\begin{equation}\label{fram}
N_*=(-S_\eta,T_\eta,0),\qquad P_*=e_3.
\end{equation}
This frame satisfies the following identities:
\begin{equation}\label{fraorth}
|\Omega_*|=|N_*|=|P_*|=1,\quad\Omega_*\cdot N_*=\Omega_*\cdot P_*=N_*\cdot P_*=0.
\end{equation}
\begin{equation}\label{crosspro}
\Omega_*\wedge N_*=P_*,\quad N_*\wedge P_*=\Omega_*,\quad P_*\wedge\Omega_*=N_*,
\end{equation}
and
\begin{equation}
\label{rotafra}
\partial_x\Omega_*=-\mu S_\eta N_*,\quad
\partial_xN_*=\mu S_\eta\Omega_*,\quad
\partial_xP_*=0,\quad
e_1=T_\eta\Omega_*-S_\eta N_*.
\end{equation}
In order to remove the motion of $\mathbf M$ generated by translations and rotations,
we define the modulated magnetization $M$ by
\begin{equation}
\label{phidef}
M(t,x)=R_{-\phi(t)-\eta y(t)}\mathbf M(t,x+y(t))=\Omega_*(x)+\varepsilon(t,x).
\end{equation}
Since $\mathbf M$ satisfies \eqref{LLM}, differentiating
\eqref{phidef} gives
\begin{equation}
\label{phileft}
-\dot y\,\partial_xM+(\dot\phi+\eta\dot y)e_1\wedge M+\partial_tM
=M\wedge H(M),
\end{equation}
where $H(M)=\partial_x^2M+\mu^2(M)_1e_1$.

We now derive the system for the perturbation components $\alpha$, $\beta$, and $\gamma$, which satisfy 
\begin{equation}
\label{constr}
(1+\alpha)^2+\beta^2+\gamma^2=1.
\end{equation}
For convenience in the following computations, we define
\begin{equation}
\label{v12}
V_1=\mu^2(T_\eta^2-S_\eta^2)=\mu^2(1-2S_\eta^2),\qquad
V_2=2\mu^2S_\eta T_\eta,\qquad
V_3=2\mu S_\eta.
\end{equation}
By \eqref{rotafra}, a direct computation gives
\begin{equation*}
H(M)=\Bigl[\partial_x^2\alpha+V_1(1+\alpha)-V_2\beta+V_3\partial_x\beta\Bigr]\Omega_*
+\Bigl[\partial_x^2\beta-V_3\partial_x\alpha\Bigr]N_*
+(\partial_x^2\gamma)P_*.
\end{equation*}
The cross product identities \eqref{crosspro} then yield
\begin{equation}
\label{llphi}
\begin{aligned}
M\wedge H(M)
&=\Bigl[\beta\partial_x^2\gamma-\gamma\partial_x^2\beta+V_3\gamma\partial_x\alpha\Bigr]\Omega_*\\
&\quad+\Bigl[\mathcal L_\eta\gamma+\gamma\partial_x^2\alpha-\alpha\partial_x^2\gamma
+V_1\alpha\gamma-V_2\beta\gamma+V_3\gamma\partial_x\beta\Bigr]N_*\\
&\quad+\Bigl[-\mathcal L_\eta\beta+\alpha\partial_x^2\beta-\beta\partial_x^2\alpha
-V_1\alpha\beta+V_2\beta^2
-V_3\bigl((1+\alpha)\partial_x\alpha+\beta\partial_x\beta\bigr)\Bigr]P_*.
\end{aligned}
\end{equation}
To compute the modulation terms in \eqref{phileft},
\eqref{rotafra} gives
\begin{equation*}
\partial_xM
=\left(\partial_x\alpha+\mu S_\eta\beta\right)\Omega_*
+\left(\partial_x\beta-\mu S_\eta(1+\alpha)\right)N_*
+(\partial_x\gamma)P_*
\end{equation*}
and
\begin{equation*}
e_1\wedge M
=-S_\eta\gamma\Omega_*
-T_\eta\gamma N_*
+\bigl[S_\eta(1+\alpha)+T_\eta\beta\bigr]P_*.
\end{equation*}
Hence the left hand side of \eqref{phileft} is
\begin{equation}
\label{phicomp}
\begin{aligned}
&\Bigl\{-\dot y\bigl(\partial_x\alpha+\mu S_\eta\beta\bigr)
-(\dot\phi+\eta\dot y)S_\eta\gamma+\partial_t\alpha\Bigr\}\Omega_*\\
&+\Bigl\{-\dot y\bigl[\partial_x\beta-\mu S_\eta(1+\alpha)\bigr]
-(\dot\phi+\eta\dot y)T_\eta\gamma+\partial_t\beta\Bigr\}N_*\\
&+\Bigl\{-\dot y\,\partial_x\gamma
+(\dot\phi+\eta\dot y)\bigl[S_\eta(1+\alpha)+T_\eta\beta\bigr]
+\partial_t\gamma\Bigr\}P_*.
\end{aligned}
\end{equation}
Comparing \eqref{llphi} and \eqref{phicomp}, we obtain the following system for $\alpha$, $\beta$, and $\gamma$:
\begin{equation}
\label{systemcoe}
\left\{
\begin{aligned}
\partial_t\alpha
&=\beta\partial_x^2\gamma-\gamma\partial_x^2\beta
+V_3\gamma\partial_x\alpha+\operatorname{mod}_1,\\
\partial_t\beta
&=\mathcal L_\eta\gamma+\gamma\partial_x^2\alpha-\alpha\partial_x^2\gamma
+V_1\alpha\gamma-V_2\beta\gamma
+V_3(\partial_x\beta)\gamma+\operatorname{mod}_2,\\
\partial_t\gamma
&=-\mathcal L_\eta\beta+\alpha\partial_x^2\beta-\beta\partial_x^2\alpha
-V_1\alpha\beta+V_2\beta^2
-V_3\bigl[\beta\partial_x\beta+(1+\alpha)\partial_x\alpha\bigr]
+\operatorname{mod}_3.
\end{aligned}
\right.
\end{equation}
where
\begin{equation*}
\begin{aligned}
\operatorname{mod}_1&=\dot y\bigl(\partial_x\alpha+\mu S_\eta\beta\bigr)+(\dot\phi+\eta\dot y)S_\eta\gamma,\\
\operatorname{mod}_2&=\dot y\bigl[\partial_x\beta-\mu S_\eta(1+\alpha)\bigr]+(\dot\phi+\eta\dot y)T_\eta\gamma,\\
\operatorname{mod}_3&=\dot y\,\partial_x\gamma-(\dot\phi+\eta\dot y)\bigl[S_\eta(1+\alpha)+T_\eta\beta\bigr].
\end{aligned}
\end{equation*}
Let $u=\beta+i\gamma$. Combining the second and third equations in
\eqref{systemcoe}, we obtain
\begin{equation}
\label{NLSfor}
i\partial_tu-\mathcal L_\eta u=N+Mod,
\end{equation}
where
\begin{align}
\label{Nfina}
N&=u\partial_x^2\alpha-\alpha\partial_x^2u+V_1\alpha u-\frac12V_2u(u+\overline u)+\frac12V_3(u-\overline u)\partial_xu,\\
Mod&=i\dot y\Bigl[\partial_xu-\mu S_\eta(1+\alpha)\Bigr]+(\dot\phi+\eta\dot y)\Bigl[T_\eta u+S_\eta(1+\alpha)\Bigr].\notag
\end{align}

The perturbation equation\eqref{NLSfor} presents two main difficulties. First, the linearized operator $\mathcal L_\eta$ has a threshold resonance at the bottom of its continuous spectrum. Indeed, 
\begin{equation*}
\mathcal L_\eta T_\eta=\mu^2T_\eta,\quad T_\eta\in L^\infty(\mathbb R)\setminus L^2(\mathbb R).
\end{equation*}
The slow local decay caused by this threshold resonance poses a significant challenge to the analysis of nonlinear interactions. Second, the nonlinearity in \eqref{Nfina} contains $u\partial_x^2\alpha-\alpha\partial_x^2u$. Direct estimates of these terms lead to a loss of derivatives, making it difficult to close the bootstrap argument for the $u$ equation.
To reduce the loss of derivatives while preserving the decay and scattering properties of $u$, we introduce a nonlinear change of variables that agrees with the identity up to higher order terms. Specifically, let
\begin{equation}
\label{u-z}
u=\frac{z}{1+\frac14|z|^2}.
\end{equation}
Using \eqref{constr}, one has 
$$(1+\alpha)^2
=1-|u|^2
=1-\frac{|z|^2}{\left(1+\frac14|z|^2\right)^2}
=\left(\frac{1-\frac14|z|^2}{1+\frac14|z|^2}\right)^2.$$
Hence
\begin{equation}
\label{z-def}
z=\frac{2u}{2+\alpha}.
\end{equation}
From \eqref{Mwallde},
\eqref{adapdecom} and \eqref{phidef}, $z$ can also be rewritten as
\begin{equation}
\label{zmphi}
z=\frac{2\bigl(M\cdot N_*+iM\cdot P_*\bigr)}{1+M\cdot\Omega_*}.
\end{equation}
We next use the constraint \eqref{constr} together with
the equation \eqref{NLSfor} to derive the equation for $z$. Differentiating \eqref{z-def} and using the definition of $\mathcal L_\eta$, we obtain
\begin{equation}
\label{yphidiff}
i\partial_tz-\mathcal L_\eta z
=\frac{2}{2+\alpha}\bigl(i\partial_tu-\mathcal L_\eta u\bigr)
-\frac{2iu}{(2+\alpha)^2}\partial_t\alpha
-\frac{4\partial_x\alpha}{(2+\alpha)^2}\partial_xu
-\frac{2u}{(2+\alpha)^2}\partial_x^2\alpha
+\frac{4u}{(2+\alpha)^3}(\partial_x\alpha)^2.
\end{equation}
Substituting \eqref{systemcoe},
\eqref{NLSfor}, \eqref{u-z} and
\eqref{z-def} into \eqref{yphidiff}, we obtain
\begin{equation}
\label{zevol}
i\partial_t z-\mathcal L_\eta z
=\mathcal N_2+\mathcal N_3+\mathcal N_{\ge 4}+\mathcal Mod.
\end{equation}
where
\begin{align}
\mathcal Mod&=i\dot y\left[\partial_xz-\mu S_\eta\left(1+\frac14z^2\right)\right]
+(\dot\phi+\eta\dot y)\left[T_\eta z+S_\eta\left(1-\frac14z^2\right)\right],\nonumber\\
\mathcal N_2&=-\frac12V_2z(z+\overline{z})+\frac12V_3(z-\overline{z})\partial_xz,\nonumber\\
\mathcal N_3&=\frac12\overline{z}(\partial_xz)^2-\frac12V_1|z|^2z,\nonumber\\
\mathcal N_{\ge 4}&=-\frac{|z|^2}{4+|z|^2}
\left[\mathcal N_2(z,\overline{z})+\mathcal N_3(z,z,\overline{z})\right].\label{nonlinear}
\end{align}

\begin{rem}\label{nondeg}
We assume that
\begin{equation}\label{stereo}
1+M\cdot\Omega_*=2+\alpha\in\left[\frac32,2\right].
\end{equation}
This nondegeneracy condition will be guaranteed later by \eqref{dymodsma}, see
Remark~\ref{conAnon}.
\end{rem}

\subsection{Modulation parameters}

We shall determine the modulation parameters $y$ and $\phi$ by imposing the orthogonality condition
\begin{equation}\label{zorthog}
\int_{\mathbb R} z(t,x)S_\eta(x)\,\mathrm dx=0.
\end{equation}

Let $\mathbf V:\mathbb R\to\mathbb S^2$ be a vector-valued function sufficiently close to $\Omega_*$ in $H^2$. For $(y,\phi)\in\mathbb{R}^2$ close to $(0,0)$, set 
\begin{align}\label{hhm}
V(x)=R_{-\phi-\eta y}\mathbf V(x+y).
\end{align}
We define 
\begin{equation}
\label{phiV}
z_{y,\phi}[\mathbf V](x)=\frac{2\bigl(V(x)\cdot N_*(x)+iV(x)\cdot P_*(x)\bigr)}{1+V(x)\cdot\Omega_*(x)}.
\end{equation}

\begin{lem}\label{modula}
There exist $\varepsilon_*>0$ and $r_*>0$ such that, for every $\mathbf V$ satisfying $\|\mathbf V-\Omega_*\|_{H^2}<\varepsilon_*$, there exists a unique pair parameters $(y,\phi)\in(-r_*,r_*)^2$ with 
\begin{equation}\label{parasize}
|y|+|\phi|\lesssim\|\mathbf V-\Omega_*\|_{H^2}.
\end{equation}
such that 
\begin{equation*}
\int_{\mathbb R}z_{y,\phi}[\mathbf V](x)S_\eta(x){\rm d}x=0.
\end{equation*}
Moreover, the parameters $y$ and $\phi$ are $C^1$ functions of $\mathbf V$ in an $H^2$ neighborhood of $\Omega_*$.
\end{lem}

\begin{proof}
Consider the map
$$\mathcal A:(\mathbf V,y,\phi)\longmapsto\begin{pmatrix}
\displaystyle \operatorname{Re}\int_{\mathbb R}z_{y,\phi}[\mathbf V](x)S_\eta(x){\rm d}x\\[1mm]
\displaystyle \operatorname{Im}\int_{\mathbb R}z_{y,\phi}[\mathbf V](x)S_\eta(x){\rm d}x
\end{pmatrix}.$$
For $\mathbf V$ in an $H^2$-neighborhood of $\Omega_*$ and $(y,\phi)$ sufficiently close to $(0,0)$, this map is $C^1$ from $(\Omega_*+H^2)\times\mathbb R^2$ into $\mathbb R^2$. Moreover, $\mathcal A(\Omega_*,0,0)=0$. By direct computation,
$$\left.\frac{\partial\mathcal A}{\partial(y,\phi)}\right|_{(\Omega_*,0,0)}=
\begin{pmatrix}
-2 & 0\\[1mm]
-\dfrac{2\eta}{\mu} & -\dfrac{2}{\mu}
\end{pmatrix}.$$
Hence, by the implicit function theorem, for $\mathbf V$ sufficiently close to $\Omega_*$ in $H^2$, there exists a unique pair $(y,\phi)$ close to $(0,0)$ such that $\mathcal A(\mathbf V,y,\phi)=0$. Moreover, $y$ and $\phi$ are $C^1$ functions of $\mathbf V$ with
$$|y|+|\phi|\lesssim\|\mathbf V-\Omega_*\|_{H^2}.$$
\end{proof}
Let $I$ be a time interval such that
\begin{equation}\label{dymodsma}
\|\mathbf M(t)-\Omega_*\|_{H^2}<\varepsilon_*, \quad \forall t\in I.
\end{equation}
For each $t\in I$, Lemma~\ref{modula} applied to $\mathbf V=\mathbf M(t)$ determines a unique pair $(y(t),\phi(t))\in(-r_*,r_*)^2$ such that the corresponding modulated perturbation satisfies \eqref{zorthog}. We now derive the modulation equations on the interval $I$. To ensure that the modulation equations are well defined, we require
\begin{equation}\label{moAsmall}
|A(t)|<\frac{1}{\mu},\quad t\in I,
\end{equation}
where
\begin{equation}\label{Adef}
A(t) = \int_{\mathbb R}S_\eta T_\eta z\,\mathrm dx-\frac14\int_{\mathbb R}S_\eta^2z^2\,\mathrm dx.
\end{equation}
In fact, the following lemma shows that, for $\varepsilon_*>0$ sufficiently small, \eqref{dymodsma} implies \eqref{moAsmall}.
\begin{lem}\label{dysmallA}
Fix $\eta\in (-1,1)$. If $\varepsilon_*>0$ sufficiently small, then \eqref{dymodsma} implies \eqref{moAsmall}.
\end{lem}
\begin{proof}
Fix $t\in I$ and set $\mathbf V=\mathbf M(t)$. Let $V$ be the corresponding modulated map defined in (\ref{hhm}). By the definition of $V$, the fundamental theorem of calculus, and
\begin{equation}
\label{wallrota}
\bigl\|R_{-\phi-\eta y}\Omega_*(\cdot+y)-\Omega_*(\cdot+y)\bigr\|_{H^2}{\le } 2\left|\sin\frac{\phi+\eta y}{2}\right|\|S_\eta\|_{H^2},
\end{equation}
together with \eqref{parasize} and \eqref{dymodsma}, we have
\begin{equation}
\label{H2small}
\begin{aligned}
\|V-\Omega_*\|_{H^2}\leq&\|\mathbf V-\Omega_*\|_{H^2}+\|R_{-\phi-\eta y}\Omega_*(\cdot+y)-\Omega_*\|_{H^2}\\
\leq&\|\mathbf V-\Omega_*\|_{H^2}+|y|\|\partial_x\Omega_*\|_{H^2}+2\left|\sin\frac{\phi+\eta y}{2}\right|\|S_\eta\|_{H^2}\\
\lesssim&\|\mathbf V-\Omega_*\|_{H^2}\lesssim\varepsilon_*.
\end{aligned}
\end{equation}
The bound\eqref{H2small} and the one dimensional Sobolev embedding give, for $\varepsilon_*$ sufficiently small,
\begin{equation}\label{mapLinf}
\|V-\Omega_*\|_{L^\infty}\leq\frac12.
\end{equation}
Combining \eqref{fraorth} and \eqref{mapLinf}, we obtain
\begin{equation}\label{moddelo}
1+V(x)\cdot\Omega_*(x)=2+(V(x)-\Omega_*(x))\cdot\Omega_*(x)\geq\frac32.
\end{equation}
Moreover, by \eqref{fraorth},
\begin{equation}\label{modnume}
\bigl|V\cdot N_*+iV\cdot P_*\bigr|\leq|V-\Omega_*|.
\end{equation}
Therefore, by \eqref{phiV}, \eqref{moddelo}, and \eqref{modnume}, one has 
\begin{equation}\label{poimod}
|z(x)|\lesssim |V(x)-\Omega_*(x)|.
\end{equation}
Combining \eqref{poimod} with \eqref{H2small}, one obtains
\begin{equation}\label{L2dysmall}
\|z\|_{L^2}\lesssim \|V-\Omega_*\|_{L^2}\lesssim\varepsilon_*.
\end{equation}
Applying \eqref{Adef}, \eqref{L2dysmall} and using the Cauchy--Schwarz inequality, we have
\begin{equation}\label{Acontr}
\begin{aligned}
|A|\leq&\left|\int_{\mathbb R}S_\eta T_\eta z{\rm d}x\right|+\frac14\left|\int_{\mathbb R}S_\eta^2z^2{\rm d}x\right|\\
\lesssim&\|z\|_{L^2}+\|z\|_{L^2}^2<\frac1\mu.
\end{aligned}
\end{equation}
Thus \eqref{moAsmall} holds on $I$.
\end{proof}

\begin{rem}
\label{conAnon}
The proof of Lemma~\ref{dysmallA} also shows
that \eqref{dymodsma} guarantees \eqref{stereo}. Thus the assumption in Remark~\ref{nondeg} is satisfied on $I$.
\end{rem}
Therefore, differentiating \eqref{zorthog} with respect to $t$ on $I$, we obtain
\begin{equation}\label{modueq}
\begin{cases}
\displaystyle \dot y=\frac{\frac2\mu\operatorname{Im}R+\operatorname{Im}(\overline A R)}{\frac4\mu-\mu|A|^2},\\[6pt]
\displaystyle \dot\phi+\eta\dot y=\frac{-2\operatorname{Re}R+\mu\operatorname{Re}(A\overline R)}{\frac4\mu-\mu|A|^2},
\end{cases}
\end{equation}
where
\begin{equation*}
R(t)=\int_{\mathbb R}\frac{S_\eta\left[\mathcal N_2(z,\overline z)+\mathcal N_3(z,z,\overline z)\right]}{1+\frac14|z|^2}{\rm d}x.
\end{equation*}

\section{Geometric reduction}
\label{georedu}
In this section, we construct a parallel frame and reduce \eqref{LLM} to a complex scalar equation, thereby overcoming the loss of derivatives in \eqref{zevol}. The corresponding modulated equation is then derived, and its unknown is expressed explicitly in terms of $z$.

\subsection{Geometric reduction to a complex scalar equation}\label{subsec:three-basic-principles}
The variables $u$ and $z$ introduced above provide natural coordinates for describing the perturbation of $m$ relative to the reference domain wall $\omega_*$. In order to fully exploit the information of the static domain wall, we observe that the domain wall satisfies
\begin{equation}\label{firstorder}
\partial_x\Omega_*=\mu\bigl(e_1-(\Omega_*)_1\Omega_*\bigr).
\end{equation}
Let $\mathbf M$ be the unique global solution to \eqref{LLM} given by Appendices~\ref{applwp} and~\ref{appgwp}. Moreover, define
\begin{equation}\label{Bdefini}
\mathbf{B}:=\mathbf{B}(\mathbf M)=\partial_x\mathbf M-\mu(e_1-\mathbf M_1\mathbf M).
\end{equation}
Since $|\mathbf M|=1$, we have
\begin{equation}\label{Btangen}
\mathbf{B}\cdot\mathbf M=0.
\end{equation}
Then, for any tangent vector field $V$ along $\mathbf M$, we define $\nabla_x,\nabla_t$ by
\begin{equation}\label{cova}
\nabla_xV=\partial_xV+(\partial_x\mathbf M\cdot V)\mathbf M.
\end{equation}
\begin{equation}\label{t-cova}
\nabla_tV=\partial_tV+(\partial_t\mathbf M\cdot V)\mathbf M,
\end{equation}
which are in fact the induced covariant derivatives on the pullback bundle ${\mathbf M}^*T\mathbb S^2$. 
The following proposition gives the evolution equation for $\mathbf{B}$.

\begin{prop}
Let $\mathbf M$ solve \eqref{LLM}, and let $\mathbf{B}$ be defined by \eqref{Bdefini}. Then
\begin{equation}\label{MevolB}
\partial_t\mathbf M=\mathbf M\wedge\left(\nabla_x\mathbf{B}-\mu\mathbf M_1\mathbf{B}\right),
\end{equation}
and
\begin{equation}\label{BcovEvol}
\nabla_t\mathbf{B}=\mathbf M\wedge\left(\nabla_x^2\mathbf{B}-(\mu^2+\mu \mathbf{B}_1)\mathbf{B}\right).
\end{equation}
\end{prop}

\begin{proof}
By direct computation,
\begin{align*}
\nabla_x\mathbf{B}
=&\partial_x\mathbf{B}+(\partial_x\mathbf M\cdot \mathbf{B})\mathbf M\\
=&\partial_x^2\mathbf M+\mu(\partial_x\mathbf M_1)\mathbf M+\mu\mathbf M_1\partial_x\mathbf M+\left(|\partial_x\mathbf M|^2-\mu\partial_x\mathbf M_1\right)\mathbf M\\
=&\partial_x^2\mathbf M+\mu\mathbf M_1\partial_x\mathbf M+|\partial_x\mathbf M|^2\mathbf M.
\end{align*}
Therefore, one finds 
$$
\begin{aligned}
\nabla_x\mathbf{B}-\mu\mathbf M_1\mathbf{B}
=&\partial_x^2\mathbf M+\mu\mathbf M_1\partial_x\mathbf M+|\partial_x\mathbf M|^2\mathbf M-\mu\mathbf M_1\left[\partial_x\mathbf M-\mu(e_1-\mathbf M_1\mathbf M)\right]\\
=&\partial_x^2\mathbf M+\mu^2\mathbf M_1e_1+\left(|\partial_x\mathbf M|^2-\mu^2\mathbf M_1^2\right)\mathbf M.
\end{aligned}
$$
Taking the cross product with $\mathbf M$ and using \eqref{LLM}, we obtain
$$
\mathbf M\wedge\left(\nabla_x\mathbf{B}-\mu\mathbf M_1\mathbf{B}\right)=\mathbf M\wedge\left(\partial_x^2\mathbf M+\mu^2\mathbf M_1e_1\right)=\partial_t\mathbf M,
$$
which proves \eqref{MevolB}.

Next, differentiating \eqref{Bdefini} in time and using $\mathbf{B}\cdot\mathbf M=0$, we obtain
\begin{align*}
\nabla_t\mathbf{B}
=&\partial_t\mathbf{B}+(\partial_t\mathbf M\cdot \mathbf{B})\mathbf M\\
=&\partial_x\partial_t\mathbf M+\mu(\partial_t\mathbf M_1)\mathbf M+\mu\mathbf M_1\partial_t\mathbf M+(\partial_t\mathbf M\cdot \mathbf{B})\mathbf M\\
=&\partial_x\partial_t\mathbf M+(\partial_x\mathbf M\cdot\partial_t\mathbf M)\mathbf M+\mu\mathbf M_1\partial_t\mathbf M\\
=&\nabla_x\partial_t\mathbf M+\mu\mathbf M_1\partial_t\mathbf M.
\end{align*}
Applying \eqref{MevolB} and the the compatibility of the covariant derivative and the complex structure, namely,
\begin{equation}
\label{cccom}
\nabla_x(\mathbf M\wedge V)=\mathbf M\wedge\nabla_xV,
\end{equation}
we obtain
\begin{align*}
\nabla_t\mathbf{B}
=&\nabla_x\left(\mathbf M\wedge\left(\nabla_x\mathbf{B}-\mu\mathbf M_1\mathbf{B}\right)\right)+\mu\mathbf M_1\mathbf M\wedge\left(\nabla_x\mathbf{B}-\mu\mathbf M_1\mathbf{B}\right)\\
=&\mathbf M\wedge\left(\nabla_x\left(\nabla_x\mathbf{B}-\mu\mathbf M_1\mathbf{B}\right)+\mu\mathbf M_1\left(\nabla_x\mathbf{B}-\mu\mathbf M_1\mathbf{B}\right)\right)\\
=&\mathbf M\wedge\left(\nabla_x^2\mathbf{B}-\mu(\partial_x\mathbf M_1)\mathbf{B}-\mu^2\mathbf M_1^2\mathbf{B}\right).
\end{align*}
Taking the $e_1$ component in \eqref{Bdefini} gives
$$\mathbf{B}_1=\partial_x\mathbf M_1-\mu(1-\mathbf M_1^2),$$
and hence
$$\mu\partial_x\mathbf M_1+\mu^2\mathbf M_1^2=\mu \mathbf{B}_1+\mu^2.$$
Substituting this identity into the preceding equation yields
$$\nabla_t\mathbf{B}=\mathbf M\wedge\left(\nabla_x^2\mathbf{B}-(\mu^2+\mu \mathbf{B}_1)\mathbf{B}\right),$$
which proves \eqref{BcovEvol}.
\end{proof}

To reduce the equation to a complex scalar equation, we construct an orthonormal frame $(\mathbf{E}_1,\mathbf{E}_2)$ along $\mathbf M$ as follows.\\
At $t=0$, choose $\mathbf{E}_1(0,x)\in T_{\mathbf M(0,x)}\mathbb S^2$ satisfying
\begin{equation}
\label{E1nor}
|\mathbf{E}_1(0,x)|=1,\quad \nabla_x\mathbf{E}_1(0,x)=0,\quad \lim_{x\to-\infty}\mathbf{E}_1(0,x)=-e_2.
\end{equation}
At $x=0$, we prescribe the time evolution of $\mathbf{E}_1$ by
\begin{equation}\label{frameEvo}
\nabla_t\mathbf{E}_1(t,0)=-\left(\mu \mathbf{B}_1(t,0)+\frac12|\mathbf{B}(t,0)|^2\right)\bigl(\mathbf M(t,0)\wedge \mathbf{E}_1(t,0)\bigr),\qquad t\geq0.
\end{equation}
For each $t\geq0$, we extend $\mathbf{E}_1(t,0)$ along the $x$-direction by
\begin{equation}\label{Xpara}
\nabla_x\mathbf{E}_1(t,x)=0.
\end{equation}
We then define
\begin{equation}\label{E2def}
\mathbf{E}_2(t,x)=\mathbf M(t,x)\wedge \mathbf{E}_1(t,x).
\end{equation}
The following two propositions, together with\eqref{cccom}, shows that this construction yields a unique orthonormal frame $(\mathbf{E}_1,\mathbf{E}_2)$ of $T_{\mathbf M}\mathbb S^2$ that is parallel in $x$:
\begin{equation}\label{frameX}
\nabla_x\mathbf{E}_1=\nabla_x\mathbf{E}_2=0.
\end{equation}
\begin{prop}
There exists a unique vector field $\mathbf{E}_1(0,\cdot)$ satisfying \eqref{E1nor}.

\begin{proof}
Choose a unit vector $F_1(0)\in T_{\mathbf M(0,0)}\mathbb S^2$ and solve $\nabla_xF_1=0$ using the existence and uniqueness theory for linear ordinary differential equations. Since $\partial_x(F_1\cdot\mathbf M)=0$ and $\partial_x|F_1|^2=0$, we have $F_1(x)\in T_{\mathbf M(0,x)}\mathbb S^2$ and $|F_1(x)|=1$ for all $x\in\mathbb R$. Moreover, for $j=2,3$ and $x<y$, we obtain
$$
\|(F_1(y)-F_1(x))\cdot e_j\|\le\int_x^y|\partial_x\mathbf M||\mathbf M_j|{\rm d}s\le\|\partial_x\mathbf M\|_{L^2(x,y)}\|\mathbf M_j\|_{L^2(x,y)}<\infty.
$$
Hence there exist $a,b\in\mathbb R$ such that
$$\lim_{x\to-\infty}F_1(x)\cdot e_2=a,\quad \lim_{x\to-\infty}F_1(x)\cdot e_3=b.$$
Since $\mathbf M(0,x)\to-e_1$ as $x\to-\infty$, we have
$$\lim_{x\to-\infty}F_1(x)=ae_2+be_3.$$
Therefore
$$\lim_{x\to-\infty}F_1(x)=ae_2+be_3,\quad a^2+b^2=1.$$
Let $F_2=\mathbf M\wedge F_1$ and set $\mathbf{E}_1=-aF_1-bF_2$. Then $\nabla_x\mathbf{E}_1=0$ by \eqref{cccom}, $|\mathbf{E}_1|=1$, $\mathbf{E}_1\cdot\mathbf M=0$, and $\mathbf{E}_1(x)\to-e_2$ as $x\to-\infty$. Finally, if $\widetilde{\mathbf{E}}_1$ is another vector field satisfying these properties, then $\partial_x(\mathbf{E}_1\cdot\widetilde{\mathbf{E}}_1)=0$, while $\mathbf{E}_1\cdot\widetilde{\mathbf{E}}_1\to1$ as $x\to-\infty$. Thus $\mathbf{E}_1\cdot\widetilde{\mathbf{E}}_1\equiv1$, and therefore $\widetilde{\mathbf{E}}_1=\mathbf{E}_1$.
\end{proof}
\end{prop}

\begin{prop}
The time evolution \eqref{frameEvo} and the spatial extension \eqref{Xpara} preserve the tangency and unit length of $\mathbf{E}_1$.
\end{prop}

\begin{proof}
The evolution \eqref{frameEvo} preserves the properties $\mathbf{E}_1(t,0)\in T_{\mathbf M(t,0)}\mathbb S^2$ and $|\mathbf{E}_1(t,0)|=1$, as can be verified from $\partial_t\bigl(\mathbf{E}_1(t,0)\cdot\mathbf M(t,0)\bigr)=0$ and $\partial_t|\mathbf{E}_1(t,0)|^2=0$. Similarly, the extension \eqref{Xpara} preserves $\mathbf{E}_1(t,x)\in T_{\mathbf M(t,x)}\mathbb S^2$ and $|\mathbf{E}_1(t,x)|=1$.
\end{proof}

Let $V=a\mathbf{E}_1+b\mathbf{E}_2$ be a tangent vector field along $\mathbf M$. The definition \eqref{E2def} gives
\begin{equation}\label{coordina}
\mathbf M\wedge V=-b\mathbf{E}_1+a\mathbf{E}_2,
\end{equation}
which together with \eqref{frameX} yields 
\begin{equation}\label{X-cord}
\nabla_xV=(\partial_xa)\mathbf{E}_1+(\partial_xb)\mathbf{E}_2.
\end{equation}
For the time derivative, using
\begin{equation*}
\nabla_t\mathbf{E}_1=(\nabla_t\mathbf{E}_1\cdot \mathbf{E}_2)\mathbf{E}_2,\qquad \nabla_t\mathbf{E}_2=-(\nabla_t\mathbf{E}_1\cdot \mathbf{E}_2)\mathbf{E}_1,
\end{equation*}
we obtain
\begin{equation}\label{corvtfra}
\begin{aligned}
\nabla_tV
=&(\partial_ta)\mathbf{E}_1+a\nabla_t\mathbf{E}_1+(\partial_tb)\mathbf{E}_2+b\nabla_t\mathbf{E}_2\\
=&\left[\partial_ta-b(\nabla_t\mathbf{E}_1\cdot \mathbf{E}_2)\right]\mathbf{E}_1+\left[\partial_tb+a(\nabla_t\mathbf{E}_1\cdot \mathbf{E}_2)\right]\mathbf{E}_2.
\end{aligned}
\end{equation}

\begin{rem}
Identify the tangent vector field $V=a\mathbf{E}_1+b\mathbf{E}_2$ with its complex coordinate $a+ib$. Then \eqref{coordina}, \eqref{X-cord}, and \eqref{corvtfra} show that, in these coordinates, $\mathbf M\wedge(\,\cdot\,)$ is represented by multiplication by $i$, $\nabla_x$ is represented by $\partial_x$, while $\nabla_t$ is represented by $\partial_t+i(\nabla_t\mathbf{E}_1\cdot \mathbf{E}_2)$.
\end{rem}

Since $\mathbf{B}$ is tangent to $\mathbb S^2$ along $\mathbf M$, we define
\begin{equation}
\label{Q-defi}
\mathbf Q=\frac12\left(\mathbf{B}\cdot \mathbf{E}_1+i\mathbf{B}\cdot \mathbf{E}_2\right).
\end{equation}
Applying \eqref{coordina} and \eqref{corvtfra} to \eqref{BcovEvol}, we obtain
\begin{equation*}
\left(\partial_t+i\nabla_t\mathbf{E}_1\cdot \mathbf{E}_2\right)\mathbf Q=i\left(\partial_x^2\mathbf Q-(\mu^2+\mu \mathbf{B}_1)\mathbf Q\right).
\end{equation*}
Equivalently, one has 
\begin{equation}
\label{Qequ}
i\partial_t\mathbf Q-H_\eta\mathbf Q=\left(\mu \mathbf{B}_1+\nabla_t\mathbf{E}_1\cdot \mathbf{E}_2\right)\mathbf Q.
\end{equation}
By  \eqref{frameX} and the compatibility of the Levi-Civita connection with the metric, i.e. $\partial_x(V\cdot W)=(\nabla_xV)\cdot W+V\cdot\nabla_xW$, we get
\begin{equation*}
\partial_x(\nabla_t\mathbf{E}_1\cdot \mathbf{E}_2)=(\nabla_x\nabla_t\mathbf{E}_1)\cdot \mathbf{E}_2.
\end{equation*}
Combining the curvature identity $\nabla_x\nabla_tV-\nabla_t\nabla_xV=(\partial_t\mathbf M\cdot V)\partial_x\mathbf M-(\partial_x\mathbf M\cdot V)\partial_t\mathbf M$ on $\mathbb S^2$ with \eqref{frameX}, we obtain
\begin{align*}
\nabla_x\nabla_t\mathbf{E}_1
=&(\partial_t\mathbf M\cdot \mathbf{E}_1)\partial_x\mathbf M-(\partial_x\mathbf M\cdot \mathbf{E}_1)\partial_t\mathbf M.
\end{align*}
Substituting \eqref{MevolB} yields
\begin{align*}
\partial_t\mathbf M\cdot \mathbf{E}_1
=&\left[\mathbf M\wedge\left(\nabla_x\mathbf{B}-\mu\mathbf M_1\mathbf{B}\right)\right]\cdot \mathbf{E}_1
=-\left(\nabla_x\mathbf{B}-\mu\mathbf M_1\mathbf{B}\right)\cdot \mathbf{E}_2,\\
\partial_t\mathbf M\cdot \mathbf{E}_2
=&\left[\mathbf M\wedge\left(\nabla_x\mathbf{B}-\mu\mathbf M_1\mathbf{B}\right)\right]\cdot \mathbf{E}_2
=\left(\nabla_x\mathbf{B}-\mu\mathbf M_1\mathbf{B}\right)\cdot \mathbf{E}_1.
\end{align*}
Thus
\begin{align}
\partial_x(\nabla_t\mathbf{E}_1\cdot \mathbf{E}_2)
={}&-\left(\nabla_x\mathbf{B}-\mu\mathbf M_1\mathbf{B}\right)\cdot \mathbf{E}_2\left(\partial_x\mathbf M\cdot \mathbf{E}_2\right) \notag\\
-&\left(\partial_x\mathbf M\cdot \mathbf{E}_1\right)\left(\nabla_x\mathbf{B}-\mu\mathbf M_1\mathbf{B}\right)\cdot \mathbf{E}_1 \notag\\
={}&-\partial_x\mathbf M\cdot\left(\nabla_x\mathbf{B}-\mu\mathbf M_1\mathbf{B}\right).
\label{Xtderi}
\end{align}
By \eqref{cova}, we have
\begin{align*}
\mu\partial_x \mathbf{B}_1
=&\mu e_1\cdot\partial_x\mathbf{B}
=\mu e_1\cdot\nabla_x\mathbf{B}-\mu\mathbf M_1(\partial_x\mathbf M\cdot \mathbf{B}).
\end{align*}
Together with \eqref{Xtderi} and \eqref{Bdefini}, this yields
\begin{align}
\partial_x\left(\nabla_t\mathbf{E}_1\cdot \mathbf{E}_2+\mu \mathbf{B}_1\right)
=&-(\partial_x\mathbf M-\mu e_1)\cdot\nabla_x\mathbf{B} \notag\\
=&-(\mathbf{B}-\mu\mathbf M_1\mathbf M)\cdot\nabla_x\mathbf{B} \notag\\
=&-\mathbf{B}\cdot\nabla_x\mathbf{B} \notag\\
=&-\frac12\partial_x|\mathbf{B}|^2.
\label{tiBnor}
\end{align}
The definition\eqref{Q-defi} and tangency \eqref{Btangen} imply
\begin{equation}
\label{BQre}
|\mathbf{B}|^2=(\mathbf{B}\cdot \mathbf{E}_1)^2+(\mathbf{B}\cdot \mathbf{E}_2)^2=4|\mathbf Q|^2.
\end{equation}
Hence \eqref{tiBnor} becomes
\begin{equation*}
\partial_x\left(\nabla_t\mathbf{E}_1\cdot \mathbf{E}_2+\mu \mathbf{B}_1+2|\mathbf Q|^2\right)=0.
\end{equation*}
By \eqref{frameEvo}, we get
\begin{equation*}
\nabla_t\mathbf{E}_1(t,0)\cdot \mathbf{E}_2(t,0)+\mu \mathbf{B}_1(t,0)+\frac12|\mathbf{B}(t,0)|^2=0.
\end{equation*}
Integrating \eqref{tiBnor} in $x$ therefore gives
\begin{equation}
\label{conne}
\nabla_t\mathbf{E}_1\cdot \mathbf{E}_2+\mu \mathbf{B}_1+\frac12|\mathbf{B}|^2=0.
\end{equation}
The relation \eqref{BQre} then yields
\begin{equation}
\label{timeQcon}
\nabla_t\mathbf{E}_1\cdot \mathbf{E}_2+\mu \mathbf{B}_1=-2|\mathbf Q|^2.
\end{equation}
Substituting \eqref{timeQcon} into \eqref{Qequ}, we obtain
\begin{equation*}
i\partial_t\mathbf Q-H_\eta\mathbf Q+2|\mathbf Q|^2\mathbf Q=0.
\end{equation*}
The preceding argument and the well-posedness results in Appendix~\ref{applwp} and Appendix~\ref{appgwp} prove the following proposition.

\begin{prop}
Let $N\ge6$ be an integer, and let $\mathbf Q$ be defined by \eqref{Q-defi}. Then $\mathbf Q \in L^\infty_{\rm loc}\bigl(\mathbb R_{\ge0} ; H^{N-1}(\mathbb R)\bigr)$ and $\partial_t\mathbf Q\in L^\infty_{\rm loc}\bigl(\mathbb R_{\ge0};H^{N-3}(\mathbb R)\bigr)$. Moreover,
\begin{equation}\label{Qcub1}
i\partial_t\mathbf Q-H_\eta\mathbf Q+2|\mathbf Q|^2\mathbf Q=0
\end{equation}
\end{prop}

\begin{rem}
Kundu and Pashaev \cite{KP1983} found the gauge equivalence of the cubic NLS with \eqref{LLM} using a Lax pair. Here we get the cubic NLS from the evolution of $B$ which measures the distance to domain walls and gauge choices .
\end{rem}

\begin{prop}
Let $(\mathbf{E}_1,\mathbf{E}_2)$ be the frame constructed above. Then, for all $t\in\mathbb R_{\ge0}$,
\begin{equation}\label{franormli}
\lim_{x\to-\infty}\mathbf{E}_1(t,x)=-e_2,\qquad \lim_{x\to-\infty}\mathbf{E}_2(t,x)=e_3.
\end{equation}
\end{prop}

\begin{proof}
The assertion at $t=0$ follows from \eqref{E1nor} and \eqref{E2def}. Fix $t>0$. Theorem~\ref{gwpM} and Sobolev estimates give $\mathbf M_2,\mathbf M_3\in L^\infty([0,t];H^N(\mathbb R))$, $\partial_x\mathbf M\in L^\infty([0,t];H^{N-1}(\mathbb R))$, $\partial_t\mathbf M\in L^\infty([0,t];H^{N-2}(\mathbb R))$, and $e_1-\mathbf M_1\mathbf M\in L^\infty([0,t];H^1(\mathbb R))$. Since $N\geq6$, the definition \eqref{Bdefini} gives $\mathbf B,\partial_t\mathbf M\in L^\infty([0,t];H^1(\mathbb R))$, so that, for almost every $s\in[0,t]$, $\mathbf B(s,x)\to0$ and $\partial_s\mathbf M(s,x)\to0$ as $x\to-\infty$ by $H^1(\mathbb R)\hookrightarrow C_0(\mathbb R)$.

The  identity \eqref{conne} implies
\begin{equation*}
\nabla_t\mathbf E_1=-\left(\mu\mathbf B_1+\frac12|\mathbf B|^2\right)\mathbf E_2.
\end{equation*}
Expanding the covariant derivative using \eqref{t-cova}, we have
\begin{equation}\label{E1identity}
\partial_t\mathbf E_1=-\left(\mu\mathbf B_1+\frac12|\mathbf B|^2\right)\mathbf E_2-(\partial_t\mathbf M\cdot\mathbf E_1)\mathbf M,
\end{equation}
and hence
\begin{equation}\label{dE1eq}
|\partial_t\mathbf E_1|\leq\mu|\mathbf B|+\frac12|\mathbf B|^2+|\partial_t\mathbf M|.
\end{equation}
Moreover, by \eqref{E1identity}, for almost every $s\in[0,t]$,
\begin{equation}\label{sE1}
\lim_{x\to-\infty}\partial_s\mathbf E_1(s,x)=0.
\end{equation}
By \eqref{dE1eq}, for every fixed $x$,
\begin{align}
\|\partial_s\mathbf E_1(\cdot,x)\|_{L_s^\infty(0,t)}
&\leq\mu\|\mathbf B\|_{L_s^\infty L_x^\infty}+\frac12\|\mathbf B\|_{L_s^\infty L_x^\infty}^2+\|\partial_s\mathbf M\|_{L_s^\infty L_x^\infty}\notag\\
&\lesssim\|\mathbf B\|_{L_s^\infty H_x^1}+\|\mathbf B\|_{L_s^\infty H_x^1}^2+\|\partial_s\mathbf M\|_{L_s^\infty H_x^1}<\infty.
\label{E1tdL}
\end{align}
Since $|\mathbf E_1|=1$, we also have
\begin{equation*}
\|\mathbf E_1(\cdot,x)\|_{L_s^\infty(0,t)}\leq1.
\end{equation*}
Therefore,
\begin{equation*}
\mathbf E_1(\cdot,x)\in W^{1,\infty}(0,t),
\end{equation*}
and hence $\mathbf E_1(\cdot,x)$ is absolutely continuous on $[0,t]$. Thus,
\begin{equation}\label{E1integ}
\mathbf E_1(t,x)-\mathbf E_1(0,x)=\int_0^t\partial_s\mathbf E_1(s,x){\rm d}s.
\end{equation}
Applying the dominated convergence theorem in \eqref{E1integ}, with \eqref{sE1} and \eqref{E1tdL}, gives
\begin{equation}\label{E1minus}
\lim_{x\to-\infty}\bigl(\mathbf E_1(t,x)-\mathbf E_1(0,x)\bigr)=0.
\end{equation}
Using \eqref{E1nor} and \eqref{E1minus}, we obtain
\begin{equation*}
\lim_{x\to-\infty}\mathbf E_1(t,x)=-e_2.
\end{equation*}
Similarly, by \eqref{E2def}, $\partial_t\mathbf E_2=\partial_t\mathbf M\wedge\mathbf E_1+\mathbf M\wedge\partial_t\mathbf E_1$. Thus, the preceding estimates imply that, for every fixed $x$, $\mathbf E_2(\cdot,x)\in W^{1,\infty}(0,t)$, while, for almost every $s\in[0,t]$, $\partial_s\mathbf E_2(s,x)\to0$ as $x\to-\infty$ and $|\partial_s\mathbf E_2(s,x)|\lesssim1$. Therefore, by the dominated convergence theorem, \eqref{E2def}, and \eqref{E1nor}, together with $\mathbf M(0,x)\to-e_1$ as $x\to-\infty$, we obtain $\lim_{x\to-\infty}\mathbf E_2(t,x)=e_3$.
\end{proof}

\subsection{The modulated cubic NLS equation and reconstruction of the perturbation}

We now carry out the geometric reduction for the modulated magnetization $M$, which satisfies \eqref{phileft}. We begin by defining the associated tangent vector field: 
\begin{equation}
\label{BMdef}
B:=B(M) = \partial_xM-\mu(e_1-M_1M).
\end{equation}
The relation between $M$ and $\mathbf M$, together with \eqref{BMdef} and \eqref{Bdefini}, gives
\begin{equation}\label{covaBM}
B(t,x) = R_{-\phi(t)-\eta y(t)}\mathbf B\bigl(t,x+y(t)\bigr).
\end{equation}
For all $t\in\mathbb{R}_{\ge 0}$, let $E_1(t,\cdot)$ be the unique unit tangent vector field along $M(t,\cdot)$ that is parallel in the $x$-direction and satisfies $\lim_{x\to-\infty}E_1(t,x)=-e_2$. We then define $E_2(t,x)=M(t,x)\wedge E_1(t,x)$. Since $M(t,x)\to-e_1$ as $x\to-\infty$, the pair $(E_1,E_2)$ is a spatially parallel oriented orthonormal frame along $M$, with
\begin{equation}\label{eMframe}
\lim_{x\to-\infty}E_1(t,x)=-e_2,\quad\lim_{x\to-\infty}E_2(t,x)=e_3.
\end{equation}
We define
\begin{equation}
\label{QMee}
Q(t,x)=\frac12\left(B(t,x)\cdot E_1(t,x)+iB(t,x)\cdot E_2(t,x)\right).
\end{equation}The following proposition establishes the relation between $Q$ and $\mathbf Q$.
\begin{prop}\label{Qcopro}
For all $t\in I$ and $x\in\mathbb R$, we have
\begin{equation}\label{Qmath}
Q(t,x) = e^{i(\phi(t)+\eta y(t))} \mathbf Q\bigl(t,x+y(t)\bigr).
\end{equation}
\end{prop}

\begin{proof}
Substituting \eqref{covaBM} into the definition \eqref{QMee} gives
\begin{align}
2Q(t,x)
=&B(t,x)\cdot E_1(t,x)+iB(t,x)\cdot E_2(t,x) \notag\\
=&R_{-\phi(t)-\eta y(t)}\mathbf B\bigl(t,x+y(t)\bigr)\cdot E_1(t,x) \notag\\
&+iR_{-\phi(t)-\eta y(t)}\mathbf B\bigl(t,x+y(t)\bigr)\cdot E_2(t,x).
\label{Qfirp}
\end{align}
To rewrite \eqref{Qfirp} in terms of $\mathbf Q$, we express
$E_1$ and $E_2$ in terms of the corresponding frame along $\mathbf M$.
Since
$$M(t,x)=R_{-\phi(t)-\eta y(t)}\mathbf M\bigl(t,x+y(t)\bigr),$$
the vectors
$$R_{-\phi(t)-\eta y(t)}\mathbf E_j\bigl(t,x+y(t)\bigr),\quad j=1,2,$$
form an orthonormal basis of $T_{M(t,x)}\mathbb S^2$. Hence there exist real valued
functions $a(t,x)$ and $b(t,x)$ satisfying $a(t,x)^2+b(t,x)^2=1$ such that
\begin{equation}\label{undecoe}
\begin{aligned}
E_1(t,x)& = a(t,x)R_{-\phi(t)-\eta y(t)}\mathbf E_1\bigl(t,x+y(t)\bigr)\\
&\quad-b(t,x)R_{-\phi(t)-\eta y(t)}\mathbf E_2\bigl(t,x+y(t)\bigr),\\
E_2(t,x)&=b(t,x)R_{-\phi(t)-\eta y(t)}\mathbf E_1\bigl(t,x+y(t)\bigr)\\
&\quad+a(t,x)R_{-\phi(t)-\eta y(t)}\mathbf E_2\bigl(t,x+y(t)\bigr).
\end{aligned}
\end{equation}
Applying $\nabla_x$ to \eqref{undecoe} and using \eqref{frameX}, we get
\begin{align*}
\partial_x a(t,x)R_{-\phi(t)-\eta y(t)}\mathbf E_1\bigl(t,x+y(t)\bigr)-\partial_x b(t,x) R_{-\phi(t)-\eta y(t)}\mathbf E_2\bigl(t,x+y(t)\bigr)=0.
\end{align*}
Hence
\begin{equation*}
\partial_x a(t,x)=\partial_x b(t,x)=0.
\end{equation*}
Therefore, we arrive at
\begin{align*}
E_1(t,x)&=a(t)R_{-\phi(t)-\eta y(t)}\mathbf E_1\bigl(t,x+y(t)\bigr) -b(t)R_{-\phi(t)-\eta y(t)}
\mathbf E_2\bigl(t,x+y(t)\bigr).
\end{align*}
For each $t\in I$, taking $x\to-\infty$ and using \eqref{franormli} and \eqref{eMframe} yields
\begin{equation*}
-e_2=-a(t)R_{-\phi(t)-\eta y(t)}e_2-b(t)R_{-\phi(t)-\eta y(t)}e_3.
\end{equation*}
By \eqref{eRphi}, we have
\begin{equation*}
a(t)=\cos\bigl(\phi(t)+\eta y(t)\bigr),
\quad
b(t)=\sin\bigl(\phi(t)+\eta y(t)\bigr).
\end{equation*}
Thus
\begin{equation}
\label{transf}
\begin{aligned}
E_1(t,x)&=\cos\bigl(\phi(t)+\eta y(t)\bigr)R_{-\phi(t)-\eta y(t)}\mathbf E_1\bigl(t,x+y(t)\bigr)\\
&\quad-\sin\bigl(\phi(t)+\eta y(t)\bigr)R_{-\phi(t)-\eta y(t)}\mathbf E_2\bigl(t,x+y(t)\bigr),\\
E_2(t,x)&=\sin\bigl(\phi(t)+\eta y(t)\bigr)R_{-\phi(t)-\eta y(t)}\mathbf E_1\bigl(t,x+y(t)\bigr)\\
&\quad+\cos\bigl(\phi(t)+\eta y(t)\bigr)R_{-\phi(t)-\eta y(t)}\mathbf E_2\bigl(t,x+y(t)\bigr).
\end{aligned}
\end{equation}
Substituting \eqref{transf} into \eqref{Qfirp}, we obtain
\begin{align*}
2Q(t,x)
=&\cos\bigl(\phi(t)+\eta y(t)\bigr)\mathbf B\bigl(t,x+y(t)\bigr)\cdot\mathbf E_1\bigl(t,x+y(t)\bigr)\\
&-\sin\bigl(\phi(t)+\eta y(t)\bigr)\mathbf B\bigl(t,x+y(t)\bigr)\cdot\mathbf E_2\bigl(t,x+y(t)\bigr)\\
&+i\sin\bigl(\phi(t)+\eta y(t)\bigr)\mathbf B\bigl(t,x+y(t)\bigr)\cdot\mathbf E_1\bigl(t,x+y(t)\bigr)\\
&+i\cos\bigl(\phi(t)+\eta y(t)\bigr)\mathbf B\bigl(t,x+y(t)\bigr)\cdot\mathbf E_2\bigl(t,x+y(t)\bigr)\\
=&2e^{i(\phi(t)+\eta y(t))}\mathbf Q\bigl(t,x+y(t)\bigr).
\end{align*}
which proves \eqref{Qmath}.
\end{proof}
Proposition~\ref{Qcopro}
relates $Q$ to $\mathbf Q$ and gives
\begin{equation*}
i\partial_tQ+(\partial_x^2-\mu^2)Q+2|Q|^2Q=i\dot y\,\partial_xQ-(\dot\phi+\eta\dot y)Q.
\end{equation*}
We now derive its relation to the perturbation variable $z$.
Starting from \eqref{zmphi}, we have
\begin{equation}
\label{comNp}
M\cdot N_*=\frac12\bigl(1+M\cdot\Omega_*\bigr)\operatorname{Re}z,
\quad M\cdot P_*=\frac12\bigl(1+M\cdot\Omega_*\bigr)\operatorname{Im}z.
\end{equation}
The component formulas \eqref{comNp} and the orthogonality relations \eqref{fraorth} give
\begin{align}
&\bigl(1+M\cdot\Omega_*\bigr)\left[M\cdot\Omega_*-1+\frac14\bigl(1+M\cdot\Omega_*\bigr)|z|^2\right]\notag\\
=&\bigl(M\cdot\Omega_*\bigr)^2-1+\frac14\bigl(1+M\cdot\Omega_*\bigr)^2|z|^2\notag\\
=&\bigl(M\cdot\Omega_*\bigr)^2+\bigl(M\cdot N_*\bigr)^2+\bigl(M\cdot P_*\bigr)^2-1\notag\\
=&0.
\label{omegaid}
\end{align}
By Remark~\ref{conAnon}, \eqref{omegaid} yields, for $t\in I$ and $x\in\mathbb R$,
\begin{equation}\label{Momeg}
M(t,x)\cdot\Omega_*(x)=\frac{1-\frac14|z(t,x)|^2}{1+\frac14|z(t,x)|^2}.
\end{equation}
Substituting \eqref{Momeg} into\eqref{comNp}, we obtain
\begin{equation*}
M\cdot N_*=\frac{\operatorname{Re}z}{1+\frac14|z|^2},\quad M\cdot P_*=\frac{\operatorname{Im}z}{1+\frac14|z|^2}.
\end{equation*}
Therefore
\begin{equation}
\label{msre}
M=\frac{1-\frac14|z|^2}{1+\frac14|z|^2}\Omega_*+\frac{\operatorname{Re}z}{1+\frac14|z|^2}N_*+\frac{\operatorname{Im}z}{1+\frac14|z|^2}e_3.
\end{equation}
To express $E_1$ and $E_2$ in terms of $z$, we introduce an orthonormal frame $\{M,\mathcal E_1,\mathcal E_2\}$ along $M$ by setting
\begin{equation}\label{nafraE}
\begin{aligned}
\mathcal E_1&=\frac{1}{1+\frac14|z|^2}\left[-\operatorname{Re}z\Omega_*+\left(1+\frac{(\operatorname{Im}z)^2-(\operatorname{Re}z)^2}{4}\right)N_*-\frac{\operatorname{Re}z\operatorname{Im}z}{2}e_3\right],\\
\mathcal E_2&=\frac{1}{1+\frac14|z|^2}\left[-\operatorname{Im}z\Omega_*-\frac{\operatorname{Re}z\operatorname{Im}z}{2}N_*+\left(1+\frac{(\operatorname{Re}z)^2-(\operatorname{Im}z)^2}{4}\right)e_3\right].
\end{aligned}
\end{equation}
A direct computation gives
\begin{equation*}
|\mathcal E_1|=|\mathcal E_2|=1,\quad \mathcal E_1\cdot\mathcal E_2=0,\quad M\wedge\mathcal E_1=\mathcal E_2.
\end{equation*}
Thus $\{M,\mathcal E_1,\mathcal E_2\}$ is an orthonormal frame along $M$. Since $z(t,\cdot)\in H^1(\mathbb R)$ for $t\in I$, we have $z(t,x)\to0$ as $x\to-\infty$. Hence, by \eqref{nafraE} and $N_*(x)\to-e_2$,
\begin{equation}\label{nanor}
\mathcal E_1(t,x)\to-e_2,\quad \mathcal E_2(t,x)\to e_3,\quad \text{as }x\to-\infty.
\end{equation}
Note that
\begin{equation}\label{frcon}
\partial_x\mathcal E_1\cdot\mathcal E_2=-\frac{\frac12\operatorname{Im}(\overline z\partial_xz)+\mu S_\eta\operatorname{Im}z}{1+\frac14|z|^2}.
\end{equation}
Moreover,
$$\int_{\mathbb R}\left|\frac{\frac12\operatorname{Im}(\overline z\,\partial_xz)+\mu S_\eta\operatorname{Im}z}{1+\frac14|z|^2}\right|{\rm d}x \lesssim \|z\|_{L^2}\|\partial_xz\|_{L^2}+\|S_\eta\|_{L^2}\|z\|_{L^2}<\infty.$$
We therefore define
\begin{equation}\label{xchi}
\chi(t,x)=-\int_{-\infty}^x\frac{\frac12\operatorname{Im}(\overline z\,\partial_sz)+\mu S_\eta\operatorname{Im}z}{1+\frac14|z|^2}(t,s){\rm d}s.
\end{equation}
By \eqref{frcon}, we have
\begin{equation}\label{xchie}
\partial_x\chi=\partial_x\mathcal E_1\cdot\mathcal E_2,
\end{equation} 
while
\eqref{xchi} gives 
\begin{equation}
\label{chidecay}
\chi(t,x)\to 0 \quad \text{as } x\to -\infty.
\end{equation}
Now define
$$\widehat E_1 = \cos\chi\,\mathcal E_1-\sin\chi\,\mathcal E_2,\quad
\widehat E_2 = \sin\chi\,\mathcal E_1+\cos\chi\,\mathcal E_2.$$

Since $\{M,\mathcal E_1,\mathcal E_2\}$ is an oriented orthonormal frame, so is $\{M,\widehat E_1,\widehat E_2\}$. Equation \eqref{xchie} gives $\nabla_x\widehat E_1=\nabla_x\widehat E_2=0$. Moreover, \eqref{nanor} and \eqref{chidecay} imply $\widehat E_1(t,x)\to-e_2$ and $\widehat E_2(t,x)\to e_3$. Uniqueness of the spatially parallel frame satisfying \eqref{eMframe} therefore gives
\begin{equation}
\label{rotae}
E_1 = \cos\chi\,\mathcal E_1-\sin\chi\,\mathcal E_2, \quad E_2 = \sin\chi\,\mathcal E_1+\cos\chi\,\mathcal E_2.
\end{equation}
The preceding analysis yields the following useful representation of $Q$.

\begin{rem}
The choice of $\mathcal E_1$ and $\mathcal E_2$ in \eqref{nafraE} is naturally induced by the representation \eqref{msre}. Indeed,
$$\mathcal E_1=\left(1+\frac14|z|^2\right)\partial_{\operatorname{Re}z}M,\qquad
\mathcal E_2=\left(1+\frac14|z|^2\right)\partial_{\operatorname{Im}z}M.$$
\end{rem}

\begin{prop}
For $t\in I$, $Q$ admits the representation
\begin{equation}\label{Q-z}
Q = -\frac12e^{i\chi}\frac{D_\eta^*z+\frac{\mu S_\eta}{2}z^2}{1+\frac14|z|^2},
\end{equation}
where $\chi$ is defined in (\ref{xchi}). 
\end{prop}

\begin{proof}
Using \eqref{msre}, \eqref{BMdef}, \eqref{rotafra}, and \eqref{nafraE}, a direct computation gives
\begin{equation*}
B\cdot\mathcal E_1+iB\cdot\mathcal E_2 = \frac{\partial_xz+\mu T_\eta z-\frac{\mu S_\eta}{2}z^2}{1+\frac14|z|^2}.
\end{equation*}
Thus  by \eqref{rotae} one has 
\begin{align*}
2Q&= B\cdot E_1+iB\cdot E_2 =e^{i\chi}\left(B\cdot\mathcal E_1+iB\cdot\mathcal E_2\right)\\
&=e^{i\chi}\frac{\partial_xz+\mu T_\eta z-\frac{\mu S_\eta}{2}z^2}{1+\frac14|z|^2}\\
&=-e^{i\chi}\frac{D_\eta^*z+\frac{\mu S_\eta}{2}z^2}{1+\frac14|z|^2},
\end{align*}
which proves \eqref{Q-z}.
\end{proof}

\subsection{Weighted profile estimates for the cubic NLS}\label{4}
 We begin with a trilinear estimate for the free Schr\"odinger flow that isolates the leading cubic interaction and bounds the remainder. We will use it below and again in Section~\ref{ssca}.

The following lemma and proposition are somewhat standard, see   Kato and Pusateri~\cite{KP2011}. We include a proof for readers' convenience. 

\begin{lem}
\label{tril}
Let $t\geq1$,  and let $f_1,f_2,f_3$ satisfy $f_j, \ xf_j\in L^2(\mathbb R), \ j=1,2,3.$ Then
\begin{equation}\label{trili}
\begin{aligned}
&e^{it\xi^2}\mathcal F\left(e^{it\partial_x^2}f_1\,e^{it\partial_x^2}f_2\,\overline{e^{it\partial_x^2}f_3}\right)(\xi)\\
=&\frac{1}{2t}
\mathcal F\Bigg[e^{-ix^2/(4t)}\mathcal F^{-1}\Bigg(\mathcal F\left[e^{iz^2/(4t)}f_1(z)\right]\mathcal F\left[e^{iz^2/(4t)}f_2(z)\right]\overline{\mathcal F\left[e^{iz^2/(4t)}f_3(z)\right]}\Bigg)\Bigg](\xi).
\end{aligned}
\end{equation}
Moreover,
\begin{equation}\label{freere}
\begin{aligned}
&\Bigg\|e^{it\xi^2}\mathcal F\left[e^{it\partial_x^2}f_1e^{it\partial_x^2}f_2\overline{e^{it\partial_x^2}f_3}\right]-\frac{1}{2t}\widehat{f_1}\widehat{f_2}\overline{\widehat{f_3}}\Bigg\|_{L^\infty_\xi}\\
\lesssim&t^{-1-\alpha}\sum_{j=1}^3\left(\|f_j\|_{L^2_x}+\|xf_j\|_{L^2_x}\right)\prod_{\substack{1\leq k\leq3\\ k\neq j}}\left[\|\mathcal Ff_k\|_{L^\infty_\xi}+t^{-\alpha}\left(\|f_k\|_{L^2_x}+\|xf_k\|_{L^2_x}\right)\right].
\end{aligned}
\end{equation}
\end{lem}

\begin{proof}
Set
\begin{align*}
I(t,\xi)=&\int_{\mathbb R}e^{it(y-\xi)^2}\mathcal F\left[e^{iz^2/(4t)}f_1(z)\right](y)\mathcal F\left[e^{iz^2/(4t)}f_2(z)\right](y)\overline{\mathcal F\left[e^{iz^2/(4t)}f_3(z)\right](y)}
{\rm d}y.
\end{align*}
On the one hand, by the change of variables $x=2ty$ and (\ref{freesch}) give
\begin{align*}
I(t,\xi)=&\frac{e^{it\xi^2}}{2t}
\int_{\mathbb R}e^{-ix\xi}e^{ix^2/(4t)}\mathcal F\left[e^{iz^2/(4t)}f_1(z)\right]\left(\frac{x}{2t}\right)\mathcal F\left[e^{iz^2/(4t)}f_2(z)\right]\left(\frac{x}{2t}\right)\overline{\mathcal F\left[e^{iz^2/(4t)}f_3(z)\right]\left(\frac{x}{2t}\right)}{\rm d}x\\
=&\sqrt{2t}\,e^{i\pi/4}e^{it\xi^2}\int_{\mathbb R}e^{-ix\xi}e^{it\partial_x^2}f_1(x)e^{it\partial_x^2}f_2(x)\overline{e^{it\partial_x^2}f_3(x)}{\rm d}x\\
=&2\sqrt{\pi t}\,e^{i\pi/4}e^{it\xi^2}\mathcal F\left[e^{it\partial_x^2}f_1 e^{it\partial_x^2}f_2\overline{e^{it\partial_x^2}f_3}\right](\xi).
\end{align*}
On the other hand, the identity
\begin{equation*}
\mathcal F\left[e^{-ix^2/(4t)}\right](\xi)=\sqrt{2t}e^{-i\pi/4}e^{it\xi^2}.
\end{equation*}
and the convolution formula give
\begin{align*}
I(t,\xi)=&\frac{e^{i\pi/4}}{\sqrt{2t}}\int_{\mathbb R}\mathcal F\left[e^{-ix^2/(4t)}\right](\xi-y)\mathcal F\left[e^{iz^2/(4t)}f_1(z)\right](y)\mathcal F\left[e^{iz^2/(4t)}f_2(z)\right](y)\overline{\mathcal F\left[e^{iz^2/(4t)}f_3(z)\right](y)}{\rm d}y\\
=&\sqrt{\frac{\pi}{t}}\,e^{i\pi/4}\mathcal F\Bigg[e^{-ix^2/(4t)}\mathcal F^{-1}\Bigg(\mathcal F\left[e^{iz^2/(4t)}f_1(z)\right]\mathcal F\left[e^{iz^2/(4t)}f_2(z)\right]\overline{\mathcal F\left[e^{iz^2/(4t)}f_3(z)\right]}\Bigg)\Bigg](\xi).
\end{align*}
Combining the two identities above, we obtain~\eqref{trili}.

We next prove~\eqref{freere}.
For any $f$ such that $f,xf\in L^2(\mathbb R)$, since
\begin{equation*}
\left|e^{\pm ix^2/(4t)}-1\right|
\lesssim
t^{-\alpha}|x|^{2\alpha}.
\end{equation*}
we have
\begin{equation}
\label{errorS}
\left\|
\mathcal F\left[
\left(e^{\pm ix^2/(4t)}-1\right)f
\right]
\right\|_{L^\infty_\xi}
\lesssim
t^{-\alpha}
\left(
\|f\|_{L^2_x}
+
\|xf\|_{L^2_x}
\right).
\end{equation}
Therefore, for $j=1,2,3$,
\begin{equation}
\label{eSchLin}
\left\|
\mathcal F\left[e^{ix^2/(4t)}f_j\right]
\right\|_{L^\infty_\xi}
\lesssim
\|\widehat{f_j}\|_{L^\infty_\xi}
+
t^{-\alpha}
\left(
\|f_j\|_{L^2_x}
+
\|xf_j\|_{L^2_x}
\right).
\end{equation}
Moreover, by Plancherel's theorem,
\begin{equation}
\label{triL2}
\begin{aligned}
\left\|
\mathcal F\left[e^{ix^2/(4t)}f_j\right]
\right\|_{L^2_\xi}
&=
\|f_j\|_{L^2_x},\\
\left\|
\partial_\xi
\mathcal F\left[e^{ix^2/(4t)}f_j\right]
\right\|_{L^2_\xi}
&=
\left\|
xe^{ix^2/(4t)}f_j
\right\|_{L^2_x}
=
\|xf_j\|_{L^2_x}.
\end{aligned}
\end{equation}
By~\eqref{trili}, subtracting the leading term gives
\begin{equation}
\label{dectril}
\begin{aligned}
&e^{it\xi^2}
\mathcal F\left[
e^{it\partial_x^2}f_1\,
e^{it\partial_x^2}f_2\,
\overline{e^{it\partial_x^2}f_3}
\right](\xi)
-\frac{1}{2t}
\widehat{f_1}(\xi)
\widehat{f_2}(\xi)
\overline{\widehat{f_3}}(\xi)
\\
={}&
\frac{1}{2t}
\mathcal F\Bigg[
\left(e^{-ix^2/(4t)}-1\right)
\mathcal F^{-1}\Bigg(
\mathcal F\left[e^{iz^2/(4t)}f_1(z)\right]
\mathcal F\left[e^{iz^2/(4t)}f_2(z)\right]
\overline{
\mathcal F\left[e^{iz^2/(4t)}f_3(z)\right]
}
\Bigg)
\Bigg](\xi)
\\
&+
\frac{1}{2t}
\Bigg[
\mathcal F\left[e^{iz^2/(4t)}f_1(z)\right](\xi)
\mathcal F\left[e^{iz^2/(4t)}f_2(z)\right](\xi)
\overline{
\mathcal F\left[e^{iz^2/(4t)}f_3(z)\right](\xi)
}
-
\widehat{f_1}(\xi)
\widehat{f_2}(\xi)
\overline{\widehat{f_3}}(\xi)
\Bigg].
\end{aligned}
\end{equation}
For the first term on the right-hand side of
\eqref{dectril}, applying
\eqref{errorS} and then using Plancherel's theorem,
the Leibniz rule, \eqref{eSchLin}, and
\eqref{triL2}, we obtain
\begin{align*}
&\left\|
\mathcal F\Bigg[
\left(e^{-ix^2/(4t)}-1\right)
\mathcal F^{-1}\Bigg(
\mathcal F\left[e^{iz^2/(4t)}f_1(z)\right]
\mathcal F\left[e^{iz^2/(4t)}f_2(z)\right]
\overline{
\mathcal F\left[e^{iz^2/(4t)}f_3(z)\right]
}
\Bigg)
\Bigg]
\right\|_{L^\infty_\xi}
\\
\lesssim&
t^{-\alpha}
\sum_{j=1}^3
\left(
\|f_j\|_{L^2_x}
+
\|xf_j\|_{L^2_x}
\right)
\prod_{\substack{1\leq k\leq3\\ k\neq j}}
\left[
\|\widehat{f_k}\|_{L^\infty_\xi}
+
t^{-\alpha}
\left(
\|f_k\|_{L^2_x}
+
\|xf_k\|_{L^2_x}
\right)
\right].
\end{align*}
For the second term on the right-hand side of
\eqref{dectril}, the triangle inequality,
\eqref{errorS}, and
\eqref{eSchLin} yield
\begin{align*}
&\Bigg\|
\mathcal F\left[e^{iz^2/(4t)}f_1(z)\right]
\mathcal F\left[e^{iz^2/(4t)}f_2(z)\right]
\overline{
\mathcal F\left[e^{iz^2/(4t)}f_3(z)\right]
}
-
\widehat{f_1}
\widehat{f_2}
\overline{\widehat{f_3}}
\Bigg\|_{L^\infty_\xi}
\\
\le&
\left\|
\mathcal F\left[e^{iz^2/(4t)}f_1(z)\right]
-\widehat{f_1}
\right\|_{L^\infty_\xi}
\left\|
\mathcal F\left[e^{iz^2/(4t)}f_2(z)\right]
\right\|_{L^\infty_\xi}
\left\|
\mathcal F\left[e^{iz^2/(4t)}f_3(z)\right]
\right\|_{L^\infty_\xi}
\\
&+
\|\widehat{f_1}\|_{L^\infty_\xi}
\left\|
\mathcal F\left[e^{iz^2/(4t)}f_2(z)\right]
-\widehat{f_2}
\right\|_{L^\infty_\xi}
\left\|
\mathcal F\left[e^{iz^2/(4t)}f_3(z)\right]
\right\|_{L^\infty_\xi}
\\
&+
\|\widehat{f_1}\|_{L^\infty_\xi}
\|\widehat{f_2}\|_{L^\infty_\xi}
\left\|
\mathcal F\left[e^{iz^2/(4t)}f_3(z)\right]
-\widehat{f_3}
\right\|_{L^\infty_\xi}
\\
\lesssim&
t^{-\alpha}
\sum_{j=1}^3
\left(
\|f_j\|_{L^2_x}
+
\|xf_j\|_{L^2_x}
\right)
\prod_{\substack{1\leq k\leq3\\ k\neq j}}
\left[
\|\widehat{f_k}\|_{L^\infty_\xi}
+
t^{-\alpha}
\left(
\|f_k\|_{L^2_x}
+
\|xf_k\|_{L^2_x}
\right)
\right].
\end{align*}
Combining the last two estimates with
\eqref{dectril} proves
\eqref{freere}.
\end{proof}

In Subsection~\ref{subsec:three-basic-principles}, we reduced \eqref{LLM} to the cubic nonlinear Schr\"odinger equation \eqref{Qcub1} for $\mathbf Q$. We now establish weighted estimates for its profile, defined by
$$
\mathbf q(t)=e^{itH_\eta}\mathbf Q(t).
$$

\begin{prop}
\label{Qas}
Fix an integer $N\ge 6$ and let  $\delta$ and   $\varepsilon_1$ be the constants fixed in (\ref{Nota})--(\ref{Nota2}).
If 
$$
\|\mathbf Q_0\|_{H^{N-1}}
+
\|x\mathbf Q_0\|_{H^{N-2}}
\ll \varepsilon_1,
$$
where $\mathbf Q_0=\mathbf Q(0)$, then for every $0\leq k\leq N-2$ and
$t\geq0$, 
\eqref{Qcub1} has a global solution $\mathbf Q$ satisfying 
$$
\|\xi^k\widehat{\mathbf q}(t)\|_{L^\infty_\xi\cap L^2_\xi}
\leq
\varepsilon_1,
\qquad
\|\xi^k\partial_\xi\widehat{\mathbf q}(t)\|_{L^2_\xi}
\leq
\varepsilon_1\langle t\rangle^\delta.
$$
\end{prop}

\begin{rem}
The initial assumption \eqref{inismall} implies the hypothesis of Proposition~\ref{Qas}. Indeed, by \eqref{Q-defi}, \eqref{frameX}, \eqref{Bdefini}, and \eqref{firstorder}, together with the Sobolev product estimates,
$$
\|\mathbf Q_0\|_{H^{N-1}}+\|x\mathbf Q_0\|_{H^{N-2}}
\lesssim
\|\mathbf M_0-\Omega_*\|_{H^N}+\|x(\mathbf M_0-\Omega_*)\|_{H^{N-1}}
 \ll \varepsilon_1
$$
where we also used $\mathbf M_0-\Omega_*=R_{\eta x}(m_0-\omega_*)$.
\end{rem}

 \begin{rem}
In the notation above, the results of Hayashi and Naumkin \cite{HN1998} and Kato and Pusateri \cite{KP2011} correspond to $N=2$.
Proposition~\ref{Qas} establishes the corresponding weighted estimates for arbitrary $N$.
\end{rem}

\begin{proof}

By the standard well-posedness theory for nonlinear Schr\"odinger equations in weighted Sobolev spaces \cite[Theorem~4]{HNT1988}, there exists $T_0>0$ such that
$$\mathbf Q\in C\bigl([0,T_0];H^{N-1}(\mathbb R)\bigr),\quad x\mathbf Q\in C\bigl([0,T_0];L^2(\mathbb R)\bigr).$$
Using the conserved energies for the one dimensional cubic nonlinear Schr\"odinger equation \cite[Theorem~1.1 and Corollary~1.2]{KochTataru2018}, the local solution extends globally and satisfies $\|\mathbf Q(t)\|_{H^k}\lesssim \|\mathbf Q_0\|_{H^{N-1}}$ for every $0\leq k\leq N-1$ and $t\geq0$. Hence, by Plancherel's theorem,
$$\|\xi^k\widehat{\mathbf q}(t)\|_{L^2_\xi}=\|\partial_x^k\mathbf Q(t)\|_{L^2_x}\ll \varepsilon_1, \quad 0\leq k\leq N-2,\quad t\ge0.
$$
Therefore,  we get
$$\|\xi^k\widehat{\mathbf q}(t)\|_{L^2_\xi}\le\frac12\varepsilon_1,\quad 0\leq k\leq N-2,\quad t\geq0.
$$
The continuation criterion in \cite[Theorem~4]{HNT1988} and Sobolev embedding show that the global $H^{N-1}$ bound also precludes finite time blow up in the weighted space. Thus
$$\mathbf Q\in C\bigl(\mathbb R_{\ge0};H^{N-1}(\mathbb R)\bigr),\quad x\mathbf Q\in C\bigl(\mathbb R_{\ge0};L^2(\mathbb R)\bigr).$$
Note that 
$$
i\partial_t(x\mathbf Q)-H_\eta(x\mathbf Q)=2\partial_x\mathbf Q-2|\mathbf Q|^2(x\mathbf Q),\qquad (x\mathbf Q)(0)=x\mathbf Q_0.
$$
Consider the linear Cauchy problem
$$
i\partial_t v-H_\eta v=2\partial_x\mathbf Q-2|\mathbf Q|^2v,\qquad v(0)=x\mathbf Q_0.
$$
Its Duhamel formulation is
$$
v(t)=e^{-itH_\eta}x\mathbf Q_0-i\int_0^t e^{-i(t-s)H_\eta}\left(2\partial_x\mathbf Q(s)-2|\mathbf Q(s)|^2v(s)\right){\rm d}s.
$$
Since $N\geq6$, we have
$$
\|\partial_x\mathbf Q(t)\|_{H^{N-2}}\ll \varepsilon_1,
\qquad \||\mathbf Q(t)|^2v(t)\|_{H^{N-2}}\lesssim\varepsilon_1^2\|v(t)\|_{H^{N-2}}.
$$
Gronwall's inequality yields
$$
\|v(t)\|_{H^{N-2}}\ll \varepsilon_1(1+t)e^{C\varepsilon_1^2t}.
$$
Thus, for every finite $T>0$, we have $\sup_{0\leq t\leq T}\|v(t)\|_{H^{N-2}}<\infty.$ Moreover, since $e^{-itH_\eta}$ is a strongly continuous unitary group on $H^{N-2}$, the Duhamel formula implies $v\in C\bigl([0,T];H^{N-2}(\mathbb R)\bigr)$. Applying Gronwall's inequality once again to the difference between $v$ and $x\mathbf Q$,  we conclude that $v=x\mathbf Q$, and hence
$x\mathbf Q\in C\bigl([0,T];H^{N-2}(\mathbb R)\bigr).$

Hence for every $0\leq k\leq N-2$, the maps $t\longmapsto \xi^k\widehat{\mathbf q}(t)$ and $t\longmapsto \xi^k\partial_\xi\widehat{\mathbf q}(t)$ are continuous from $\mathbb R_{\geq0}$ into $L^2_\xi$. Moreover, the one-dimensional Sobolev inequality yields $\xi^k\widehat{\mathbf q}\in C\bigl(\mathbb R_{\geq0};L^\infty_\xi\bigr).$ Note that

$$\|\xi^k\widehat{\mathbf q}(0)\|_{L^\infty_\xi}+\|\xi^k\partial_\xi\widehat{\mathbf q}(0)\|_{L^2_\xi}\lesssim\|\mathbf Q_0\|_{H^{N-1}}+\|x\mathbf Q_0\|_{H^{N-2}}\ll\varepsilon_1.$$
We can therefore define
$$
T_1
=
\sup\Bigl\{
T>0:
\|\xi^k\widehat{\mathbf q}(t)\|_{L^\infty_\xi}
\leq\varepsilon_1,\ 
\|\xi^k\partial_\xi\widehat{\mathbf q}(t)\|_{L^2_\xi}
\leq
\varepsilon_1\langle t\rangle^\delta,
\ \forall\,t\in[0,T),\
\forall\,0\leq k\leq N-2
\Bigr\}>0.
$$
We first consider the weighted
$L^2_\xi$ bound. By Duhamel's formula,
\begin{equation}
\label{qDuha}
\widehat{\mathbf q}(t,\xi)
=
\widehat{\mathbf q}(0,\xi)
+
\frac{i}{\pi}
\int_{[0,t]\times\mathbb R^3}
e^{is\Phi(\xi,a,b,c)}
\widehat{\mathbf q}(s,a)
\widehat{\mathbf q}(s,b)
\widehat{\overline{\mathbf q}}(s,c)
\delta(\xi-a-b-c)
\,{\rm d}(s,a,b,c),
\end{equation}
where
$\Phi(\xi,a,b,c)=\xi^2-a^2-b^2+c^2$.\\
Applying the commutation lemma from Section~12.2 of
Collot--Germain \cite{CG2025} to
\eqref{qDuha}, we obtain
\begin{align}
\label{weiDuh}
\xi^k\partial_\xi\widehat{\mathbf q}(t,\xi)
={}&
\xi^k\partial_\xi\widehat{\mathbf q}(0,\xi)
\nonumber\\
&+
\frac{i}{\pi}
\sum_{\substack{
k_1+k_2+k_3=k\\
k_j\geq0}}
\frac{k!}{k_1!k_2!k_3!}
\int_{[0,t]\times\mathbb R^3}
e^{is\Phi(\xi,a,b,c)}
a^{k_1}\partial_a\widehat{\mathbf q}(s,a)\,
b^{k_2}\widehat{\mathbf q}(s,b)\,
c^{k_3}\widehat{\overline{\mathbf q}}(s,c)
\nonumber\\
&\qquad\cdot
\delta(\xi-a-b-c)
{\rm d}(s,a,b,c)
+\{\text{similar and easier terms}\}.
\end{align}
For $0\leq t\leq1$, the desired estimate follows directly from
\eqref{weiDuh}, the conserved $H^{N-1}$ bound,
Sobolev embedding, and the bootstrap assumptions. Particularly, for any $0\le k\le N-2$  we have 
\begin{align}\label{wuqiong}
\|\xi^k\widehat{\mathbf q}(1)\|_{L^\infty_\xi}\le \frac{1}{2}\varepsilon_1.
\end{align}
We therefore assume $t\geq1$. Applying Plancherel's theorem, H\"older's
inequality, (\ref{staphase}), and the bootstrap
assumptions, for $1\leq s<T_1$, we obtain
\begin{align*}
&\left\|\int_{\mathbb R^3}
e^{is\Phi(\xi,a,b,c)}
a^{k_1}\partial_a\widehat{\mathbf q}(s,a)
b^{k_2}\widehat{\mathbf q}(s,b)
c^{k_3}\widehat{\overline{\mathbf q}}(s,c)
\delta(\xi-a-b-c){\rm d}(a,b,c)
\right\|_{L^2_\xi}\\
=&\left\|\int_{\mathbb{R}^3}
e^{-is\langle a\rangle_{\mu}^2}
a^{k_1}\partial_a\widehat{\mathbf q}(s,a)
e^{-is\langle b\rangle_{\mu}^2}
b^{k_2}\widehat{\mathbf q}(s,b)
e^{is\langle c\rangle_{\mu}^2}
c^{k_3}\widehat{\overline{\mathbf q}}(s,c)
\delta(\xi-a-b-c){\rm d}(a,b,c)
\right\|_{L^2_\xi}\\
=&\Bigg\|\int_{\mathbb{R}^3}
\mathcal F\left[e^{-isH_\eta}\mathcal F^{-1}\left(\xi^{k_1}\partial_\xi\widehat{\mathbf q}(s,\xi)\right)\right](a)
\mathcal F\left[e^{-isH_\eta}\mathcal F^{-1}\left(\xi^{k_2}\widehat{\mathbf q}(s,\xi)\right)\right](b)\\
&\qquad\qquad\qquad\qquad\qquad\qquad\qquad\cdot
\mathcal F\left[e^{isH_\eta}\mathcal F^{-1}\left(\xi^{k_3}\widehat{\overline{\mathbf q}}(s,\xi)\right)\right](c)
\delta(\xi-a-b-c){\rm d}(a,b,c)\Bigg\|_{L^2_\xi}\\
\lesssim&\left\|
e^{-isH_\eta}\mathcal F^{-1}\left(\xi^{k_1}\partial_\xi\widehat{\mathbf q}(s,\xi)\right)
e^{-isH_\eta}\mathcal F^{-1}\left(\xi^{k_2}\widehat{\mathbf q}(s,\xi)\right)
e^{isH_\eta}\mathcal F^{-1}\left(\xi^{k_3}\widehat{\overline{\mathbf q}}(s,\xi)\right)
\right\|_{L^2_x}\\
\lesssim&\varepsilon_1^3\langle s\rangle^{-1+\delta}.
\end{align*}
Thus, choosing $\varepsilon_1$ sufficiently small, we obtain
$$
\left\|
\xi^k\partial_\xi\widehat{\mathbf q}(t)
\right\|_{L^2_\xi}
\leq
\frac12\varepsilon_1\langle t\rangle^\delta,
\qquad
0\leq k\leq N-2,\quad 0\leq t<T_1.
$$

We next estimate
$\|\xi^k\widehat{\mathbf q}(t)\|_{L^\infty_\xi}$.
For $0\leq t\leq1$, \eqref{qDuha}, Young's inequality, and the bootstrap assumptions imply
$$
\|\xi^k\widehat{\mathbf q}(t)\|_{L^\infty_\xi}
\ll
\varepsilon_1,
\qquad
0\leq k\leq N-2.
$$
We henceforth consider $t\geq1$. For each decomposition
$k_1+k_2+k_3=k$ and $j=1,2,3$, we define
$f_j=\mathcal F^{-1}\left[\xi^{k_j}\widehat{\mathbf q}(t,\xi)\right].$
Differentiating \eqref{qDuha}, we obtain
\begin{align*}
&\xi^k\partial_t\widehat{\mathbf q}(t,\xi)\\
=&
\frac{i}{\pi}e^{it\xi^2}
\sum_{k_1+k_2+k_3=k}
\frac{k!}{k_1!k_2!k_3!}
\int_{\mathbb R^3}
e^{-ita^2}a^{k_1}\widehat{\mathbf q}(t,a)
e^{-itb^2}b^{k_2}\widehat{\mathbf q}(t,b)
e^{itc^2}c^{k_3}\widehat{\overline{\mathbf q}}(t,c)
\delta(\xi-a-b-c){\rm d}(a,b,c)\\
=&2i
\sum_{k_1+k_2+k_3=k}
\frac{k!}{k_1!k_2!k_3!}
(-1)^{k_3}
e^{it\xi^2}
\mathcal F\Big(
e^{it\partial_x^2}f_1\,
e^{it\partial_x^2}f_2\,
\overline{e^{it\partial_x^2}f_3}
\Big)(\xi).
\end{align*}
By the bootstrap assumptions and Plancherel's identity, for
$j=1,2,3$, we have
\begin{equation}
\label{boot}
\|\widehat{f_j}(t)\|_{L^\infty_\xi}
+
\|f_j(t)\|_{L^2_x}
\lesssim
\varepsilon_1,
\qquad
\|xf_j(t)\|_{L^2_x}
\lesssim
\varepsilon_1 t^\delta.
\end{equation}
Applying Lemma~\ref{tril} with the bounds \eqref{boot}, we obtain
\begin{align*}
&
\Bigg\|
\xi^k\partial_t\widehat{\mathbf q}(t,\xi)
-
\frac{i}{t}
\sum_{k_1+k_2+k_3=k}
\frac{k!}{k_1!k_2!k_3!}
(-1)^{k_3}
\widehat{f_1}(t,\xi)
\widehat{f_2}(t,\xi)
\overline{\widehat{f_3}(t,\xi)}
\Bigg\|_{L^\infty_\xi}
\\
\lesssim&
t^{-1-\alpha}
\sum_{k_1+k_2+k_3=k}
\sum_{j=1}^3
\left(
\|f_j(t)\|_{L^2_x}
+
\|xf_j(t)\|_{L^2_x}
\right)
\prod_{\substack{1\leq \ell\leq3\\ \ell\neq j}}
\left[
\|\widehat{f_\ell}(t)\|_{L^\infty_\xi}
+
t^{-\alpha}
\left(
\|f_\ell(t)\|_{L^2_x}
+
\|xf_\ell(t)\|_{L^2_x}
\right)
\right]
\\
\lesssim&
\varepsilon_1^3t^{-1-\alpha+\delta}.
\end{align*}
Note that
\begin{align*}
&
\sum_{k_1+k_2+k_3=k}
\frac{k!}{k_1!k_2!k_3!}
(-1)^{k_3}
\widehat{f_1}(t,\xi)
\widehat{f_2}(t,\xi)
\overline{\widehat{f_3}(t,\xi)}
=
\xi^k
|\widehat{\mathbf q}(t,\xi)|^2
\widehat{\mathbf q}(t,\xi).
\end{align*}
It follows that
\begin{align*}
\left\|
\xi^k
\left[
\partial_t\widehat{\mathbf q}(t,\xi)
-
\frac{i}{t}
|\widehat{\mathbf q}(t,\xi)|^2
\widehat{\mathbf q}(t,\xi)
\right]
\right\|_{L^\infty_\xi}
\lesssim
\varepsilon_1^3t^{-1-\alpha+\delta},
\quad
0\leq k\leq N-2,
\end{align*}
which combined with (\ref{wuqiong}) implies
\begin{align*}
\|\xi^k\widehat{\mathbf q}(t)\|_{L^\infty_\xi}
&\leq \|\xi^k\widehat{\mathbf q}(1)\|_{L^\infty_\xi}
+\int_1^t\left\|\xi^k\left(\partial_s\widehat{\mathbf q}(s)-\frac{i}{s}|\widehat{\mathbf q}(s)|^2\widehat{\mathbf q}(s)\right)\right\|_{L^\infty_\xi}{\rm d}s\\
&\le  \frac{1}{2}\varepsilon_1  +C\varepsilon_1^3.
\end{align*}
After choosing $\varepsilon_1$ sufficiently small, we have
$$
\|\xi^k\widehat{\mathbf q}(t)\|_{L^\infty_\xi}
\leq
\frac34\varepsilon_1,
\qquad
0\leq k\leq N-2,
\qquad
1\leq t<T_1.
$$
Suppose that $T_1<\infty$, it is possible to extend $T_1$ up to some time $T_1+\delta_1$, with $\delta_1>0$. This contradicts the definition of $T_1$, and thus proves the proposition.
\end{proof}

\section{Closing the bootstrap}
\label{sectransf}
In this section, we close the bootstrap for the perturbation and the modulation parameters. We reverse transfer the cubic NLS estimates obtained in Section~\ref{georedu}  to $z$ and its distorted Fourier profile, yielding pointwise decay, Sobolev bounds, and weighted profile estimates.  We then exploit the cancellation in the leading quadratic modulation terms to obtain integrable bounds for $\dot y$ and $\dot\phi$. These estimates establish the asymptotic stability asserted in the main theorem.

 By Theorem~\ref{gwpM},  \eqref{LLM}  has a unique global solution $\mathbf M$  in time.
Define 
\begin{equation}\label{Tstar}
T_2=
\sup\left\{T>0:
\|\mathbf M(t)-\Omega_*\|_{H^2}
\le \varepsilon_*,
|\dot y(t)|+|\dot\phi(t)|
\leq
\varepsilon_1\langle t\rangle^{-1-\delta_*}
\text{for all }0\leq t\leq T
\right\}.
\end{equation}
We first claim that $T_2>0$. 
In fact, by \eqref{inismall}, we have
\begin{equation}\label{eH2ets}
\|\mathbf M(0)-\Omega_*\|_{H^2}
\lesssim
\|m_0-\omega_*\|_{H^2}
\ll \varepsilon_1.
\end{equation}
In order to show 
\begin{equation*}
|\dot{y}(0)|+|\dot{\phi}(0)|\le\frac{1}{2}\varepsilon_1,
\end{equation*}
it suffices to dominate $\|z_0\|_{H^1(\mathbb{R})}$. Indeed, \eqref{modueq} and \eqref{Acontr} give
\begin{align*}
|\dot y(0)|+|\dot\phi(0)|\lesssim
|\dot y(0)|+|\dot\phi(0)+\eta\dot y(0)|
\lesssim|R(0)|\lesssim \|z_0\|_{H^1}^2 + \|z_0\|_{H^1}^3.
\end{align*}
The representation \eqref{zmphi}, together with \eqref{H2small} and \eqref{Tstar}, gives for $t\in[0,T_2)$ that 
 \begin{equation}\label{z0es}
\begin{aligned}
\|z_0\|_{L^2}
&\lesssim \|M(0)-\Omega_*\|_{L^2},\\
\|\partial_x z_0\|_{L^2}
&\lesssim \|M(0)-\Omega_*\|_{H^1}
+\|M(0)-\Omega_*\|_{H^1}^2.
\end{aligned}
\end{equation}
The initial estimate \eqref{z0es} then yields
\begin{equation}\label{zoH1}
\|z_0\|_{H^1}\lesssim \|M(0)-\Omega_*\|_{H^2}+\|M(0)-\Omega_*\|_{H^2}^2.
\end{equation}
Applying \eqref{H2small} once more, we obtain
\begin{equation*}
\|z_0\|_{H^1}\lesssim \|\mathbf{M}(0)-\Omega_*\|_{H^2}+ \|\mathbf{M}(0)-\Omega_*\|^2_{H^2}\ll \varepsilon_1,
\end{equation*}
provided that $\varepsilon_1$ is sufficiently small. Then the claim follows from the time regularity of $\mathbf M$ and $y,\phi$.

\begin{rem}
\label{zH^2}
The argument used for \eqref{zoH1}, together with \eqref{zmphi}, \eqref{H2small}, and \eqref{Tstar}, gives, for $t\in[0,T_2)$,
\begin{align*}
\|z(t)\|_{H^2}
\lesssim&\ \|M(t)-\Omega_*\|_{H^2}+\|M(t)-\Omega_*\|_{H^2}^2+\|M(t)-\Omega_*\|_{H^2}^3\\
\lesssim&\ \|M(t)-\Omega_*\|_{H^2}\\
\lesssim&\ \varepsilon_*.
\end{align*}
\end{rem}

Combining \eqref{parasize}, \eqref{eH2ets}, and \eqref{Tstar}, we obtain
\begin{equation}\label{parboot}
\begin{aligned}
|y(t)|+|\phi(t)|
&\leq
|y(0)|+|\phi(0)|
+\int_0^t\bigl(|\dot y(s)|+|\dot\phi(s)|\bigr)\,{\rm d}s\\
&\lesssim
 \varepsilon_1
+\varepsilon_1\int_0^\infty\langle s\rangle^{-1-\delta_*}\,{\rm d}s\\
&\lesssim\varepsilon_1.
\end{aligned}
\end{equation}
Since
\begin{equation}
\label{eqphi}
\widehat q(t,\xi)
=
e^{i(\phi(t)+\eta y(t))}e^{iy(t)\xi}
\widehat{\mathbf q}(t,\xi),
\end{equation}
Proposition~\ref{Qcopro} allows us to transfer the weighted estimates for $\mathbf{q}$ to $q$, where $q(t)=e^{itH_\eta}Q(t)$.

\begin{prop}
\label{prop:Q-yphi-profile}
For all $0\leq k\le N-2$ and $t\in [0,T_2)$, we have
\begin{equation}
\label{L2infy}
\|\xi^k\widehat q(t)\|_{L^2_\xi\cap L^\infty_\xi}
\leq
\varepsilon_1,\quad
\|\xi^k\partial_\xi\widehat q(t)\|_{L^2_\xi}
\lesssim
\varepsilon_1\langle t\rangle^\delta.
\end{equation}
\end{prop}

\begin{proof}
By \eqref{eqphi} and Proposition~\ref{Qas}, we obtain
\begin{equation*}
\|\xi^k\widehat q(t)\|_{L^2_\xi\cap L^\infty_\xi}
=
\|\xi^k\widehat{\mathbf q}(t)\|_{L^2_\xi\cap L^\infty_\xi}
\leq
\varepsilon_1.
\end{equation*}
Differentiating \eqref{eqphi} in $\xi$, we have
\begin{equation*}
\partial_\xi\widehat q
=
e^{i(\phi+\eta y)}e^{iy\xi}
\left(
\partial_\xi\widehat{\mathbf q}
+
iy\widehat{\mathbf q}
\right).
\end{equation*}
Hence, by \eqref{parboot} and
Proposition~\ref{Qas}, we obtain
\begin{align*}
\|\xi^k\partial_\xi\widehat q(t)\|_{L^2_\xi}
&\leq
\|\xi^k\partial_\xi\widehat{\mathbf q}(t)\|_{L^2_\xi}
+|y(t)|\|\xi^k\widehat{\mathbf q}(t)\|_{L^2_\xi} 
\lesssim
\varepsilon_1\langle t\rangle^\delta.
\end{align*}
The proof is now completed.\end{proof}

\begin{rem}
\label{HkQ}
Using Proposition~\ref{prop:Q-yphi-profile} and (\ref{staphase}), for all $0\le k\le N-2$, $t\in[0,T_2)$, we have
\begin{equation}\label{QHk}
\|\partial_x^kQ(t)\|_{L^\infty}\lesssim \varepsilon_1\langle t\rangle^{-1/2},
\qquad
\|Q\|_{H^k}\lesssim\varepsilon_1.
\end{equation}
\end{rem}

\subsection{Reverse transfer of weighted estimates}
This subsection is devoted to weighted estimates for $z$ and its profile. The relation between $Q$ and $z$ given by \eqref{xchi}--\eqref{Q-z} is nonlinear and cannot be inverted globally. Although \eqref{Q-z} yields $$D_\eta^*z=-2e^{-i\chi}Q+o_{L^2}(1),$$
the scattering asymptotics of $Q$ do not directly imply those of $z$, since the current estimates do not determine the asymptotic behavior of the spatially dependent gauge factor $e^{-i\chi(t,x)}$ on the dispersive scale. Lemma~\ref{Detapro},   the orthogonality condition \eqref{zorthog} and (\ref{Q-z}) give the identity
\begin{equation}\label{loczinv}
z=-2D_\eta H_\eta^{-1}\left[e^{-i\chi}\left(1+\frac14|z|^2\right)Q\right]-\frac\mu2D_\eta H_\eta^{-1}\left(S_\eta z^2\right).
\end{equation}
\begin{prop}\label{z-decay}
For all $0\le j\leq N-1$, we have
\begin{equation}\label{zinfdecay}
\|\partial_x^j z(t)\|_{L^\infty}\lesssim\varepsilon_1\langle t\rangle^{-1/2},\qquad \forall \mbox{ } 0\leq t<T_2.
\end{equation}
\end{prop}

\begin{proof}
We prove the proposition by induction on $j$, starting with the cases $j=0,1$.

By  \eqref{DeHe}, Remark~\ref{zH^2} and \eqref{QHk}, we obtain from the continuity argument that 
\begin{equation}\label{11}
\|z(t)\|_{L^\infty}\lesssim\varepsilon_1\langle t\rangle^{-1/2},\qquad 0\leq t<T_2.
\end{equation}
The representation \eqref{Q-z} gives
\begin{equation}\label{Qzz}
\partial_xz=2e^{-i\chi}\left(1+\frac14|z|^2\right)Q-\mu T_\eta z+\frac{\mu S_\eta}{2}z^2.
\end{equation}
Using Remark~\ref{zH^2}, \eqref{QHk} and \eqref{11}, we obtain
\begin{equation*}
\|\partial_xz(t)\|_{L^\infty}\lesssim \varepsilon_1\langle t\rangle^{-1/2}.
\end{equation*}

Now let $1\leq m\leq N-2$ and assume that
\begin{equation*}
\|\partial_x^rz(t)\|_{L^\infty}\lesssim_N\mu^{-1}\varepsilon_1\langle t\rangle^{-1/2},\quad 0\leq r\leq m,\quad 0\leq t<T_2.
\end{equation*}
It remains to prove the case $r=m+1$. Differentiating \eqref{xchi} in $x$, we have
\begin{equation}
\label{xchi2}
\partial_x\chi=-\frac{\frac12\operatorname{Im}(\overline{z}\partial_xz)+\mu S_\eta\operatorname{Im}z}{1+\frac14|z|^2}.
\end{equation}
It follows that
\begin{equation}
\label{chide3}
\|\partial_x^r\chi(t)\|_{L^\infty}\lesssim_N\mu^{-1}\varepsilon_1\langle t\rangle^{-1/2},\qquad 1\leq r\leq m.
\end{equation}
Thus
\begin{equation}\label{Hasi}
\left\|e^{-i\chi}\left(1+\frac14|z|^2\right)\right\|_{W^{m,\infty}}\lesssim1.
\end{equation}
Then the estimate \eqref{Hasi}, Remark~\ref{HkQ}, (\ref{Qzz}) and the induction hypothesis complete the induction.
\end{proof}
In particular, \eqref{u-z} and \eqref{zinfdecay} yield \eqref{Lindeca}.
The same induction argument also gives the following Sobolev estimates.

\begin{prop}
\label{Sobes}
Let $1\leq m\leq N$ be an integer. Then
\begin{equation}\label{Hmze}
\|z(t)\|_{H^m}\lesssim\varepsilon_1,\quad 0\leq t<T_2.
\end{equation}
\end{prop}
\begin{proof}
We argue by induction on $m$. The case $m=1$ follows from Lemma~\ref{Deliest} and \eqref{loczinv}. Now let $1\leq m\leq N-1$ and assume that
$$
\|z(t)\|_{H^m}\lesssim\varepsilon_1,\quad 0\leq t<T_2.
$$
For the last step $m=N-1$, \eqref{Qmath} and the uniform $H^{N-1}$ bound for $\mathbf Q$ in the proof of Proposition~\ref{Qas} give $\|Q(t)\|_{H^{N-1}}\lesssim\varepsilon_1$.
Then Lemma~\ref{Deliest}, \eqref{loczinv}, and \eqref{Hasi} complete the induction.
\end{proof}
 We next estimate $\|\xi^k\widetilde{g}\|_{L^2\cap L^\infty}$ and $\|\xi^k\partial_\xi\widetilde{g}\|_{L^2}$.
\begin{prop} 
We have uniformly for $t\in [0,T_2)$ and $0\le k\le N-2$  that 
\begin{align}
\|\xi^k \widetilde{g}\|_{L^2}&\lesssim \varepsilon_1\label{a1}\\
\|\xi^k \partial_{\xi}\widetilde g\|_{L^2}&\lesssim \varepsilon_1 \langle t\rangle^{\delta}\label{a2}\\
\|\xi^k \widetilde{g}\|_{L^{\infty}}&\lesssim \varepsilon_1.\label{a3}
\end{align}
 \end{prop}
\begin{proof}
For the $L^2$ estimate, set $w=D_\eta^*z$. The definition \eqref{gpro2} and Proposition~\ref{disfour} give
\begin{equation}\label{tig}
\widetilde g(t,\xi)=\frac{e^{it(\xi^2+\mu^2)}\widehat{w}(t,\xi)}{\sqrt{\xi^2+\mu^2}}.
\end{equation}
Using \eqref{zinfdecay} and \eqref{Hmze}, for all $0\leq j\leq N-2$, one has 
\begin{equation}\label{Linw}
\|\partial_x^jw(t)\|_{L^\infty}\lesssim\varepsilon_1\langle t\rangle^{-1/2},\qquad \|w(t)\|_{H^{N-1}}\lesssim \|z(t)\|_{H^N}\lesssim\varepsilon_1.
\end{equation}
Therefore, \eqref{tig}, \eqref{Linw}, and Plancherel's identity yield, for $0\leq k\leq N-2$ and $0\leq t<T_2$,
\begin{equation}
\begin{aligned}
\|\xi^k\widetilde g(t)\|_{L^2_\xi}=&\left\|\frac{\xi^k}{\sqrt{\xi^2+\mu^2}}e^{it(\xi^2+\mu^2)}\widehat{w}(t,\xi)\right\|_{L^2_\xi}\\
\lesssim&\|\xi^k\widehat{w}(t)\|_{L^2_\xi}=\|\partial_x^kw(t)\|_{L^2_x}\\
\lesssim&\varepsilon_1,
\end{aligned}
\label{33}
\end{equation}
from which (\ref{a1}) follows. 

We next control $\|\xi^k\partial_\xi\widetilde{g}\|_{L^2}$. Differentiating \eqref{tig} with respect to $\xi$, we obtain
\begin{equation*}
\partial_\xi\widetilde g(t,\xi)=-i\frac{e^{it(\xi^2+\mu^2)}\widehat{Jw}(t,\xi)}{\sqrt{\xi^2+\mu^2}}-\frac{\xi e^{it(\xi^2+\mu^2)}\widehat{w}(t,\xi)}{(\xi^2+\mu^2)^{3/2}},
\end{equation*}
where $J$ is the Galilean vector field defined by $J=x+2it\partial_x$.\\
By  \eqref{Linw} and
\begin{align}
\|\xi^k\partial_\xi\widetilde g(t)\|_{L^2_\xi}\lesssim&\|\xi^k\widehat{Jw}(t)\|_{L^2_\xi}+\|\xi^k\widehat{w}(t)\|_{L^2_\xi},
\label{L2xiw}
\end{align}
it suffices to control $\|\xi^k\widehat{Jw}(t)\|_{L^2_\xi}$. Since \eqref{Q-z} is equivalent to
\begin{equation*}
w=-2e^{-i\chi}\left(1+\frac14|z|^2\right)Q-\frac\mu2S_\eta z^2.
\end{equation*}
Note that
\begin{align*}
J\left[e^{-i\chi}\left(1+\frac14|z|^2\right)Q\right]=&e^{-i\chi}\left(1+\frac14|z|^2\right)JQ+2it\partial_x\left[e^{-i\chi}\left(1+\frac14|z|^2\right)\right]Q,\\
J(S_\eta z^2)=&xS_\eta z^2+2it(\partial_xS_\eta)z^2+4itS_\eta z\partial_xz.
\end{align*}
Therefore
\begin{align}
Jw=&-2e^{-i\chi}\left(1+\frac14|z|^2\right)JQ-4it\partial_x\left[e^{-i\chi}\left(1+\frac14|z|^2\right)\right]Q\notag\\
-&\frac\mu2xS_\eta z^2-i\mu t(\partial_xS_\eta)z^2-2i\mu tS_\eta z\partial_xz.
\label{iw2}
\end{align}
By definition of the profile of $Q$ and the Galilean vector field $J$, we get the identity
\begin{align}
\widehat{JQ}(t,\xi)=&ie^{-it(\xi^2+\mu^2)}\partial_\xi\widehat q(t,\xi).
\label{jq12}
\end{align}
With the help of Plancherel's identity and \eqref{jq12}, for every $0\leq j\leq N-2$, we have
\begin{align*}
\|\partial_x^jJQ(t)\|_{L^2_x}=&\|\xi^j\widehat{JQ}(t)\|_{L^2_\xi}\\
=&\|\xi^j\partial_\xi\widehat q(t)\|_{L^2_\xi}\lesssim \varepsilon_1\langle t\rangle^\delta.
\end{align*}
and hence
\begin{equation}
\label{JQ2a}
\|JQ(t)\|_{H^k}\lesssim_N \varepsilon_1\langle t\rangle^\delta,
\quad 0\leq k\leq N-2,\quad 0\leq t<T_2.
\end{equation}
Therefore, the first term of \eqref{iw2} is controlled by \eqref{Hasi} and \eqref{JQ2a}:
\begin{equation}
\left\|e^{-i\chi}\left(1+\frac14|z|^2\right)JQ(t)\right\|_{H^k}\lesssim_N \|JQ(t)\|_{H^k}\lesssim_N \varepsilon_1\langle t\rangle^\delta.
\label{JWa1}
\end{equation}
Using \eqref{xchi2}, we obtain 
\begin{align}
&\partial_x\left[e^{-i\chi}\left(1+\frac14|z|^2\right)\right]\notag\\
=&e^{-i\chi}\left[-i(\partial_x\chi)\left(1+\frac14|z|^2\right)+\frac14\left(\partial_xz\,\overline{z}+z\,\partial_x\overline{z}\right)\right]\notag\\
=&e^{-i\chi}\left[\frac12\overline{z}\,\partial_xz+i\mu S_\eta\operatorname{Im}z\right],
\label{Hasi1q}
\end{align}
By \eqref{chide3} and chain rule, for every $0\leq k\leq N-2$, we get
\begin{equation*}
\|e^{-i\chi(t)}\|_{W^{k,\infty}}\lesssim 1,\quad 0\leq t<T_2.
\end{equation*}
By  (\ref{zinfdecay}) , \eqref{Hmze} and the H\"older inequality imply  that
\begin{align*}
\|S_\eta\operatorname{Im}z(t)\|_{H^k}
&\lesssim \sum_{\ell=0}^k\sum_{r=0}^{\ell}\|\partial_x^rS_\eta\|_{L^2}\|\partial_x^{\ell-r}z(t)\|_{L^\infty}
\lesssim \varepsilon_1\langle t\rangle^{-1/2},\\
\|\overline{z}\,\partial_xz(t)\|_{H^k}
&\lesssim \sum_{\ell=0}^k\sum_{r=0}^{\ell}\|\partial_x^rz(t)\|_{L^\infty}\|\partial_x^{\ell-r+1}z(t)\|_{L^2}
\lesssim \varepsilon_1^2\langle t\rangle^{-1/2},
\end{align*}
which together with \eqref{Hasi1q} therefore yields
\begin{align*}
\left\|\partial_x\left[e^{-i\chi}\left(1+\frac14|z|^2\right)\right]\right\|_{H^k}
\lesssim\varepsilon_1\langle t\rangle^{-1/2}.
\end{align*}
Using Remark~\ref{HkQ},   the second term in \eqref{iw2} is dominated as 
\begin{align}
t\left\|\partial_x\left[e^{-i\chi}\left(1+\frac14|z|^2\right)\right]Q(t)\right\|_{H^k}\
\lesssim\varepsilon_1^2.
\label{jw3z}
\end{align}
By (\ref{zinfdecay})  and \eqref{Hmze}, for all $0\le k\le N-2$, we have
\begin{align*}
\|xS_\eta z^2(t)\|_{H^k}
&\lesssim\varepsilon_1^2\langle t\rangle^{-1},\qquad
t\|(\partial_xS_\eta)z^2(t)\|_{H^k}
\lesssim\varepsilon_1^2,\qquad
t\|S_\eta z\partial_xz(t)\|_{H^k}
\lesssim\varepsilon_1^2.
\end{align*}
Combining \eqref{JWa1} and \eqref{jw3z}, we obtain from (\ref{iw2}) that 
\begin{equation}
\label{jws1}
\|Jw(t)\|_{H^k}\lesssim\varepsilon_1\langle t\rangle^\delta,
\quad 0\leq k\leq N-2,\quad 0\leq t<T_2.
\end{equation}
Therefore, (\ref{a2}) follows by (\ref{jws1}),  (\ref{Linw})  and (\ref{L2xiw}).

Finally, we use Corollary~\ref{restaph} and \eqref{tig} to
control $\|\xi^k\widetilde{g}\|_{L^\infty}$ for all $0\leq k\leq N-2$.\\
For $0\leq t\leq1$, by \eqref{33}, \eqref{a2} and the one-dimensional Sobolev
inequality, we get
\begin{align*}
\|\xi^k\widetilde g(t)\|_{L^\infty_\xi}\lesssim\|\xi^k\widetilde g(t)\|_{L^2_\xi}+\|\partial_\xi(\xi^k\widetilde g)(t)\|_{L^2_\xi}
\lesssim\varepsilon_1.
\end{align*}
For $1\leq t<T_2$,  Corollary~\ref{restaph} shows 
\begin{align*}
\|\xi^k\widetilde{g}\|_{L^\infty}\lesssim&\|\xi^k e^{it(\xi^2+\mu^2)}\widehat w\|_{L^\infty_\xi}=\|\mathcal{F}(e^{itH_\eta}i^{-k}\partial_x^kw)\|_{L^\infty}\\
\lesssim&t^{1/2}\|\partial_x^kw\|_{L^\infty_x}+t^{-1/4}\left\|\partial_\xi\left(\xi^k e^{it(\xi^2+\mu^2)}\widehat w(t,\xi)\right)\right\|_{L^2_\xi}.
\end{align*}
By the definition of $J$, we get
$$
i\partial_\xi\left[e^{it(\xi^2+\mu^2)}\widehat w(t,\xi)\right]=e^{it(\xi^2+\mu^2)}\widehat{Jw}(t,\xi).
$$
Using \eqref{Linw} and \eqref{jws1}, one obtains 
\begin{align*}
\left\|\partial_\xi\left[\xi^k e^{it(\xi^2+\mu^2)}\widehat w(t,\xi)\right]\right\|_{L^2_\xi}\lesssim&\|w(t)\|_{H^k}+\|\xi^k\widehat{Jw}(t)\|_{L^2_\xi}\notag\\
\lesssim&\varepsilon_1\langle t\rangle^\delta,
\end{align*}
which, together with \eqref{Linw},  implies
\begin{equation}
\label{gLlt}
\|\xi^k\widetilde g(t)\|_{L^\infty_\xi}
\lesssim\varepsilon_1,
\quad
1\leq t<T_2.
\end{equation}
Thus (\ref{a3}) is done. 
\end{proof}

\subsection{Control of the modulation parameters}
The aim of this subsection is to prove \eqref{asphiy}.
In order to control the modulation parameters $y$ and $\phi$ through
\eqref{modueq}, we first isolate the leading quadratic part
of $R$. Define
\begin{align*}
R_2(w)=\int_{\mathbb R}S_\eta\mathcal N_2(w,\overline w){\rm d}x.
\end{align*}
The next two propositions identify the structure of $R_2$ and justify replacing $z$ by $z^*$.
\begin{lem}\label{zstarcha}
Let $z^{*}=-2D_\eta H_\eta^{-1}Q$. Then
\begin{equation}\label{r2b}
R_2(z^{*})=-2\mu\int_{\mathbb R}S_\eta^2\overline{z^{*}}Q{\rm d}x.
\end{equation}
Moreover, for every $\theta\in\mathbb R$ and $w\in H^1(\mathbb R)$, $R_2(e^{-i\theta}w)=R_2(w).$
\end{lem}

\begin{proof}
Define
$$\mathcal I(w)=-\mu^2\int_{\mathbb R}S_\eta^2T_\eta w^2\,\mathrm dx+\mu\int_{\mathbb R}S_\eta^2w\partial_xw\,\mathrm dx.$$
Integrating by parts implies $\mathcal{I}(w)=0.$ 
Using $ Q=\frac12\left(\partial_x z^{*}+\mu T_\eta z^{*}\right) $ and
$$
R_2(z^{*})=R_2(z^{*})-\mathcal I(z^{*})=-\mu\int_{\mathbb R}S_\eta^2\overline{z^{*}}\left(\partial_xz^{*}+\mu T_\eta z^{*}\right)\,\mathrm dx.
$$
we get \eqref{r2b}.
Note that
\begin{align*}
R_2(e^{-i\theta}w)-\mathcal I(e^{-i\theta}w)=&-\mu^2\int_{\mathbb R}S_\eta^2T_\eta|w|^2{\rm d}x-\mu\int_{\mathbb R}S_\eta^2\overline w\partial_xw{\rm d}x.\\
R_2(w)-\mathcal I(w)=&-\mu^2\int_{\mathbb R}S_\eta^2T_\eta|w|^2{\rm d}x-\mu\int_{\mathbb R}S_\eta^2\overline w\partial_xw{\rm d}x.
\end{align*}
Therefore
$$
R_2(e^{-i\theta}w)-\mathcal I(e^{-i\theta}w)=R_2(w)-\mathcal I(w),
$$
and the conclusion follows from $\mathcal I(e^{-i\theta}w)=\mathcal I(w)=0$.
\end{proof}

\begin{prop}
\label{Rz2p}
For $0\leq t<T_2$, we have
\begin{equation}\label{R2q1}
|R_2(z)-R_2(z^{*})|\lesssim\varepsilon_1^3\langle t\rangle^{-3/2}.
\end{equation}
\end{prop}

\begin{proof}
Note that by \eqref{loczinv} and the definition of $z^{*}$ in Lemma~\ref{zstarcha}, 
\begin{equation}
\begin{aligned}
z-e^{-i\chi(t,0)}z^{*}=D_\eta H_\eta^{-1}\Bigg\{2\left[e^{-i\chi(t,0)}-e^{-i\chi(t,x)}\left(1+\frac14|z|^2\right)\right]Q-\frac{\mu S_\eta}{2}z^2\Bigg\},
\end{aligned}
\label{0inver}
\end{equation}
Using  (\ref{chide3}), \eqref{QHk}, Proposition~\ref{z-decay}, and the fundamental theorem of calculus, we obtain
\begin{align*}
&S_\eta^{1/2}\left|2\left[e^{-i\chi(t,0)}-e^{-i\chi}\left(1+\frac14|z|^2\right)\right]Q-\frac{\mu S_\eta}{2}z^2\right|\\
\lesssim&S_\eta^{1/2}|\chi(t,x)-\chi(t,0)||Q|+S_\eta^{1/2}|z|^2|Q|+S_\eta^{3/2}|z|^2\\
\lesssim&\varepsilon_1^2\langle t\rangle^{-1}.
\end{align*}
Therefore, applying Lemma~\ref{weiDeta} to \eqref{0inver} gives
\begin{equation}\label{Setaz}
\left\|S_\eta\left(z-e^{-i\chi(t,0)}z^{*}\right)\right\|_{L^2}+\left\|S_\eta\partial_x\left(z-e^{- i\chi(t,0)}z^{*}\right)\right\|_{L^2}\lesssim\varepsilon_1^2\langle t\rangle^{-1}.
\end{equation}
By \eqref{zinfdecay}, \eqref{DeHe}, and \eqref{QHk}, we have
\begin{equation}
\|z(t)\|_{W^{1,\infty}}+\|z^{*}(t)\|_{W^{1,\infty}}\lesssim\varepsilon_1\langle t\rangle^{-1/2}.
\label{zs0}
\end{equation}
By Lemma~\ref{zstarcha}, for all $0\leq t<T_2$,
\begin{align*}
R_2(z)-R_2(z^{*})=&R_2(z)-R_2\left(e^{-i\chi(t,0)}z^{*}\right)\\
=&-2\mu^2\int_{\mathbb R}S_\eta^2T_\eta\Bigg\{\left(z-e^{-i\chi(t,0)}z^{*}\right)\operatorname{Re}z+e^{-i\chi(t,0)}z^{*}\operatorname{Re}\left(z-e^{-i\chi(t,0)}z^{*}\right)\Bigg\}{\rm d}x\\
+&2i\mu\int_{\mathbb R}S_\eta^2\Bigg\{\operatorname{Im}\left(z-e^{-i\chi(t,0)}z^{*}\right)\partial_xz+\operatorname{Im}\left(e^{-i\chi(t,0)}z^{*}\right)\partial_x\left(z-e^{-i\chi(t,0)}z^{*}\right)\Bigg\}{\rm d}x.
\end{align*}
Hence, \eqref{Setaz}, \eqref{zs0} and the H\"older inequality implies that
\begin{align*}
|R_2(z)-R_2(z^{*})|\lesssim&\left(\|z\|_{W^{1,\infty}}+\|z^{*}\|_{W^{1,\infty}}\right)\left(\left\|S_\eta\left(z-e^{-i\chi(t,0)}z^{*}\right)\right\|_{L^2}+\left\|S_\eta\partial_x\left(z-e^{-i\chi(t,0)}z^{*}\right)\right\|_{L^2}\right)\\
\lesssim&\varepsilon_1^3\langle t\rangle^{-3/2}.
\end{align*}
This proves \eqref{R2q1}.
\end{proof}

\begin{prop}
\label{pmodre}
We have uniformly for  $0\leq t<T_2$ that 
\begin{align*}
\left|\dot y-\frac12\operatorname{Im}R_2(z)\right|+\left|\dot\phi+\frac\mu2\operatorname{Re}R_2(z)+\frac\eta2\operatorname{Im}R_2(z)\right|\lesssim\varepsilon_1^3\langle t\rangle^{-3/2}.
\end{align*}
\end{prop}

\begin{proof}
A direct computation gives
\begin{align*}
\dot y-\frac12\operatorname{Im}R_2(z)=&\frac{\frac2\mu\operatorname{Im}R+\operatorname{Im}(\overline A R)-\left(\frac2\mu-\frac\mu2|A|^2\right)\operatorname{Im}R_2(z)}{\frac4\mu-\mu|A|^2}\\
=&\frac{\frac2\mu\operatorname{Im}\bigl(R-R_2(z)\bigr)+\operatorname{Im}(\overline A R)+\frac\mu2|A|^2\operatorname{Im}R_2(z)}{\frac4\mu-\mu|A|^2}.
\end{align*}
Then, using Lemma~\ref{dysmallA}, we arrive at
$$
\left|\dot y-\frac12\operatorname{Im}R_2(z)\right|\lesssim\left|R-R_2\right|+\left|A\right|\left|R\right|+|A|^2|R_2|.
$$
Similarly,
$$
\left|\dot\phi+\frac\mu2\operatorname{Re}R_2(z)+\frac\eta2\operatorname{Im}R_2(z)\right|\lesssim\left|R-R_2\right|+\left|A\right|\left|R\right|+|A|^2|R_2|.
$$
Since
\begin{equation}
\label{Rga}
R(z)-R_2(z)=\int_{\mathbb R}\frac{S_\eta}{1+\frac14|z|^2}\left[\mathcal N_3(z,z,\overline{z})-\frac14|z|^2\mathcal N_2(z,\overline{z})\right]{\rm d}x,
\end{equation}
\eqref{zinfdecay} and \eqref{Rga} then give 
\begin{equation*}
|A|\lesssim\mu^{-2}\varepsilon_1\langle t\rangle^{-1/2},\quad |R_2|\lesssim\varepsilon_1^2\langle t\rangle^{-1},\quad |R-R_2|\lesssim\varepsilon_1^3\langle t\rangle^{-3/2}+\varepsilon_1^4\langle t\rangle^{-2}.
\end{equation*}
The proof is done.
\end{proof}

\begin{prop}\label{R2main}
For $1\leq t<T_2$, let $\zeta=\frac{x}{2t}$. Then
\begin{align}\label{R2w9}
\left|\int_{\mathbb R}S_\eta^2(x)\frac{\zeta}{\zeta^2+\mu^2}|\widehat q(t,\zeta)|^2\,\mathrm dx\right|
&\lesssim\varepsilon_1^2t^{-1},\\
\label{R2o8}
\left|\int_{\mathbb R}S_\eta^2(x)T_\eta(x)\frac{1}{\zeta^2+\mu^2}|\widehat q(t,\zeta)|^2\,\mathrm dx\right|
&\lesssim\varepsilon_1^2t^{-1/2+\delta}.
\end{align}
\end{prop}

\begin{proof}
Estimate \eqref{R2w9} follows directly from \eqref{L2infy}. It remains to prove \eqref{R2o8}. Note that
$$
|\widehat q(t,\zeta)-\widehat q(t,0)| \le |\zeta|^{1/2}\|\partial_\xi\widehat q(t)\|_{L^2_\xi}.
$$
The oddness of $S_\eta^2 T_\eta$ and \eqref{L2infy} give
\begin{align*}
&\left|\int_{\mathbb R}S_\eta^2(x)T_\eta(x)\frac{1}{\zeta^2+\mu^2}|\widehat q(t,\zeta)|^2{\rm d}x\right|\\
=&\left|\int_{\mathbb R}S_\eta^2(x)T_\eta(x)\left[\frac{|\widehat q(t,\zeta)|^2}{\zeta^2+\mu^2}-\frac{|\widehat q(t,0)|^2}{\mu^2}\right] {\rm d}x\right|\\
\lesssim&\int_{\mathbb R}S_\eta^2\left(\frac{\zeta^2}{\zeta^2+\mu^2}|\widehat q(t,\zeta)|^2+\big||\widehat q(t,\zeta)|^2-|\widehat q(t,0)|^2\big|\right){\rm d}x\\
\lesssim&\ \varepsilon_1^2t^{-1/2+\delta}.
\end{align*}
\end{proof}

\begin{prop}
We have
\begin{equation}\label{paraimp}
|\dot y(t)|+|\dot\phi(t)|\leq\frac12\varepsilon_1\langle t\rangle^{-1-\delta_*},\quad \forall \mbox{ } 0\le t\le T_2.
\end{equation}
\end{prop}

\begin{proof}
In order to prove \eqref{paraimp}, by Proposition~\ref{pmodre}, Proposition~\ref{Rz2p}, and the definition of $R_2(z)$, it suffices  to control $R_2(z^*)$ on $t\geq1$. Letting $\zeta=\frac{x}{2t}$, by Lemma~\ref{Stapha} and \eqref{L2infy}, we have
\begin{equation}\label{Qsta1}
Q(t,x)=\frac{e^{-i(\mu^2t+\pi/4)}e^{ix^2/(4t)}}{\sqrt{2t}}\widehat q(t,\zeta)+r_1(t,x),
\end{equation}
where
$$
\|r_1(t)\|_{L^\infty_x}\lesssim t^{-3/4}\|\partial_\xi\widehat q(t)\|_{L^2_\xi}\lesssim\varepsilon_1t^{-3/4+\delta}.
$$
Similarly, since
$$
z^{*}=-2D_\eta H_\eta^{-1}Q=\frac{e^{-i\mu^2t}}{\sqrt{2\pi}}\int_{\mathbb R}e^{ix\xi-it\xi^2}\frac{2(\mu T_\eta(x)-i\xi)}{\xi^2+\mu^2}\widehat q(t,\xi){\rm d}\xi,
$$
we get
\begin{equation}\label{zsta7}
z^{*}(t,x)=\frac{e^{-i(\mu^2t+\pi/4)}e^{ix^2/(4t)}}{\sqrt{2t}}\frac{2(\mu T_\eta(x)-i\zeta)}{\zeta^2+\mu^2}\widehat q(t,\zeta)+r_2(t,x),
\end{equation}
where
\begin{align*}
\|r_2(t)\|_{L^\infty_x}\lesssim&t^{-3/4}\sup_{x\in\mathbb R}\left\|\partial_\xi\left(\frac{\mu T_\eta(x)-i\xi}{\xi^2+\mu^2}\widehat q(t,\xi)\right)\right\|_{L^2_\xi}\\
\lesssim&t^{-3/4}\left(\|\widehat q(t)\|_{L^2_\xi}+\|\partial_\xi\widehat q(t)\|_{L^2_\xi}\right)\\
\lesssim&\varepsilon_1t^{-3/4+\delta}.
\end{align*}
Substituting \eqref{Qsta1} and \eqref{zsta7} into \eqref{r2b} implies
\begin{equation*}
R_2(z^{*})=\frac{2\mu}{t}\int_{\mathbb R}S_\eta^2(x)\frac{-i\zeta-\mu T_\eta(x)}{\zeta^2+\mu^2}|\widehat q(t,\zeta)|^2{\rm d}x+O\left(\varepsilon_1^2\left(t^{-5/4+\delta}+t^{-3/2+2\delta}\right)\right).
\end{equation*}
Proposition~\ref{R2main} therefore gives
$$
|R_2(z^{*})(t)|\lesssim\varepsilon_1^2\left(t^{-5/4+\delta}+t^{-3/2+2\delta}\right)\lesssim\varepsilon_1^2t^{-1-\delta_*}.
$$
Choosing  $\varepsilon_1$ sufficiently small, \eqref{paraimp} is proved.
\end{proof}

{\bf Proof of (\ref{asphiy}).}
Proposition~\ref{Sobes} and \eqref{msre} imply that 
$$\|M(t)-\Omega_*\|_{H^2}\lesssim\varepsilon_1$$
uniformly in $t\in [0,T_2)$.
The definition \eqref{phidef}, the estimate \eqref{wallrota}, and \eqref{parboot} then give
\begin{align}
\|\mathbf M(t)-\Omega_*\|_{H^2}
&\leq \|M(t)-\Omega_*\|_{H^2}+|y(t)|\|\partial_x\Omega_*\|_{H^2}+2\left|\sin\frac{\phi(t)+\eta y(t)}{2}\right|\|S_\eta\|_{H^2}\notag\\
&\lesssim \varepsilon_1,\quad 0\leq t\leq T_2.
\label{M-O1}
\end{align}
Suppose that $T_2<\infty$. The bounds \eqref{M-O1} and \eqref{paraimp}, the time continuity of $\mathbf M-\Omega_*$ in $H^2$, and Lemma~\ref{modula} allow the modulation parameters and their derivatives to extend continuously to $T_2+\delta_2$ for some $\delta_2>0$. This contradicts the definition of $T_2$, so $T_2=+\infty$. Define
$$
y_\infty=y(0)+\int_0^\infty\dot y(s){\rm d}s,\quad \phi_\infty=\phi(0)+\int_0^\infty\dot\phi(s){\rm d}s,
$$
then
\begin{align*}
|y(t)-y_\infty|+|\phi(t)-\phi_\infty|\leq&\int_t^\infty\left(|\dot y(s)|+|\dot\phi(s)|\right){\rm d}s\lesssim\varepsilon_1^2\langle t\rangle^{-\delta_*}.
\end{align*}
Hence (\ref{asphiy}) is done. 

{\bf Proof of (\ref{Lindeca}).}
Proposition~\ref{z-decay}, together with \eqref{u-z} and sphere constraint, gives
\begin{align*}
\bigl\|m(t)-R_{\phi(t)}\omega_*(\cdot-y(t))\bigr\|_{L^\infty_x}
&=\bigl\|\alpha(t)\omega_*+\operatorname{Re}u(t)n_*+\operatorname{Im}u(t)p_*\bigr\|_{L^\infty_x}\\
&\lesssim \|\alpha(t)\|_{L^\infty_x}+\|u(t)\|_{L^\infty_x}\\
&\lesssim \|z(t)\|_{L^\infty_x}
\lesssim\varepsilon_1\langle t\rangle^{-1/2}.
\end{align*}

\section{Modified scattering}\label{ssca}
\subsection{Evolution equation for the profile}
The modulation argument above yields $P_dz=0$. 
Define
\begin{equation}
\label{gpro2}
g(t,x)=e^{it\mathcal L_\eta}z(t,x).
\end{equation}
By \eqref{zevol} and Proposition~\ref{disfour},
\begin{equation}
\label{gproev}
\partial_t\widetilde g(t,\xi)
=-i e^{it\langle\xi\rangle_\mu^2}
\widetilde{\mathcal F}
\left(
\mathcal N_2+\mathcal N_3+\mathcal N_{\ge4}
+\mathcal Mod
\right)(t,\xi).
\end{equation}
We denote by $\mathcal{G}_2$, $\mathcal{G}_3$, $\mathcal{G}_{\ge4}$, and
$\mathcal{G}_{\rm mod}$ the quadratic, cubic, quartic and higher, and modulation
contributions to $\partial_t\widetilde g$ generated by $\mathcal N_2$, $\mathcal N_3$, $\mathcal N_{\ge4}$, and $\mathcal Mod$, respectively. Then $\mathcal{G}_2$, $\mathcal{G}_3$, $\mathcal{G}_{\ge4}$, and $\mathcal{G}_{\rm mod}$ can be computed explicitly.

\begin{prop}
\label{pqfns}
Define
\begin{align}
\label{mpj}
\Phi_{++}(\xi,a,b)
=\langle\xi\rangle_\mu^2-\langle a\rangle_\mu^2-\langle b\rangle_\mu^2, 
\quad
\Phi_{+-}(\xi,a,b)
=\langle\xi\rangle_\mu^2-\langle a\rangle_\mu^2+\langle b\rangle_\mu^2.
\end{align}
Then the quadratic contribution is given by
\begin{equation}\label{qc}
\begin{aligned}
\mathcal{G}_2(t,\xi)=&\frac{i}{2}\int_{\mathbb R^2}e^{it\Phi_{++}(\xi,a,b)}q_{++}(\xi,a,b)\widetilde g(t,a)\widetilde g(t,b){\rm d}a{\rm d}b\\
&+\frac{i}{2}\int_{\mathbb R^2}e^{it\Phi_{+-}(\xi,a,b)}q_{+-}(\xi,a,b)\widetilde g(t,a)\overline{\widetilde g(t,b)}{\rm d}a{\rm d}b,
\end{aligned}
\end{equation}
where
\begin{align}
q_{++}(\xi,a,b)
&=\int_{\mathbb R}\overline{\psi_\eta(x,\xi)}
\left[V_2(x)\psi_\eta(x,a)\psi_\eta(x,b)-\frac12V_3(x)\partial_x\left(\psi_\eta(x,a)\psi_\eta(x,b)\right)\right]{\rm d}x,\nonumber\\
q_{+-}(\xi,a,b)
&=\int_{\mathbb R}\overline{\psi_\eta(x,\xi)}\left[V_2(x)\psi_\eta(x,a)\overline{\psi_\eta(x,b)}+V_3(x)\partial_x\psi_\eta(x,a)\overline{\psi_\eta(x,b)}\right]{\rm d}x. 
\label{qc1}
\end{align}
Moreover, $q_{++}$ and $q_{+-}$ admit the following null structure:
\begin{equation}
\label{quker}
q_{++}(\xi,a,b)=\Phi_{++}(\xi,a,b)m_{++}(\xi,a,b),
\quad
q_{+-}(\xi,a,b)=\Phi_{+-}(\xi,a,b)m_{+-}(\xi,a,b),
\end{equation}
where
\begin{equation}
\label{qm}
\begin{aligned}
m_{++}(\xi,a,b)&=\frac{1}{4\sqrt{2\pi}}
\frac{\langle\xi\rangle_\mu}{\langle a\rangle_\mu\langle b\rangle_\mu}
\operatorname{sech}\left(\frac{\pi(\xi-a-b)}{2\mu}\right),\\
m_{+-}(\xi,a,b)&=-\frac{1}{2\sqrt{2\pi}}
\frac{\langle a\rangle_\mu}{\langle\xi\rangle_\mu\langle b\rangle_\mu}
\operatorname{sech}\left(\frac{\pi(\xi-a+b)}{2\mu}\right).
\end{aligned}
\end{equation}
\end{prop}

\begin{proof}
We first prove \eqref{qc}. By $\mathcal G_2$ and the formula for $\mathcal N_2$, one has 
\begin{align*}
\mathcal G_2(t,\xi)
=\frac{i}{2}e^{it\langle\xi\rangle_\mu^2}
\widetilde{\mathcal F}
\left(
V_2z^2-V_3z\partial_xz
\right)(t,\xi)
+\frac{i}{2}e^{it\langle\xi\rangle_\mu^2}
\widetilde{\mathcal F}
\left(
V_2z\overline z+V_3\overline z\,\partial_xz
\right)(t,\xi).
\end{align*}
Using \eqref{indF} and the definition of
$\widetilde g$, we find 
\begin{equation}
\label{profile-inv}
\begin{aligned}
z(t,x)
&=
\int_{\mathbb R}
e^{-it\langle a\rangle_\mu^2}
\psi_\eta(x,a)\widetilde g(t,a)\,\mathrm da,\\
\overline{z(t,x)}
&=
\int_{\mathbb R}
e^{it\langle b\rangle_\mu^2}
\overline{\psi_\eta(x,b)}
\overline{\widetilde g(t,b)}\,\mathrm db.
\end{aligned}
\end{equation}
Symmetrizing in $a$ and $b$ yields
\begin{align*}
&e^{it\langle\xi\rangle_\mu^2}
\widetilde{\mathcal F}
\left(
V_2z^2-V_3z\partial_xz
\right)(t,\xi)\\
={}&
\int_{\mathbb R^2}
e^{it\Phi_{++}(\xi,a,b)}
\widetilde g(t,a)\widetilde g(t,b)
\int_{\mathbb R}
\overline{\psi_\eta(x,\xi)}
\Bigl(
V_2(x)\psi_\eta(x,a)\psi_\eta(x,b)
-
V_3(x)\psi_\eta(x,a)\partial_x\psi_\eta(x,b)
\Bigr)  
{\rm d}x{\rm d}a{\rm d}b\\
=&
\int_{\mathbb R^2}
e^{it\Phi_{++}(\xi,a,b)}
\widetilde g(t,a)\widetilde g(t,b)
\int_{\mathbb R}
\overline{\psi_\eta(x,\xi)}
\Bigl(
V_2(x)\psi_\eta(x,a)\psi_\eta(x,b)
-\frac12V_3(x)\partial_x
\left(
\psi_\eta(x,a)\psi_\eta(x,b)
\right)
\Bigr)  
{\rm d}x{\rm d}a{\rm d}b.
\end{align*}
Similarly,
\begin{align*}
&e^{it\langle\xi\rangle_\mu^2}
\widetilde{\mathcal F}
\left(
V_2z\overline z+V_3\overline z\,\partial_xz
\right)(t,\xi)
=
\int_{\mathbb R^2}
e^{it\Phi_{+-}(\xi,a,b)}
q_{+-}(\xi,a,b)
\widetilde g(t,a)\overline{\widetilde g(t,b)}
\,\mathrm da\,\mathrm db,
\end{align*}
which proves \eqref{qc}.
It remains to prove \eqref{quker}. The two identities follow from the same computation, so we prove only the first one.
Using \eqref{v12} and \eqref{bSederi}, integrating by parts, we calculate
\begin{align*}
q_{++}(\xi,a,b)
&=\int_{\mathbb R}\overline{\psi_\eta(x,\xi)}
\left[
2\mu^2S_\eta T_\eta\psi_\eta(x,a)\psi_\eta(x,b)
-\mu S_\eta\partial_x\left(\psi_\eta(x,a)\psi_\eta(x,b)\right)
\right]\,\mathrm dx\\
&=-\mu\int_{\mathbb R}S_\eta
D_\eta^*\overline{\psi_\eta(x,\xi)}
\psi_\eta(x,a)\psi_\eta(x,b)\,\mathrm dx\\
&=-\frac{\mu\langle\xi\rangle_\mu}{\sqrt{2\pi}}
\int_{\mathbb R}e^{-ix\xi}S_\eta
\psi_\eta(x,a)\psi_\eta(x,b)\,\mathrm dx\\
&=-\frac{\mu\langle\xi\rangle_\mu}
{(2\pi)^{3/2}\langle a\rangle_\mu\langle b\rangle_\mu}
\int_{\mathbb R}e^{-i(\xi-a-b)x}S_\eta
\left[
-ab-i\mu(a+b)T_\eta+\mu^2T_\eta^2
\right]\,\mathrm dx,
\end{align*}
By \eqref{STide}, \eqref{bSederi} and letting
$p=\xi-a-b$, we have
\begin{equation}
\label{3redu}
\begin{aligned}
&\int_{\mathbb R}e^{-ipx}S_\eta
\left[
-ab-i\mu(a+b)T_\eta+\mu^2T_\eta^2
\right]\,\mathrm dx\\
={}&
-ab\int_{\mathbb R}e^{-ipx}S_\eta\,\mathrm dx
+i(a+b)\int_{\mathbb R}e^{-ipx}\partial_xS_\eta\,\mathrm dx
+\mu^2\int_{\mathbb R}e^{-ipx}S_\eta(1-S_\eta^2)\,\mathrm dx\\
={}&
\left[-ab-(a+b)p+\mu^2\right]
\int_{\mathbb R}e^{-ipx}S_\eta\,\mathrm dx
-\mu^2\int_{\mathbb R}e^{-ipx}S_\eta^3\,\mathrm dx.
\end{aligned}
\end{equation}
Applying \eqref{STide} and \eqref{bSederi} again, one has 
\begin{equation*}
\frac{d^2}{dx^2}S_{\eta}
=
\mu^2S_\eta-2\mu^2S_\eta^3.
\end{equation*}
Hence
\begin{equation}
\label{Fou4}
\begin{aligned}
\mu^2\int_{\mathbb R}e^{-ipx}S_\eta^3\,\mathrm dx
&=
\frac{\mu^2}{2}
\int_{\mathbb R}e^{-ipx}S_\eta\,\mathrm dx
-\frac12
\int_{\mathbb R}e^{-ipx}S_{\eta}^{(2)}\,\mathrm dx\\
&=
\frac{\mu^2+p^2}{2}
\int_{\mathbb R}e^{-ipx}S_\eta\,\mathrm dx.
\end{aligned}
\end{equation}
Combining \eqref{3redu},
\eqref{Fou4}, and
\eqref{Fotrseta}, we obtain
\begin{align*}
q_{++}(\xi,a,b)
&=
-\frac{\mu\langle\xi\rangle_\mu}
{(2\pi)^{3/2}\langle a\rangle_\mu\langle b\rangle_\mu}
\left[
-ab-(a+b)p+\frac{\mu^2-p^2}{2}
\right]
\int_{\mathbb R}e^{-ipx}S_\eta(x)\,\mathrm dx\\
&=
\frac{\mu\langle\xi\rangle_\mu}
{2(2\pi)^{3/2}\langle a\rangle_\mu\langle b\rangle_\mu}
\Phi_{++}(\xi,a,b)
\int_{\mathbb R}e^{-ipx}S_\eta(x)\,\mathrm dx\\
&=
\frac{\mu\langle\xi\rangle_\mu}
{2(2\pi)^{3/2}\langle a\rangle_\mu\langle b\rangle_\mu}
\Phi_{++}(\xi,a,b)
\frac{\pi}{\mu}
\operatorname{sech}\left(\frac{\pi p}{2\mu}\right)\\
&=
\frac{\Phi_{++}(\xi,a,b)}{4\sqrt{2\pi}}
\frac{\langle\xi\rangle_\mu}
{\langle a\rangle_\mu\langle b\rangle_\mu}
\operatorname{sech}\left(
\frac{\pi(\xi-a-b)}{2\mu}
\right).
\end{align*}
\end{proof}

\begin{prop}\label{cfdpr}
Define $\Phi_3(\xi,a,b,c)=\langle\xi\rangle_\mu^2-\langle a\rangle_\mu^2-\langle b\rangle_\mu^2+\langle c\rangle_\mu^2
$, $p=\xi-a-b+c$ and let $\mathfrak m_\pm(\xi)=\frac{\xi\pm i\mu}{\langle\xi\rangle_\mu}$.
Then the cubic contribution in the right hand side of (\ref{gproev}) is given by
\begin{equation}
\label{cc}
\mathcal G_3(t,\xi)=-\frac{i}{2}\int_{\mathbb R^3}e^{it\Phi_3(\xi,a,b,c)}q_3(\xi,a,b,c)\widetilde g(t,a)\widetilde g(t,b)\overline{\widetilde g(t,c)}\,{\rm d}a\,{\rm d}b\,{\rm d}c,
\end{equation}
where
\begin{equation}
\label{3kernel}
q_3(\xi,a,b,c)=\int_{\mathbb R}\overline{\psi_\eta(x,\xi)}\overline{\psi_\eta(x,c)}
\left[\partial_x\psi_\eta(x,a)\partial_x\psi_\eta(x,b)-V_1(x)\psi_\eta(x,a)\psi_\eta(x,b)\right]\,{\rm d}x.
\end{equation}
Moreover, we have the decomposition
\begin{equation}
\label{cukede}
q_3=q_{3}^{\delta}+q_{3}^{p.v.}+q_{3}^{r},
\end{equation}
where
\begin{align}
\label{cc1}
q_3^\delta(\xi,a,b,c)
=&-\frac{ab+\mu^2}{4\pi}
\Bigl[\mathfrak m_-(\xi)\mathfrak m_+(a)\mathfrak m_+(b)\mathfrak m_-(c)
+\mathfrak m_+(\xi)\mathfrak m_-(a)\mathfrak m_-(b)\mathfrak m_+(c)\Bigr]\delta(p),\\
\label{cc2}
q_{3}^{p.v.}(\xi,a,b,c)
=&\frac{ab+\mu^2}{8\pi i\mu}
\mathfrak m_+(\xi)\mathfrak m_-(a)\mathfrak m_-(b)\mathfrak m_+(c)\left[1-\overline{t_\eta(\xi)}t_\eta(a)t_\eta(b)
\overline{t_\eta(c)}\right]
\operatorname{p.v.}\operatorname{csch}\left(\frac{\pi p}{2\mu}\right),\\
\label{cc3}
q_{3}^r(\xi,a,b,c)
&=\frac{1}{4\pi^2}\int_{\mathbb R}e^{-i(\xi-a-b+c)x}S_\eta^2(x)\sum_{j=0}^4p_3^j(\xi,a,b,c)T_\eta^j(x)\,{\rm d}x.
\end{align}
with
\begin{align*}
p_3^0(\xi,a,b,c)
&=\frac{\mu^2\left\{(ab+\mu^2)\left(\xi a+\xi b-\xi c-ab+ac+bc+\mu^2\right)-\xi c\left((a-b)^2+\mu^2\right)\right\}}
{\langle\xi\rangle_\mu\langle a\rangle_\mu\langle b\rangle_\mu\langle c\rangle_\mu},\\
p_3^{1}(\xi,a,b,c)
&=\frac{i\mu^3\left\{(ab+\mu^2)(\xi-a-b+c)+(\xi+c)\left((a-b)^2+\mu^2\right)+\xi c(a+b)\right\}}
{\langle\xi\rangle_\mu\langle a\rangle_\mu\langle b\rangle_\mu\langle c\rangle_\mu},\\
p_3^2(\xi,a,b,c)
&=\frac{\mu^4\left\{ab+(a-b)^2+(\xi+c)(a+b)-\xi c+2\mu^2\right\}}
{\langle\xi\rangle_\mu\langle a\rangle_\mu\langle b\rangle_\mu\langle c\rangle_\mu},\\
p_3^3(\xi,a,b,c)
&=\frac{i\mu^5(\xi-a-b+c)}
{\langle\xi\rangle_\mu\langle a\rangle_\mu\langle b\rangle_\mu\langle c\rangle_\mu},\\
p_3^4(\xi,a,b,c)
&=\frac{\mu^6}
{\langle\xi\rangle_\mu\langle a\rangle_\mu\langle b\rangle_\mu\langle c\rangle_\mu}.
\end{align*}
Correspondingly, we define $\mathcal G_3^{\delta}$, $\mathcal G_3^{p.v.}$, and $\mathcal G_3^r$ by replacing $q_3$ in \eqref{cc} with $q_3^{\delta}$, $q_3^{p.v.}$ and $q_3^r$, respectively.
\end{prop}
\begin{proof}
We first prove \eqref{cc}. By the definition of $\mathcal G_3$ and the formula for $\mathcal N_3$,
$$
\mathcal G_3(t,\xi)=-\frac{i}{2}e^{it\langle\xi\rangle_\mu^2}\widetilde{\mathcal F}\left(\overline z(\partial_xz)^2-V_1|z|^2z\right)(t,\xi).
$$
Using \eqref{profile-inv} and proceeding as in the proof of Proposition~\ref{pqfns}, we obtain \eqref{cc}. It remains to prove \eqref{cukede}. The explicit formula \eqref{gee} gives
\begin{equation}
\label{geig}
\partial_x\psi_\eta(x,\xi)=-\frac{1}{\sqrt{2\pi}}\frac{\xi^2+i\mu T_\eta(x)\xi+\mu^2S_\eta^2(x)}{\langle\xi\rangle_\mu}e^{i\xi x}.
\end{equation}
Substituting \eqref{gee} and \eqref{geig} into \eqref{3kernel}, using \eqref{v12} and letting $p=\xi-a-b+c$, we get
\begin{align*}
q_3(\xi,a,b,c)
&=\frac{1}{4\pi^2\langle\xi\rangle_\mu\langle a\rangle_\mu\langle b\rangle_\mu\langle c\rangle_\mu}
\int_{\mathbb R}e^{-ipx}\left(i\xi+\mu T_\eta\right)\left(ic+\mu T_\eta\right)\\
&\cdot\Bigl[\left(a^2+i\mu aT_\eta+\mu^2S_\eta^2\right)\left(b^2+i\mu bT_\eta+\mu^2S_\eta^2\right)
-\mu^2(1-2S_\eta^2)\left(ia-\mu T_\eta\right)\left(ib-\mu T_\eta\right)\Bigr]\,{\rm d}x.
\end{align*}
Notice that
\begin{align*}
&\left(a^2+i\mu aT_\eta+\mu^2S_\eta^2\right)\left(b^2+i\mu bT_\eta+\mu^2S_\eta^2\right)\\
={}&-ab\left(ia-\mu T_\eta\right)\left(ib-\mu T_\eta\right)
+\mu^2S_\eta^2\left(a^2+b^2+i\mu(a+b)T_\eta+\mu^2S_\eta^2\right).
\end{align*}
Then  \eqref{STide} implies that 
\begin{align*}
&\left(a^2+i\mu aT_\eta+\mu^2S_\eta^2\right)\left(b^2+i\mu bT_\eta+\mu^2S_\eta^2\right)
-\mu^2(1-2S_\eta^2)\left(ia-\mu T_\eta\right)\left(ib-\mu T_\eta\right)\\
={}&-(ab+\mu^2)\left(ia-\mu T_\eta\right)\left(ib-\mu T_\eta\right)\\
&+\mu^2S_\eta^2\left[2\left(ia-\mu T_\eta\right)\left(ib-\mu T_\eta\right)
+a^2+b^2+i\mu(a+b)T_\eta+\mu^2S_\eta^2\right]\\
={}&-(ab+\mu^2)\left(ia-\mu T_\eta\right)\left(ib-\mu T_\eta\right)
+\mu^2S_\eta^2\left[\left(ia-\mu T_\eta\right)\left(ib-\mu T_\eta\right)+a^2-ab+b^2+\mu^2\right].
\end{align*}
Therefore, with $p=\xi-a-b+c$, one has 
\begin{equation}
\begin{aligned}
&q_3(\xi,a,b,c)\nonumber\\
={}&-\frac{ab+\mu^2}{4\pi^2}\int_{\mathbb R}e^{-ipx}
\frac{\left(i\xi+\mu T_\eta\right)\left(ia-\mu T_\eta\right)\left(ib-\mu T_\eta\right)\left(ic+\mu T_\eta\right)}
{\langle\xi\rangle_\mu\langle a\rangle_\mu\langle b\rangle_\mu\langle c\rangle_\mu}\,{\rm d}x\\
&+\frac{\mu^2}{4\pi^2}\int_{\mathbb R}e^{-ipx}S_\eta^2
\frac{\left(i\xi+\mu T_\eta\right)\left(ic+\mu T_\eta\right)}
{\langle\xi\rangle_\mu\langle a\rangle_\mu\langle b\rangle_\mu\langle c\rangle_\mu}
\left[\left(ia-\mu T_\eta\right)\left(ib-\mu T_\eta\right)+a^2-ab+b^2+\mu^2\right]\,{\rm d}x.
\end{aligned}
\label{q3d}
\end{equation}
Notice that by the definition $\mathfrak m_\pm(\xi)$ and  \eqref{STide}
\begin{equation}
\begin{aligned}
&\frac{\left(i\xi+\mu T_\eta\right)\left(ia-\mu T_\eta\right)\left(ib-\mu T_\eta\right)\left(ic+\mu T_\eta\right)}
{\langle\xi\rangle_\mu\langle a\rangle_\mu\langle b\rangle_\mu\langle c\rangle_\mu}\\
={}&\frac{1+T_\eta}{2}\mathfrak m_-(\xi)\mathfrak m_+(a)\mathfrak m_+(b)\mathfrak m_-(c)
+\frac{1-T_\eta}{2}\mathfrak m_+(\xi)\mathfrak m_-(a)\mathfrak m_-(b)\mathfrak m_+(c)
+S_\eta^2\sum_{j=0}^2r_j(\xi,a,b,c)T_\eta^j,
\end{aligned}
\label{eq6}
\end{equation}
where
\begin{align*}
r_0(\xi,a,b,c)
&=\frac{\mu^2\left(ab-ac-a\xi-bc-b\xi+c\xi-\mu^2\right)}
{\langle\xi\rangle_\mu\langle a\rangle_\mu\langle b\rangle_\mu\langle c\rangle_\mu},\\
r_1(\xi,a,b,c)
&=\frac{i\mu^3(a+b-c-\xi)}
{\langle\xi\rangle_\mu\langle a\rangle_\mu\langle b\rangle_\mu\langle c\rangle_\mu},\\
r_2(\xi,a,b,c)
&=-\frac{\mu^4}
{\langle\xi\rangle_\mu\langle a\rangle_\mu\langle b\rangle_\mu\langle c\rangle_\mu}.
\end{align*}
Substituting \eqref{eq6} into \eqref{q3d}, using
\eqref{Ftone} and \eqref{Ftt}, we calculate
\begin{align*}
&q_3(\xi,a,b,c)\\
=&-\frac{ab+\mu^2}{4\pi}
\Bigl[\mathfrak m_-(\xi)\mathfrak m_+(a)\mathfrak m_+(b)\mathfrak m_-(c)
+\mathfrak m_+(\xi)\mathfrak m_-(a)\mathfrak m_-(b)\mathfrak m_+(c)\Bigr]\delta(p)\\
&+\frac{ab+\mu^2}{8\pi i\mu}
\Bigl[\mathfrak m_+(\xi)\mathfrak m_-(a)\mathfrak m_-(b)\mathfrak m_+(c)
-\mathfrak m_-(\xi)\mathfrak m_+(a)\mathfrak m_+(b)\mathfrak m_-(c)\Bigr]
\operatorname{p.v.}\operatorname{csch}\left(\frac{\pi p}{2\mu}\right)\\
&+\frac{1}{4\pi^2}\int_{\mathbb R}e^{-ipx}S_\eta^2(x)
\Biggl\{
-(ab+\mu^2)\sum_{j=0}^2r_j(\xi,a,b,c)T_\eta^j(x)
+\frac{\mu^2\left(i\xi+\mu T_\eta(x)\right)\left(ic+\mu T_\eta(x)\right)}
{\langle\xi\rangle_\mu\langle a\rangle_\mu\langle b\rangle_\mu\langle c\rangle_\mu}\\
&\quad\cdot\left[\left(ia-\mu T_\eta(x)\right)\left(ib-\mu T_\eta(x)\right)
+a^2-ab+b^2+\mu^2\right]
\Biggr\}{\rm d}x\\
&=q_3^\delta(\xi,a,b,c)+q_3^{p.v.}(\xi,a,b,c)+q_3^{r}(\xi,a,b,c).
\end{align*}
Moreover, 
\begin{align*}
q_3^{p.v.}(\xi,a,b,c)
=&\frac{ab+\mu^2}{8\pi i\mu}
\Bigl[\mathfrak m_+(\xi)\mathfrak m_-(a)\mathfrak m_-(b)\mathfrak m_+(c)
-\mathfrak m_-(\xi)\mathfrak m_+(a)\mathfrak m_+(b)\mathfrak m_-(c)\Bigr]
\operatorname{p.v.}\operatorname{csch}\left(\frac{\pi p}{2\mu}\right)\\
=&\frac{ab+\mu^2}{8\pi i\mu}
\mathfrak m_+(\xi)\mathfrak m_-(a)\mathfrak m_-(b)\mathfrak m_+(c)\left[1-\overline{t_\eta(\xi)}t_\eta(a)t_\eta(b)
\overline{t_\eta(c)}\right]
\operatorname{p.v.}\operatorname{csch}\left(\frac{\pi p}{2\mu}\right).
\end{align*}
\end{proof}

Proposition~\ref{pqfns} and
Proposition~\ref{cfdpr} give precise descriptions of the structures of $\mathcal G_2$ and $\mathcal G_3$, respectively. The quartic and higher contribution is simply given by
\begin{equation}
\label{hc}
\mathcal G_{\ge4}(t,\xi)
=
-i e^{it\langle\xi\rangle_\mu^2}
\widetilde{\mathcal F}\mathcal N_{\ge4}(t,\xi).
\end{equation}
Finally, we turn to the modulation contribution 
\begin{align}
\label{mc}
\mathcal G_{\rm mod}(t,\xi)
={}&
\dot y\,e^{it\langle\xi\rangle_\mu^2}
\widetilde{\mathcal F}
\left(
\partial_xz-\frac{\mu}{4}S_\eta z^2
\right)(t,\xi)-i(\dot\phi+\eta\dot y)
e^{it\langle\xi\rangle_\mu^2}
\widetilde{\mathcal F}
\left(
T_\eta z-\frac14S_\eta z^2
\right)(t,\xi).
\end{align}
The term involving $T_\eta z$ is different from the other modulation terms,
since $T_\eta$ is nonlocalized and its Fourier transform contains a principal
value singularity. Therefore, we decompose
$$
\mathcal G_{\rm mod}=
\mathcal G_{\rm mod}^r
+
\mathcal G_{\rm mod}^s,
$$
where
\begin{equation}
\label{mr}
\mathcal G_{\rm mod}^r(t,\xi)=
\dot y e^{it\langle\xi\rangle_\mu^2}
\widetilde{\mathcal F}\left(\partial_xz-\frac{\mu}{4}S_\eta z^2\right)(t,\xi)+\frac{i}{4}(\dot\phi+\eta\dot y)e^{it\langle\xi\rangle_\mu^2}
\widetilde{\mathcal F}\left(S_\eta z^2\right)(t,\xi),
\end{equation}
and
\begin{equation}
\label{ms}
\mathcal{G}_{\rm mod}^s(t,\xi)=-i(\dot\phi+\eta\dot y)e^{it\langle\xi\rangle_\mu^2}
\widetilde{\mathcal F}\left(T_\eta z\right)(t,\xi).
\end{equation}

\begin{prop}
The singular part of modulation contribution satisfies
\begin{equation}
\label{sms}
\begin{aligned}
\mathcal G_{\rm mod}^s(t,\xi)
={}&
-\frac{\dot\phi+\eta\dot y}{2\mu}
\operatorname{p.v.}\int_{\mathbb R}
e^{it(\langle\xi\rangle_\mu^2-\langle a\rangle_\mu^2)}
\operatorname{csch}\left(
\frac{\pi(\xi-a)}{2\mu}
\right)
\widetilde g(t,a)\,\mathrm da\\
&-\frac{\dot\phi+\eta\dot y}{2\mu}
\int_{\mathbb R}
e^{it(\langle\xi\rangle_\mu^2-\langle a\rangle_\mu^2)}
\frac{(\langle\xi\rangle_\mu-\langle a\rangle_\mu)^2}
{2\langle\xi\rangle_\mu\langle a\rangle_\mu}
\operatorname{csch}\left(\frac{\pi(\xi-a)}{2\mu}\right)
\widetilde g(t,a){\rm d}a.
\end{aligned}
\end{equation}
\end{prop}

\begin{proof}
By \eqref{ms}, \eqref{dF},
and \eqref{profile-inv}, we get
\begin{equation}
\label{eGmods}
\begin{aligned}
\mathcal G_{\rm mod}^s(t,\xi)
={}&-i(\dot\phi+\eta\dot y)\int_{\mathbb R}e^{it(\langle\xi\rangle_\mu^2-\langle a\rangle_\mu^2)}
\left[\int_{\mathbb R}\overline{\psi_\eta(x,\xi)}T_\eta(x)
\psi_\eta(x,a){\rm d}x\right]\widetilde g(t,a){\rm d}a.
\end{aligned}
\end{equation}
Using
\eqref{gee},
\eqref{Ftt},
\eqref{Ftsq}, and
\eqref{Teseta}, we calculate
$$
\int_{\mathbb R}\overline{\psi_\eta(x,\xi)}T_\eta(x)\psi_\eta(x,a)\,{\rm d}x
=-\frac{i}{4\mu}\left(
\frac{\langle\xi\rangle_\mu}{\langle a\rangle_\mu}
+\frac{\langle a\rangle_\mu}{\langle\xi\rangle_\mu}\right)
\operatorname{p.v.}\operatorname{csch}
\left(\frac{\pi(\xi-a)}{2\mu}\right).
$$
Substituting this identity into \eqref{eGmods} and using \eqref{profile-inv} again yields
$$
\mathcal G_{\rm mod}^s(t,\xi)
=-\frac{\dot\phi+\eta\dot y}{4\mu}
\operatorname{p.v.}\int_{\mathbb R}
e^{it(\langle\xi\rangle_\mu^2-\langle a\rangle_\mu^2)}
\left(
\frac{\langle\xi\rangle_\mu}{\langle a\rangle_\mu}
+\frac{\langle a\rangle_\mu}{\langle\xi\rangle_\mu}\right)
\operatorname{csch}\left(\frac{\pi(\xi-a)}{2\mu}\right)
\widetilde g(t,a)\,{\rm d}a.
$$
Finally, \eqref{sms} follows from the identity
$$
\frac12\left(
\frac{\langle\xi\rangle_\mu}{\langle a\rangle_\mu}
+\frac{\langle a\rangle_\mu}{\langle\xi\rangle_\mu}\right)-1
=\frac{(\langle\xi\rangle_\mu-\langle a\rangle_\mu)^2}
{2\langle\xi\rangle_\mu\langle a\rangle_\mu}.
$$
\end{proof}
Hence
\begin{equation*}
\partial_t\widetilde g
=\mathcal G_2+\mathcal G_3^{\delta}+\mathcal G_3^{p.v.}+\mathcal G_3^r
+\mathcal G_{\ge4}+\mathcal G_{\rm mod}^r+\mathcal G_{\rm mod}^s.
\end{equation*}
To overcome the slow local decay caused by the threshold resonances of the linearized operator, we use the null structure of $q_{++}$ and $q_{+-}$ established in Proposition~\ref{pqfns} to remove the quadratic contribution by a normal form transformation. We therefore set
\begin{align}
\widetilde h(t,\xi)
={}&\widetilde g(t,\xi)
-\frac12\int_{\mathbb R^2}e^{it\Phi_{++}(\xi,a,b)}m_{++}(\xi,a,b)
\widetilde g(t,a)\widetilde g(t,b)\,{\rm d}a\,{\rm d}b\nonumber\\
&-\frac12\int_{\mathbb R^2}e^{it\Phi_{+-}(\xi,a,b)}m_{+-}(\xi,a,b)
\widetilde g(t,a)\overline{\widetilde g(t,b)}\,{\rm d}a\,{\rm d}b.\label{nfpr}
\end{align}
Then
\begin{equation}
\label{nfpe}
\begin{aligned}
\partial_t\widetilde h
={}&\mathcal G_3^{\delta}+\mathcal G_3^{p.v.}+\mathcal G_3^r+\mathcal G_{\ge4}
+\mathcal G_{\rm mod}^r+\mathcal G_{\rm mod}^s\\
&-\frac12\int_{\mathbb R^2}e^{it\Phi_{++}(\xi,a,b)}m_{++}(\xi,a,b)
\partial_t\!\left[\widetilde g(t,a)\widetilde g(t,b)\right]\,{\rm d}a\,{\rm d}b\\
&-\frac12\int_{\mathbb R^2}e^{it\Phi_{+-}(\xi,a,b)}m_{+-}(\xi,a,b)
\partial_t\!\left[\widetilde g(t,a)\overline{\widetilde g(t,b)}\right]\,{\rm d}a\,{\rm d}b.
\end{aligned}
\end{equation}

\subsection{The Dirac delta contribution}

We analyze the Dirac delta contribution $\mathcal G_3^\delta$ to the cubic interaction. Its leading term determines the logarithmic phase correction, while the remainder is integrable in time in $L^\infty_\xi$. 

Recall $\mathcal{G}^{\delta}_{3}(t,x)$ defined in (\ref{cc})--(\ref{cc1}). 
\begin{prop}
Let $\rho$ and $\delta_*$ be the constants fixed in (\ref{Nota}). Then, for every $t\geq1$,
$$
\mathcal{G}_3^{\delta}(t,\xi)
=
\frac{i}{4t}\langle\xi\rangle_\mu^2
|\widetilde g(t,\xi)|^2
\widetilde g(t,\xi)
+\mathcal R_{3,\delta}(t,\xi),
$$
where
$$
\|\mathcal R_{3,\delta}(t)\|_{L^\infty_\xi}
\lesssim\varepsilon_1^3t^{-1-\rho}.
$$
\end{prop}

\begin{proof}
Recall $\mathfrak m_{\sigma}=\frac{\xi+i\sigma \mu}{\langle \xi\rangle_{\mu}}$ for $\sigma=\pm 1$. By \eqref{cc1}, we get
\begin{equation}
\label{dctie}
\begin{aligned}
&\mathcal{G}_{3}^\delta(t,\xi)\\
=&\frac{i}{8\pi}\sum_{\sigma=\pm1}\mathfrak m_{-\sigma}(\xi)
\int_{\mathbb R^2}e^{2it(a-\xi)(b-\xi)}
a\mathfrak m_\sigma(a)\widetilde g(t,a)
b\mathfrak m_\sigma(b)\widetilde g(t,b)
\overline{\mathfrak m_\sigma(a+b-\xi)\widetilde g(t,a+b-\xi)}
\,{\rm d}a\,{\rm d}b\\
&+\frac{i\mu^2}{8\pi}\sum_{\sigma=\pm1}\mathfrak m_{-\sigma}(\xi)
\int_{\mathbb R^2}e^{2it(a-\xi)(b-\xi)}
\mathfrak m_\sigma(a)\widetilde g(t,a)
\mathfrak m_\sigma(b)\widetilde g(t,b)
\overline{\mathfrak m_\sigma(a+b-\xi)\widetilde g(t,a+b-\xi)}
\,{\rm d}a\,{\rm d}b.
\end{aligned}
\end{equation}

For the first integral of \eqref{dctie}, we find  
\begin{align*}
&\int_{\mathbb R^2}e^{2it(a-\xi)(b-\xi)}
a\mathfrak m_\sigma(a)\widetilde g(t,a)
b\mathfrak m_\sigma(b)\widetilde g(t,b)
\overline{\mathfrak m_\sigma(a+b-\xi)\widetilde g(t,a+b-\xi)}
\,{\rm d}a\,{\rm d}b\\
=&2\pi e^{it\xi^2}
\mathcal F\Bigg[
e^{it\partial_x^2}
\mathcal F^{-1}\bigl(
\lambda\mathfrak m_\sigma(\lambda)\widetilde g(t,\lambda)
\bigr)
e^{it\partial_x^2}
\mathcal F^{-1}\bigl(
\lambda\mathfrak m_\sigma(\lambda)\widetilde g(t,\lambda)
\bigr)\overline{
e^{it\partial_x^2}
\mathcal F^{-1}\bigl(
\mathfrak m_\sigma(\lambda)\widetilde g(t,\lambda)
\bigr)}
\Bigg](\xi).
\end{align*}
Letting $f_1=\mathcal F^{-1}\bigl(\lambda\mathfrak m_\sigma(\lambda)\widetilde g(t,\lambda)\bigr)$, $f_2=\mathcal F^{-1}\bigl(\mathfrak m_\sigma(\lambda)\widetilde g(t,\lambda)\bigr)$, applying \eqref{freere}, \eqref{a1}--\eqref{a3}, and using the definition of $\mathfrak m_\sigma$, one obtains  
\begin{align*}
&\Bigg\|
\int_{\mathbb R^2}e^{2it(a-\xi)(b-\xi)}
a\mathfrak m_\sigma(a)\widetilde g(t,a)
b\mathfrak m_\sigma(b)\widetilde g(t,b)
\overline{\mathfrak m_\sigma(a+b-\xi)\widetilde g(t,a+b-\xi)}
\,{\rm d}a\,{\rm d}b-\frac{\pi}{t}\widehat f_1(\xi)^2\overline{\widehat f_2(\xi)}\Bigg\|_{L^\infty_\xi}\\
\lesssim&
t^{-1-\alpha}
\Bigg\{
2\left(\|f_1\|_{L^2_x}+\|xf_1\|_{L^2_x}\right)
\left[\|\widehat f_1\|_{L^\infty_\xi}
+t^{-\alpha}\left(\|f_1\|_{L^2_x}+\|xf_1\|_{L^2_x}\right)\right]\\
\times&
\left[\|\widehat f_2\|_{L^\infty_\xi}
+t^{-\alpha}\left(\|f_2\|_{L^2_x}+\|xf_2\|_{L^2_x}\right)\right]
+
\left(\|f_2\|_{L^2_x}+\|xf_2\|_{L^2_x}\right)
\left[\|\widehat f_1\|_{L^\infty_\xi}
+t^{-\alpha}\left(\|f_1\|_{L^2_x}+\|xf_1\|_{L^2_x}\right)\right]^2
\Bigg\}\\
\lesssim&\varepsilon_1^3 t^{-1-\rho}.
\end{align*}
Therefore, there holds 
\begin{equation}\label{firin}
\begin{aligned}
&\Bigg\|
\int_{\mathbb R^2}e^{2it(a-\xi)(b-\xi)}a\mathfrak m_\sigma(a)\widetilde g(t,a)b\mathfrak m_\sigma(b)\widetilde g(t,b)\overline{\mathfrak m_\sigma(a+b-\xi)\widetilde g(t,a+b-\xi)}{\rm d}a{\rm d}b\\
&\qquad\qquad\qquad\qquad\qquad\qquad\qquad\qquad-\frac{\pi}{t}\bigl[\xi\mathfrak m_\sigma(\xi)\widetilde g(t,\xi)\bigr]^2\overline{\mathfrak m_\sigma(\xi)\widetilde g(t,\xi)}\Bigg\|_{L^\infty_\xi}\lesssim \varepsilon_1^3t^{-1-\rho}.
\end{aligned}
\end{equation}
Similarly, for the second term of \eqref{dctie}, we find
\begin{equation}\label{sw4}
\begin{aligned}
&\Bigg\|
\int_{\mathbb R^2}e^{2it(a-\xi)(b-\xi)}\mathfrak m_\sigma(a)\widetilde g(t,a)\mathfrak m_\sigma(b)\widetilde g(t,b)\overline{\mathfrak m_\sigma(a+b-\xi)\widetilde g(t,a+b-\xi)}{\rm d}a{\rm d}b\\
&\qquad\qquad\qquad\qquad\qquad\qquad\qquad\qquad-\frac{\pi}{t}\bigl[\mathfrak m_\sigma(\xi)\widetilde g(t,\xi)\bigr]^2
\overline{\mathfrak m_\sigma(\xi)\widetilde g(t,\xi)}\Bigg\|_{L^\infty_\xi}\lesssim\varepsilon_1^3t^{-1-\rho}.
\end{aligned}
\end{equation}
Combining $\overline{\mathfrak m_\sigma(\xi)}=\mathfrak m_{-\sigma}(\xi)$, $\mathfrak m_\sigma(\xi)\mathfrak m_{-\sigma}(\xi)=1$ and \eqref{dctie}, \eqref{firin} and \eqref{sw4}, we conclude
\begin{align*}
\mathcal{G}_{3}^\delta(t,\xi)=&\frac{i}{8t}\sum_{\sigma=\pm1}\mathfrak m_{-\sigma}(\xi)\bigl[\xi\mathfrak m_\sigma(\xi)\widetilde g(t,\xi)\bigr]^2\overline{\mathfrak m_\sigma(\xi)\widetilde g(t,\xi)}\\
&\quad +\frac{i\mu^2}{8t}\sum_{\sigma=\pm1}\mathfrak m_{-\sigma}(\xi)\bigl[\mathfrak m_\sigma(\xi)\widetilde g(t,\xi)\bigr]^2\overline{\mathfrak m_\sigma(\xi)\widetilde g(t,\xi)}+\mathcal R_{3,\delta}(t,\xi)\\
=&\frac{i}{4t}\langle\xi\rangle_\mu^2|\widetilde g(t,\xi)|^2\widetilde g(t,\xi)+\mathcal R_{3,\delta}(t,\xi).
\end{align*}
where $\|\mathcal R_{3,\delta}(t)\|_{L^\infty_\xi}\lesssim\varepsilon_1^3t^{-1-\rho}$.
Thus the desired result follows.
\end{proof}

\subsection{The principal value contribution}

The aim of this subsection is to prove Proposition~\ref{epvc}, the key observation is that the factor coming from the transmission coefficients splits into two parts: a regular part, which cancels the singularity of the p.v. kernel, and a part that vanishes on the space resonant set. 

We begin with the following lemma.
\begin{lem}
\label{Deipd}
For every $0\leq j\leq N$, we have
$$
\left\|\partial_x^jD_\eta^{-1}z(t)\right\|_{L^\infty}
\lesssim\varepsilon_1\langle t\rangle^{-1/2},
\quad t\ge 0.
$$
\end{lem}

\begin{proof}
We prove the conclusion by induction. The case $j=0$ is straightforward. Indeed, by the orthogonality condition and standard ODE theory, we obtain
$$D_\eta^{-1}z(t,x)=\frac{1}{S_\eta(x)}\int_{-\infty}^x S_\eta(y)z(t,y)\,{\rm d}y=-\frac{1}{S_\eta(x)}\int_x^\infty S_\eta(y)z(t,y)\,{\rm d}y.$$
Using the first representation for $x\leq0$ and the second for $x\geq0$, we have
\begin{align*}
\left\|D_\eta^{-1}z(t)\right\|_{L^\infty}
&\le\|z(t)\|_{L^\infty}\max\left\{\sup_{x\leq0}\frac{1}{S_\eta(x)}\int_{-\infty}^xS_\eta(y){\rm d}y,\sup_{x\geq0}\frac{1}{S_\eta(x)}\int_x^\infty S_\eta(y){\rm d}y\right\}\\
&\lesssim\|z(t)\|_{L^\infty}.
\end{align*}
Now assume that the conclusion holds for every integer $k\in[0,j]$, where $0\leq j\leq N-1$. Differentiating $\partial_xD_\eta^{-1}z=z+\mu T_\eta D_\eta^{-1}z$ gives
$$\partial_x^{j+1}D_\eta^{-1}z=\partial_x^jz+\mu\sum_{k=0}^jC_j^k(\partial_x^{j-k}T_\eta)\partial_x^kD_\eta^{-1}z.$$
Proposition~\ref{z-decay} and the induction hypothesis then imply
\begin{align*}
\left\|\partial_x^{j+1}D_\eta^{-1}z(t)\right\|_{L^\infty}
\lesssim
\left\|\partial_x^jz(t)\right\|_{L^\infty}
+
\sum_{k=0}^j
\left\|\partial_x^kD_\eta^{-1}z(t)\right\|_{L^\infty}\lesssim
\varepsilon_1\langle t\rangle^{-1/2}.
\end{align*}

Thus the conclusion also holds for $j+1$. By induction, the proof is complete.
\end{proof}

Recall the definition of $\mathcal{G}_3^{p.v.}$ in (\ref{cc})--(\ref{cc2}). 
\begin{prop}
\label{epvc}

Let $\rho$ be the constants fixed in (\ref{Nota}). Then, for every $t\geq1$, we have
$$
\|\mathcal{G}_3^{p.v.}\|_{L^\infty_\xi}\lesssim\varepsilon_1^3t^{-1-\rho}.
$$
\end{prop}

\begin{proof}
Recall that  $t_\eta(r)=\frac{\mathfrak m_+(r)}{\mathfrak m_-(r)}$ and the formula (\ref{cc2}) 
\begin{equation*}
q_{3}^{p.v.}(\xi,a,b,c)
=\frac{ab+\mu^2}{8\pi i\mu}\mathfrak m_+(\xi)\mathfrak m_-(a)\mathfrak m_-(b)\mathfrak m_+(c)\left(1-\overline{t_\eta(\xi)}t_\eta(a)t_\eta(b)\overline{t_\eta(c)}\right)\operatorname{p.v.}\operatorname{csch}\left(\frac{\pi p}{2\mu}\right).
\end{equation*}
Since $p=\xi-a-b+c$, we write $c=a+b-\xi+p$ and take $\partial_b$ at fixed $(\xi,a,p)$ throughout the following decomposition and integration by parts. Thus $\partial_b\Phi_3(\xi,a,b,a+b-\xi+p)=2(c-b)=2(a-\xi+p)$ and $\partial_b p=0$. Direct calculations show 
\begin{align*}
&1-\overline{t_\eta(\xi)}t_\eta(a)t_\eta(b)\overline{t_\eta(c)}\\
=&\Bigl(1-\overline{t_\eta(\xi)}t_\eta(a)\Bigr)+\overline{t_\eta(\xi)}t_\eta(a)\Bigl(1-t_\eta(b)\overline{t_\eta(c)}\Bigr)\\
=&\frac{2i\mu(a-\xi)}{(\xi+i\mu)(a-i\mu)}
+\frac{2i\mu\overline{t_\eta(\xi)}t_\eta(a)(b-c)}
{(b-i\mu)(c+i\mu)}\\
=&\partial_b\Phi_3
\left(\frac{i\mu}{(\xi+i\mu)(a-i\mu)}
-\frac{i\mu\overline{t_\eta(\xi)}t_\eta(a)}
{(b-i\mu)(c+i\mu)}\right)
-\frac{2i\mu p}{(\xi+i\mu)(a-i\mu)}.
\end{align*}
The definitions of $\mathfrak m_\pm$ and $t_\eta$ then yield 
\begin{align*}
q_3^{p.v.}(\xi,a,b,c)
={}&\frac{ab+\mu^2}{8\pi}\partial_b\Phi_3
\frac{\mathfrak m_+(\xi)\mathfrak m_-(a)\mathfrak m_-(b)\mathfrak m_+(c)}
{(\xi+i\mu)(a-i\mu)}
\operatorname{p.v.}\operatorname{csch}\left(\frac{\pi p}{2\mu}\right)\\
&-\frac{ab+\mu^2}{8\pi}\partial_b\Phi_3
\frac{\mathfrak m_+(\xi)\mathfrak m_-(a)\mathfrak m_-(b)\mathfrak m_+(c)
\overline{t_\eta(\xi)}t_\eta(a)}
{(b-i\mu)(c+i\mu)}
\operatorname{p.v.}\operatorname{csch}\left(\frac{\pi p}{2\mu}\right)\\
&-\frac{ab+\mu^2}{4\pi}
\frac{\mathfrak m_+(\xi)\mathfrak m_-(a)\mathfrak m_-(b)\mathfrak m_+(c)}
{(\xi+i\mu)(a-i\mu)}
p\operatorname{p.v.}\operatorname{csch}\left(\frac{\pi p}{2\mu}\right)\\
={}&\frac{ab+\mu^2}{8\pi}\partial_b\Phi_3
\left(
\frac{\mathfrak m_-(b)\mathfrak m_+(c)}
{\langle\xi\rangle_\mu\langle a\rangle_\mu}
-\frac{\mathfrak m_-(\xi)\mathfrak m_+(a)}
{\langle b\rangle_\mu\langle c\rangle_\mu}
\right)
\operatorname{p.v.}\operatorname{csch}\left(\frac{\pi p}{2\mu}\right)\\
&-\frac{ab+\mu^2}{4\pi}
\frac{\mathfrak m_-(b)\mathfrak m_+(c)}
{\langle\xi\rangle_\mu\langle a\rangle_\mu}
p\operatorname{p.v.}\operatorname{csch}\left(\frac{\pi p}{2\mu}\right).
\end{align*}

Therefore, we have 
\begin{align*}
\mathcal{G}_{3}^{p.v.}(t,\xi)
={}&-\frac{i}{16\pi}\int_{\mathbb R^3}e^{it\Phi_3}\partial_b\Phi_3(ab+\mu^2)
\left(
\frac{\mathfrak m_-(b)\mathfrak m_+(c)}
{\langle\xi\rangle_\mu\langle a\rangle_\mu}
-\frac{\mathfrak m_-(\xi)\mathfrak m_+(a)}
{\langle b\rangle_\mu\langle c\rangle_\mu}
\right)\nonumber\\
&\cdot \operatorname{p.v.}\operatorname{csch}\left(\frac{\pi p}{2\mu}\right)
\widetilde g(t,a)\widetilde g(t,b)\overline{\widetilde g(t,c)}
\,{\rm d}(a,b,c)\nonumber\\
&+\frac{i}{8\pi}\int_{\mathbb R^3}e^{it\Phi_3}
\frac{ab+\mu^2}{\langle\xi\rangle_\mu\langle a\rangle_\mu}
\mathfrak m_-(b)\mathfrak m_+(c)p\operatorname{p.v.}\operatorname{csch}\left(\frac{\pi p}{2\mu}\right)\nonumber\\
&\cdot \widetilde g(t,a)\widetilde g(t,b)\overline{\widetilde g(t,c)}
{\rm d}(a,b,c),
\end{align*}
which in the variables $(a,b,p)$ further writes as 
\begin{align}
\mathcal{G}_{3}^{p.v.}(t,\xi)=&-\frac{i}{16\pi}\int_{\mathbb R^3}e^{it\Phi_3}\partial_b\Phi_3(ab+\mu^2)
\left(
\frac{\mathfrak m_-(b)\mathfrak m_+(c)}
{\langle\xi\rangle_\mu\langle a\rangle_\mu}
-\frac{\mathfrak m_-(\xi)\mathfrak m_+(a)}
{\langle b\rangle_\mu\langle c\rangle_\mu}
\right)\nonumber\\
&\operatorname{p.v.}\operatorname{csch}\left(\frac{\pi p}{2\mu}\right)
 \widetilde g(t,a)\widetilde g(t,b)\overline{\widetilde g(t,c)}
\,{\rm d}(a,b,p)\nonumber\\
&+\frac{i}{8\pi}\int_{\mathbb R^3}e^{it\Phi_3}
\frac{ab+\mu^2}{\langle\xi\rangle_\mu\langle a\rangle_\mu}p
\mathfrak m_-(b)\mathfrak m_+(c)
\operatorname{p.v.}\operatorname{csch}\left(\frac{\pi p}{2\mu}\right)\nonumber\\
&\cdot \widetilde g(t,a)\widetilde g(t,b)\overline{\widetilde g(t,c)}
\,{\rm d}(a,b,p).\label{pvcti}
\end{align}
Integrating by parts in $b$ in the first term of
\eqref{pvcti}, we obtain
\begin{align*}
&\frac1{16\pi t}\int_{\mathbb R^3}
e^{it\Phi_3}
\operatorname{p.v.}\operatorname{csch}
\left(\frac{\pi p}{2\mu}\right)
(ab+\mu^2)
\left(
\frac{\mathfrak m_-(b)\mathfrak m_+(c)}
{\langle\xi\rangle_\mu\langle a\rangle_\mu}
-\frac{\mathfrak m_-(\xi)\mathfrak m_+(a)}
{\langle b\rangle_\mu\langle c\rangle_\mu}
\right)\\
&\qquad\qquad\qquad\qquad\cdot
\widetilde g(t,a)\partial_b\widetilde g(t,b)
\overline{\widetilde g(t,c)}
{\rm d}(a,b,p) 
+\{\text{similar and easier terms}\}.
\end{align*}
Using \eqref{Ftt}, Plancherel's identity, \eqref{a1}--\eqref{a3}, Corollary~\ref{cdfi} and Lemma~\ref{Deipd}, we get
\begin{equation}\label{pv3es}
\begin{aligned}
&\Bigg\|\frac1t\int_{\mathbb R^3}e^{it\Phi_3}\operatorname{p.v.}\operatorname{csch}\left(\frac{\pi p}{2\mu}\right)(ab+\mu^2)\frac{\mathfrak{m}_-(b)\mathfrak{m}_+(c)}{\langle\xi\rangle_\mu\langle a\rangle_\mu}\widetilde g(t,a)\partial_b\widetilde g(t,b)\overline{\widetilde g(t,c)}{\rm d}(a,b,p)\Bigg\|_{L^\infty_\xi}\\
\lesssim&\frac1t
\Bigg\|T_\eta\,
\mathcal F^{-1}\left(\frac{a}{\langle a\rangle_\mu}e^{-it\langle a\rangle_\mu^2}\widetilde g(t,a)\right)
\mathcal F^{-1}\left(b\mathfrak m_-(b)e^{-it\langle b\rangle_\mu^2}\partial_b\widetilde g(t,b)\right)\\
&\qquad\qquad\qquad\qquad\qquad\qquad\qquad\cdot
\overline{\mathcal F^{-1}\left(\mathfrak m_-(c)e^{-it\langle c\rangle_\mu^2}\widetilde g(t,c)\right)}
\Bigg\|_{L^1_x}+\{\text{easier term}\}\\
\lesssim{}&\frac1t
\Bigl\|\frac{a}{\langle a\rangle_\mu}\widetilde g(t,a)\Bigr\|_{L^2_a}
\Bigl\|b\partial_b\widetilde g(t,b)\Bigr\|_{L^2_b}
\Bigl\|\mathcal F^{-1}\left(\mathfrak m_-(c)e^{-it\langle c\rangle_\mu^2}\widetilde g(t,c)\right)\Bigr\|_{L^\infty_x}\\
\lesssim{}&\frac1t\|\widetilde g(t)\|_{L^2_\xi}\left(\|\partial_\xi\widetilde g(t)\|_{L^2_\xi}+\|\xi\partial_\xi\widetilde g(t)\|_{L^2_\xi}\right)\|D_\eta^{-1}z(t)\|_{W^{1,\infty}_x}\\
\lesssim{}&\varepsilon_1^3t^{-3/2+\delta}.
\end{aligned}
\end{equation}
For the second term of \eqref{pvcti}, the same argument as in \eqref{pv3es} gives
\begin{align*}
&\Bigg\|
\frac{i}{8\pi}\int_{\mathbb R^3}e^{it\Phi_3}
\frac{ab+\mu^2}{\langle\xi\rangle_\mu\langle a\rangle_\mu}
\mathfrak m_-(b)\mathfrak m_+(c)
p\operatorname{p.v.}\operatorname{csch}\left(\frac{\pi p}{2\mu}\right)
\widetilde g(t,a)\widetilde g(t,b)\overline{\widetilde g(t,c)}
\,{\rm d}(a,b,p)
\Bigg\|_{L^\infty_\xi}\\
\lesssim&
\Bigg\|
\frac{1}{\langle\xi\rangle_\mu}
\mathcal F\Bigg[
S_\eta^2\,
\mathcal F^{-1}\left(a\widehat{D_\eta^{-1}z}(t,a)\right)\,
\mathcal F^{-1}\left(b\langle b\rangle_\mu\mathfrak m_-(b)\widehat{D_\eta^{-1}z}(t,b)\right)
\overline{\mathcal F^{-1}\left(\langle c\rangle_\mu\mathfrak m_-(c)\widehat{D_\eta^{-1}z}(t,c)\right)}
\Bigg]
\Bigg\|_{L^\infty_\xi}\\
&+\{easier\ term\}\\
\lesssim&
\left\|\mathcal F^{-1}\left(a\widehat{D_\eta^{-1}z}(t,a)\right)\right\|_{L^\infty_x}
\left\|\mathcal F^{-1}\left(b\langle b\rangle_\mu\mathfrak m_-(b)\widehat{D_\eta^{-1}z}(t,b)\right)\right\|_{L^\infty_x}
\left\|\mathcal F^{-1}\left(\langle c\rangle_\mu\mathfrak m_-(c)\widehat{D_\eta^{-1}z}(t,c)\right)\right\|_{L^\infty_x}\\
\lesssim&
\|D_\eta^{-1}z(t)\|_{W^{1,\infty}_x}
\|D_\eta^{-1}z(t)\|_{W^{2,\infty}_x}
\|D_\eta^{-1}z(t)\|_{W^{1,\infty}_x}\\
\lesssim&
\varepsilon_1^3t^{-3/2}.
\end{align*}
\end{proof}

\subsection{Regular cubic terms}
We estimate the regular part of the original cubic interaction and the cubic terms generated by the quadratic normal form. We show that both contributions decay like \(\langle t\rangle^{-3/2}\) in \(L^\infty_\xi\), and hence belong to the integrable remainder in the profile equation.

Recall the definition of $\mathcal{G}^{r}_3$ in (\ref{cc})--(\ref{cc3}). And let 
\begin{align*}
\mathcal{G}_{3}^{nf}=-&\frac12\int_{\mathbb R^2}e^{it\Phi_{++}(\xi,a,b)}m_{++}(\xi,a,b)\left[\mathcal{G}_2(t,a)\widetilde g(t,b)+\widetilde g(t,a)\mathcal{G}_2(t,b)\right]{\rm d}a{\rm d}b\\
-&\frac12\int_{\mathbb R^2}e^{it\Phi_{+-}(\xi,a,b)}m_{+-}(\xi,a,b)\left[\mathcal{G}_2(t,a)\overline{\widetilde g(t,b)}+\widetilde g(t,a)\overline{\mathcal{G}_2(t,b)}\right]{\rm d}a{\rm d}b,
\end{align*}
where $\mathcal{G}_2$ is defined in (\ref{qc})--(\ref{qc1}), and $m_{++}$, $m_{+-}$ are defined in (\ref{qm}). 
\begin{prop}
We have $\|\mathcal{G}_3^{r}\|_{L^\infty_\xi}+\|\mathcal{G}_{3}^{nf}\|_{L^\infty_\xi}\lesssim\varepsilon_1^3\langle t\rangle^{-3/2}$.
\end{prop}

\begin{proof}
We first estimate $\|\mathcal{G}_3^{r}\|_{L^\infty_\xi}$. Corollary~\ref{cdfi} gives
\begin{align*}
\mathcal{G}_3^{r}(t,\xi)=&-\frac{i}{8\pi^2}e^{it\langle\xi\rangle_\mu^2}\int_{\mathbb R^4}e^{-i(\xi-a-b+c)x}\frac{S_\eta^2(x)}{\langle\xi\rangle_\mu}\sum_{j=0}^4\langle\xi\rangle_\mu\langle a\rangle_\mu\langle b\rangle_\mu\langle c\rangle_\mu p_3^j(\xi,a,b,c)T_\eta^j(x)\\
&\qquad\qquad\qquad\cdot \widehat{D_\eta^{-1}z}(t,a)\widehat{D_\eta^{-1}z}(t,b)\overline{\widehat{D_\eta^{-1}z}(t,c)}{\rm d}x{\rm d}a{\rm d}b{\rm d}c.
\end{align*}
Therefore, one has 
\begin{align*}
\|\mathcal{G}_3^{r}\|_{L^\infty_\xi}\lesssim&\|D_\eta^{-1}z(t)\|_{W^{2,\infty}_x}^3\lesssim\varepsilon_1^3\langle t\rangle^{-3/2}.
\end{align*}
We next prove that $\|\mathcal{G}_{3}^{nf}\|_{L^\infty_\xi}\lesssim\varepsilon_1^3\langle t\rangle^{-3/2}$. Substituting $\mathcal{G}_2$ into $\mathcal{G}_3^{nf}$, we obtain
\begin{equation}
\begin{aligned}
\mathcal{G}_3^{nf}(t,\xi)=&-\frac{i}{2}\int_{\mathbb R^3}e^{it\Phi_{+++}(\xi,a,b,c)}q_{+++}(\xi,a,b,c)\widetilde g(t,a)\widetilde g(t,b)\widetilde g(t,c)\,{\rm d}a{\rm d}b{\rm d}c\\
-&\frac{i}{2}\int_{\mathbb R^3}e^{it\Phi_{++-}(\xi,a,b,c)}q_{++-}(\xi,a,b,c)\widetilde g(t,a)\widetilde g(t,b)\overline{\widetilde g(t,c)}\,{\rm d}a{\rm d}b{\rm d}c\\
-&\frac{i}{2}\int_{\mathbb R^3}e^{it\Phi_{+--}(\xi,a,b,c)}q_{+--}(\xi,a,b,c)\widetilde g(t,a)\overline{\widetilde g(t,b)}\overline{\widetilde g(t,c)}\,{\rm d}a{\rm d}b{\rm d}c,
\end{aligned}
\label{G3nf}
\end{equation}
where
$$
\Phi_{+++}(\xi,a,b,c)=\langle\xi\rangle_\mu^2-\langle a\rangle_\mu^2-\langle b\rangle_\mu^2-\langle c\rangle_\mu^2,\qquad
\Phi_{++-}(\xi,a,b,c)=\langle\xi\rangle_\mu^2-\langle a\rangle_\mu^2-\langle b\rangle_\mu^2+\langle c\rangle_\mu^2,
$$
and
$$
\Phi_{+--}(\xi,a,b,c)=\langle\xi\rangle_\mu^2-\langle a\rangle_\mu^2+\langle b\rangle_\mu^2+\langle c\rangle_\mu^2.
$$
Moreover,
\begin{align*}
q_{+++}(\xi,a,b,c)=&\frac12\int_{\mathbb R}\Bigl(m_{++}(\xi,r,c)q_{++}(r,a,b)+m_{++}(\xi,a,r)q_{++}(r,b,c)\Bigr){\rm d}r,\\
q_{++-}(\xi,a,b,c)=&\frac12\int_{\mathbb R}\Bigl(m_{++}(\xi,r,b)q_{+-}(r,a,c)+m_{++}(\xi,a,r)q_{+-}(r,b,c)\\
&\qquad\qquad+m_{+-}(\xi,r,c)q_{++}(r,a,b)-m_{+-}(\xi,a,r)\overline{q_{+-}(r,c,b)}\Bigr){\rm d}r,\\
q_{+--}(\xi,a,b,c)=&\frac12\int_{\mathbb R}\Bigl(m_{+-}(\xi,r,c)q_{+-}(r,a,b)-m_{+-}(\xi,a,r)\overline{q_{++}(r,b,c)}\Bigr){\rm d}r.
\end{align*}
We first claim that
$$
|q_{+++}(\xi,a,b,c)|\lesssim_{\mu}\frac{\langle\xi\rangle_\mu}{\langle a\rangle_\mu\langle b\rangle_\mu\langle c\rangle_\mu}(1+|a|+|b|+|c|)^2e^{-\frac{\pi}{4\mu}|\xi-a-b-c|}.
$$
Indeed, applying \eqref{quker} to $q_{+++}$, we obtain
\begin{align*}
q_{+++}(\xi,a,b,c)=&\frac{\langle\xi\rangle_\mu}{64\pi\langle a\rangle_\mu\langle b\rangle_\mu\langle c\rangle_\mu}\int_{\mathbb R}(r^2-a^2-b^2-\mu^2)\operatorname{sech}\left(\frac{\pi(\xi-r-c)}{2\mu}\right)\operatorname{sech}\left(\frac{\pi(r-a-b)}{2\mu}\right){\rm d}r\\
+&\frac{\langle\xi\rangle_\mu}{64\pi\langle a\rangle_\mu\langle b\rangle_\mu\langle c\rangle_\mu}\int_{\mathbb R}(r^2-b^2-c^2-\mu^2)\operatorname{sech}\left(\frac{\pi(\xi-a-r)}{2\mu}\right)\operatorname{sech}\left(\frac{\pi(r-b-c)}{2\mu}\right){\rm d}r.
\end{align*}
In the first integral we make the change of variables $r=a+b+s$, and in the second $r=b+c+s$. This gives
\begin{align*}
q_{+++}(\xi,a,b,c)=&\frac{\langle\xi\rangle_\mu}{32\pi\langle a\rangle_\mu\langle b\rangle_\mu\langle c\rangle_\mu}\int_{\mathbb R}\Bigl[ab+bc-\mu^2+(a+2b+c)s+s^2\Bigr]\\
&\qquad\qquad\cdot\operatorname{sech}\left(\frac{\pi(\xi-a-b-c-s)}{2\mu}\right)\operatorname{sech}\left(\frac{\pi s}{2\mu}\right){\rm d}s.
\end{align*}
Therefore, there holds 
\begin{align*}
|q_{+++}(\xi,a,b,c)|
\lesssim&\frac{\langle\xi\rangle_\mu}{\langle a\rangle_\mu\langle b\rangle_\mu\langle c\rangle_\mu}(1+|a|+|b|+|c|)^2\int_{\mathbb R}(1+|s|)^2e^{-\frac{\pi}{2\mu}|\xi-a-b-c-s|}e^{-\frac{\pi}{2\mu}|s|}{\rm d}s\\
\lesssim&\frac{\langle\xi\rangle_\mu}{\langle a\rangle_\mu\langle b\rangle_\mu\langle c\rangle_\mu}(1+|a|+|b|+|c|)^2e^{-\frac{\pi}{4\mu}|\xi-a-b-c|}\\
&\qquad\qquad\qquad\qquad\qquad\qquad\cdot\int_{\mathbb R}(1+|s|)^2e^{-\frac{\pi}{4\mu}|\xi-a-b-c-s|}e^{-\frac{\pi}{4\mu}|s|}{\rm d}s\\
\lesssim&\frac{\langle\xi\rangle_\mu}{\langle a\rangle_\mu\langle b\rangle_\mu\langle c\rangle_\mu}(1+|a|+|b|+|c|)^2e^{-\frac{\pi}{4\mu}|\xi-a-b-c|}.
\end{align*}
The same argument yields 
\begin{align*}
&q_{++-}(\xi,a,b,c)\\
=&-\frac{\langle\xi\rangle_\mu\langle a\rangle_\mu}{32\pi\langle b\rangle_\mu\langle c\rangle_\mu}
\int_{\mathbb R}\frac{(a-c+s)^2-a^2+c^2+\mu^2}{\langle a-c+s\rangle_\mu^2}
\operatorname{sech}\left(\frac{\pi(\xi-a-b+c-s)}{2\mu}\right)
\operatorname{sech}\left(\frac{\pi s}{2\mu}\right){\rm d}s\\
&-\frac{\langle\xi\rangle_\mu\langle b\rangle_\mu}{32\pi\langle a\rangle_\mu\langle c\rangle_\mu}
\int_{\mathbb R}\frac{(b-c+s)^2-b^2+c^2+\mu^2}{\langle b-c+s\rangle_\mu^2}
\operatorname{sech}\left(\frac{\pi(\xi-a-b+c-s)}{2\mu}\right)
\operatorname{sech}\left(\frac{\pi s}{2\mu}\right){\rm d}s\\
&-\frac{1}{32\pi\langle\xi\rangle_\mu\langle a\rangle_\mu\langle b\rangle_\mu\langle c\rangle_\mu}
\int_{\mathbb R}\langle a+b+s\rangle_\mu^2\bigl((a+b+s)^2-a^2-b^2-\mu^2\bigr)\\
&\qquad\qquad\qquad\cdot
\operatorname{sech}\left(\frac{\pi(\xi-a-b+c-s)}{2\mu}\right)
\operatorname{sech}\left(\frac{\pi s}{2\mu}\right){\rm d}s\\
&-\frac{\langle a\rangle_\mu\langle c\rangle_\mu}{16\pi\langle\xi\rangle_\mu\langle b\rangle_\mu}
\int_{\mathbb R}\frac{(c-b-s)^2-c^2+b^2+\mu^2}{\langle c-b-s\rangle_\mu^2}
\operatorname{sech}\left(\frac{\pi(\xi-a-b+c-s)}{2\mu}\right)
\operatorname{sech}\left(\frac{\pi s}{2\mu}\right){\rm d}s
\end{align*}
and
\begin{align*}
q_{+--}(\xi,a,b,c)=&\frac{\langle a\rangle_\mu}{32\pi\langle\xi\rangle_\mu\langle b\rangle_\mu\langle c\rangle_\mu}
\int_{\mathbb R}\Bigl[2\bigl((a-b+s)^2-a^2+b^2+\mu^2\bigr)
+\bigl((b+c-s)^2-b^2-c^2-\mu^2\bigr)\Bigr]\\
&\qquad\qquad\qquad\times\operatorname{sech}\left(\frac{\pi(\xi-a+b+c-s)}{2\mu}\right)
\operatorname{sech}\left(\frac{\pi s}{2\mu}\right){\rm d}s.
\end{align*}
Therefore,
\begin{align*}
|q_{+--}(\xi,a,b,c)|\lesssim&\frac{\langle a\rangle_\mu}{\langle\xi\rangle_\mu\langle b\rangle_\mu\langle c\rangle_\mu}(1+|a|+|b|+|c|)^2e^{-\frac{\pi}{4\mu}|\xi-a+b+c|},
\end{align*}
and \begin{align*}
|q_{++-}(\xi,a,b,c)|\lesssim&e^{-\frac{\pi}{4\mu}|\xi-a-b+c|}\Bigg(\frac{\langle\xi\rangle_\mu\langle a\rangle_\mu}{\langle b\rangle_\mu\langle c\rangle_\mu}(1+|a|+|b|+|c|)^2+\frac{\langle\xi\rangle_\mu\langle b\rangle_\mu}{\langle a\rangle_\mu\langle c\rangle_\mu}(1+|a|+|b|+|c|)^2\\
&\qquad\qquad+\frac{(1+|a|+|b|)^4}{\langle\xi\rangle_\mu\langle a\rangle_\mu\langle b\rangle_\mu\langle c\rangle_\mu}+\frac{\langle a\rangle_\mu\langle c\rangle_\mu}{\langle\xi\rangle_\mu\langle b\rangle_\mu}(1+|a|+|b|+|c|)^2\Bigg).
\end{align*}
By  Corollary~\ref{cdfi}, the first term of \eqref{G3nf} is a finite sum of terms of the form
\begin{equation*}
\begin{aligned}
&\langle\xi\rangle_\mu\int_{\mathbb R^3}a^{k_1}b^{k_2}c^{k_3}
\left[\int_{\mathbb R}s^j\operatorname{sech}\left(\frac{\pi(\xi-a-b-c-s)}{2\mu}\right)\operatorname{sech}\left(\frac{\pi s}{2\mu}\right){\rm d}s\right]\\
&\qquad\qquad\qquad\qquad\qquad\qquad\qquad\qquad\qquad\qquad\cdot\widehat{D_\eta^{-1}z}(t,a)\widehat{D_\eta^{-1}z}(t,b)\widehat{D_\eta^{-1}z}(t,c)\,{\rm d}a{\rm d}b{\rm d}c,
\end{aligned}
\end{equation*}
where $0\leq j\leq2$ and $k_1+k_2+k_3+j\leq2$.\\
By \eqref{Fotrseta}, we have
$$
\mathcal F^{-1}\left[\int_{\mathbb R}s^j\operatorname{sech}\left(\frac{\pi(\,\cdot-s\,)}{2\mu}\right)\operatorname{sech}\left(\frac{\pi s}{2\mu}\right){\rm d}s\right]=2\mu^2\sqrt{\frac2\pi}i^{-j}S_\eta\partial_x^jS_\eta,\quad 0\leq j\leq2.
$$
Hence
\begin{align*}
&\left\|\langle\xi\rangle_\mu\mathcal F\left[S_\eta\partial_x^jS_\eta
\left(\partial_x^{k_1}D_\eta^{-1}z\right)
\left(\partial_x^{k_2}D_\eta^{-1}z\right)
\left(\partial_x^{k_3}D_\eta^{-1}z\right)\right]\right\|_{L^\infty_\xi}\\
\lesssim&
\left\|S_\eta\partial_x^jS_\eta
\left(\partial_x^{k_1}D_\eta^{-1}z\right)
\left(\partial_x^{k_2}D_\eta^{-1}z\right)
\left(\partial_x^{k_3}D_\eta^{-1}z\right)\right\|_{W^{1,1}_x}\\
\lesssim&
\|D_\eta^{-1}z(t)\|_{W^{3,\infty}_x}^3
\lesssim\varepsilon_1^3\langle t\rangle^{-3/2}.
\end{align*}
Similarly,
$$\left\|\int_{\mathbb R^3}e^{it\Phi_{++-}}q_{++-}\widetilde g(t,a)\widetilde g(t,b)\overline{\widetilde g(t,c)}\,{\rm d}a{\rm d}b{\rm d}c\right\|_{L^\infty_\xi}\lesssim\|D_\eta^{-1}z(t)\|_{W^{6,\infty}_x}^3,$$
and
$$\left\|\int_{\mathbb R^3}e^{it\Phi_{+--}}q_{+--}\widetilde g(t,a)\overline{\widetilde g(t,b)}\overline{\widetilde g(t,c)}\,{\rm d}a{\rm d}b{\rm d}c\right\|_{L^\infty_\xi}\lesssim\|D_\eta^{-1}z(t)\|_{W^{4,\infty}_x}^3.$$
Therefore, we conclude 
\begin{align*}
\|\mathcal G_3^{nf}(t)\|_{L^\infty_\xi}
\lesssim&\|D_\eta^{-1}z(t)\|_{W^{3,\infty}_x}^3+\|D_\eta^{-1}z(t)\|_{W^{6,\infty}_x}^3+\|D_\eta^{-1}z(t)\|_{W^{4,\infty}_x}^3
\lesssim\varepsilon_1^3\langle t\rangle^{-3/2}.
\end{align*}
\end{proof}

\subsection{Modulation and higher order terms}
The aim of this section is to control the remaining terms in \eqref{nfpe}.

Recall the definition of $\mathcal{G}_{\rm mod}$ in (\ref{mc}) and its decomposition into the sum of $\mathcal{G}^{r}_{\rm mod}$ and $\mathcal{G}^s_{\rm mod}$, which are defined  in (\ref{mr})--(\ref{ms}). 
\begin{prop}
Let $\rho$ and $\delta_*$ be the constants fixed in (\ref{Nota}). Then, for every $t\geq0$,
$$\left\|\langle\xi\rangle_\mu^k\mathcal{G}_{\rm mod}(t)\right\|_{L^2_\xi}\lesssim\varepsilon_1^3\langle t\rangle^{-1-\delta_*},\quad0\leq k\leq N-2,$$
and
$$\left\|\xi^k\mathcal{G}_{\rm mod}(t)\right\|_{L^\infty_\xi}\lesssim\varepsilon_1^3\langle t\rangle^{-1-\rho},\quad0\leq k\leq N-3.$$
\end{prop}

\begin{proof}
We first prove the weighted $L^2_\xi$ estimate.  Applying Proposition~\ref{disfour} and the Leibniz rule to the definition of $\mathcal{G}_{\rm mod}$, and using \eqref{paraimp} and Propositions~\ref{z-decay} and~\ref{Sobes}, we obtain, for $0\leq k\leq N-2$, 
\begin{align*}
\left\|\langle\xi\rangle_\mu^k\mathcal{G}_{\rm mod}(t)\right\|_{L^2_\xi}
&\lesssim\bigl(|\dot y|+|\dot\phi|\bigr)\Bigg(\left\|\langle\xi\rangle_\mu^{k-1}\left[-i\xi\mathcal F(\partial_xz)-\frac{\mu}{\sqrt{2\pi}}\widehat{T_\eta}*\mathcal F(\partial_xz)\right]\right\|_{L^2_\xi}\\
&\qquad+\left\|\langle\xi\rangle_\mu^{k-1}\left[-i\xi\mathcal F(T_\eta z)-\frac{\mu}{\sqrt{2\pi}}\widehat{T_\eta}*\mathcal F(T_\eta z)\right]\right\|_{L^2_\xi}\\
&\qquad+\left\|\langle\xi\rangle_\mu^{k-1}\left[-i\xi\mathcal F(S_\eta z^2)-\frac{\mu}{\sqrt{2\pi}}\widehat{T_\eta}*\mathcal F(S_\eta z^2)\right]\right\|_{L^2_\xi}\Bigg)\\
&\lesssim\bigl(|\dot y|+|\dot\phi|\bigr)\left(\|\partial_xz\|_{H^k_x}+\|T_\eta z\|_{H^k_x}+\|S_\eta z^2\|_{H^k_x}\right)\\
&\lesssim\varepsilon_1^3\langle t\rangle^{-1-\delta_*}.
\end{align*}
A further application of Proposition~\ref{disfour} gives for $0\leq k\leq N-3$ that 
$$
e^{it\langle\xi\rangle_\mu^2}\widetilde{\mathcal F}(\partial_xz)(t,\xi)=i\xi\widetilde g(t,\xi)+\frac{\mu^2e^{it\langle\xi\rangle_\mu^2}}{\langle\xi\rangle_\mu}\widehat{S_\eta^2z}(t,\xi).
$$
For $0\leq k\leq N-3$, notice that by Proposition~\ref{disfour} again
\begin{align*}
\Bigl\|\xi^ke^{it\langle\xi\rangle_\mu^2}\widetilde{\mathcal F}(\partial_xz)(t)\Bigr\|_{L^\infty_\xi}
&\lesssim\Bigl\|\xi^{k+1}\widetilde g(t)\Bigr\|_{L^\infty_\xi}
+\Bigl\|\frac{\xi^k}{\langle\xi\rangle_\mu}\widehat{S_\eta^2z}(t)\Bigr\|_{L^\infty_\xi}
\lesssim\varepsilon_1,\\[1ex]
\left\|\xi^ke^{it\langle\xi\rangle_\mu^2}\widetilde{\mathcal F}(S_\eta z^2)\right\|_{L^\infty_\xi}
&=\left\|\frac{\xi^k}{\langle\xi\rangle_\mu}\mathcal{F}\Bigl(D_\eta^*(S_\eta z^2)\Bigr)\right\|_{L^\infty_\xi}
\lesssim\varepsilon_1^2\langle t\rangle^{-1}.
\end{align*}
These estimates imply 
$\left\|\xi^k\mathcal{G}_{\rm mod}^{r}\right\|_{L^\infty_\xi}\lesssim\varepsilon_1^3\langle t\rangle^{-1-\rho}, 0\leq k\leq N-3.$

It remains to dominate $\left\|\xi^k\mathcal{G}_{\rm mod}^{s}\right\|_{L^\infty_\xi}$. For the second term of \eqref{sms},  when $0\le k\le N-3$, we have
\begin{align*}
&\sup_{\xi\in\mathbb R}\left|\xi^k\int_{\mathbb R}e^{it(\langle\xi\rangle_\mu^2-\langle a\rangle_\mu^2)}\frac{(\langle\xi\rangle_\mu-\langle a\rangle_\mu)^2}{\langle\xi\rangle_\mu\langle a\rangle_\mu}\operatorname{csch}\left(\frac{\pi(\xi-a)}{2\mu}\right)\widetilde g(t,a)\,da\right|\\
\lesssim&\sup_{\xi\in\mathbb R}\int_{\mathbb R}\left(|a|^k|\xi-a|^2+|\xi-a|^{k+2}\right)\left|\operatorname{csch}\left(\frac{\pi(\xi-a)}{2\mu}\right)\right||\widetilde g(t,a)|{\rm d}a\\
\lesssim&\|a^k\widetilde g(t,a)\|_{L^\infty_a}\int_{\mathbb R}|r|^2\left|\operatorname{csch}\left(\frac{\pi r}{2\mu}\right)\right|{\rm d}r+\|\widetilde g(t,a)\|_{L^\infty_a}\int_{\mathbb R}|r|^{k+2}\left|\operatorname{csch}\left(\frac{\pi r}{2\mu}\right)\right|{\rm d}r\\
\lesssim&\|a^k\widetilde g(t,a)\|_{L^\infty_a}+\|\widetilde g(t,a)\|_{L^\infty_a}\\
\lesssim&\varepsilon_1.
\end{align*}
For the first term in \eqref{sms}, we claim that
$$
\left\|\xi^k\operatorname{p.v.}\int_{\mathbb R}e^{it(\langle\xi\rangle_\mu^2-\langle a\rangle_\mu^2)}\operatorname{csch}\left(\frac{\pi(\xi-a)}{2\mu}\right)\widetilde g(t,a)\,da\right\|_{L^\infty_\xi}\lesssim\varepsilon_1\langle t\rangle^\delta, 0\leq k\leq N-3.
$$
The change of variables $p=\xi-a$ gives
$$\operatorname{p.v.}\int_{\mathbb R}e^{it(\langle\xi\rangle_\mu^2-\langle a\rangle_\mu^2)}\operatorname{csch}\left(\frac{\pi(\xi-a)}{2\mu}\right)\widetilde g(t,a)\,{\rm d}a=\operatorname{p.v.}\int_{\mathbb R}e^{it(2\xi p-p^2)}\operatorname{csch}\left(\frac{\pi p}{2\mu}\right)\widetilde g(t,\xi-p)\,{\rm d}p.$$
Splitting $\widetilde g(t,\xi-p)$ into the sum of $\widetilde g(t,\xi)$ and $\widetilde g(t,\xi-p)-\widetilde g(t,\xi)$, we thus have 
\begin{equation}\label{xipv}
\begin{aligned}
&\xi^k\operatorname{p.v.}\int_{\mathbb R}e^{it(2\xi p-p^2)}\operatorname{csch}\left(\frac{\pi p}{2\mu}\right)\widetilde g(t,\xi-p){\rm d}p\\
=&\xi^k\widetilde g(t,\xi)\operatorname{p.v.}\int_{\mathbb R}e^{it(2\xi p-p^2)}\operatorname{csch}\left(\frac{\pi p}{2\mu}\right){\rm d}p+\xi^k\int_{\mathbb R}e^{it(2\xi p-p^2)}\operatorname{csch}\left(\frac{\pi p}{2\mu}\right)\bigl(\widetilde g(t,\xi-p)-\widetilde g(t,\xi)\bigr){\rm d}p.
\end{aligned}
\end{equation}
For the second term of \eqref{xipv}, using
$$
\xi^k\bigl(\widetilde g(t,\xi-p)-\widetilde g(t,\xi)\bigr)=\bigl(\xi^k-(\xi-p)^k\bigr)\widetilde g(t,\xi-p)+\bigl((\xi-p)^k\widetilde g(t,\xi-p)-\xi^k\widetilde g(t,\xi)\bigr),
$$
the fundamental theorem of calculus and the Cauchy--Schwarz inequality, we obtain 
\begin{align*}
&\left\|\xi^k\int_{\mathbb R}e^{it(2\xi p-p^2)}\operatorname{csch}\left(\frac{\pi p}{2\mu}\right)\bigl(\widetilde g(t,\xi-p)-\widetilde g(t,\xi)\bigr){\rm d}p\right\|_{L^\infty_\xi}\\
\le&\sup_{\xi\in\mathbb R}\int_{\mathbb R}\left|\operatorname{csch}\left(\frac{\pi p}{2\mu}\right)\right|\left|\xi^k-(\xi-p)^k\right|\left|\widetilde g(t,\xi-p)\right|{\rm d}p+\left\|\partial_\xi\bigl(\xi^k\widetilde g(t,\xi)\bigr)\right\|_{L^2_\xi}\int_{\mathbb R}|p|^{1/2}\left|\operatorname{csch}\left(\frac{\pi p}{2\mu}\right)\right|{\rm d}p\\
\lesssim&\varepsilon_1\langle t\rangle^\delta.
\end{align*}
Thus it suffices to estimate the first term of \eqref{xipv}.
Let $\chi\in C_c^\infty(\mathbb R, [0,1])$ be an even function such that $\chi(p)=1$ for $|p|\leq1$ and $\chi(p)=0$ for $|p|\ge2$. 
We decompose
$$
\operatorname{csch}\left(\frac{\pi p}{2\mu}\right)=\frac{2\mu}{\pi}\frac{\chi(p)}{p}+\left[\operatorname{csch}\left(\frac{\pi p}{2\mu}\right)-\frac{2\mu}{\pi}\frac{\chi(p)}{p}\right].
$$
Since $\operatorname{csch}\left(\frac{\pi p}{2\mu}\right)-\frac{2\mu}{\pi}\frac{\chi(p)}{p}\in L^1(\mathbb R)$, it remains to prove
$$
\sup_{t\geq0}\sup_{\xi\in\mathbb R}\left|\operatorname{p.v.}\int_{\mathbb R}e^{it(2\xi p-p^2)}\frac{\chi(p)}{p}{\rm d}p\right|\lesssim1.
$$
This  bound is immediate for $t\le 1$, indeed 
\begin{align*}
&\left|\operatorname{p.v.}\int_{\mathbb R}e^{it(2\xi p-p^2)}\frac{\chi(p)}{p}{\rm d}p\right|\\
\lesssim&\left|\int_0^2e^{-itp^2}\chi(p)\frac{\sin(2t\xi p)}{p}{\rm d}p\right|\\
\lesssim&\left|\int_0^2\frac{\sin(2t\xi p)}{p}{\rm d}p\right|+\int_0^2\frac{|\chi(p)-1|}{p}{\rm d}p+\int_0^2\frac{|e^{-itp^2}-1|}{p}{\rm d}p\\
\lesssim&1.
\end{align*}
For $t\ge1$, we split
\begin{equation}
\label{cpde}
\int_{\mathbb R}e^{it(2\xi p-p^2)}\frac{\chi(p)}{p}{\rm d}p
=\int_{\mathbb R}e^{it(2\xi p-p^2)}\frac{1}{p}{\rm d}p
+\int_{\mathbb R}e^{it(2\xi p-p^2)}\frac{\chi(p)-1}{p}{\rm d}p.
\end{equation}
For the first term of \eqref{cpde}, setting $x=\sqrt{t}p$, $\lambda=2\sqrt{t}\xi$, and using \eqref{bfpi}, we obtain
$$
\sup_{t\ge1,\,\xi\in\mathbb R}\left|\operatorname{p.v.}\int_{\mathbb R}e^{it(2\xi p-p^2)}\frac{{\rm d}p}{p}\right|\lesssim1.
$$
For the second term of \eqref{cpde}, we split again
$$
\int_{\mathbb R}e^{it(2\xi p-p^2)}\frac{\chi(p)-1}{p}{\rm d}p=\left(\int_{|\xi-p|<1}+\int_{|\xi-p|\ge1}\right)e^{it(2\xi p-p^2)}\frac{\chi(p)-1}{p}{\rm d}p.
$$
The first regime integral is immediately dominated as 
$$
\left|\int_{|\xi-p|<1}e^{it(2\xi p-p^2)}\frac{\chi(p)-1}{p}{\rm d}p\right|\lesssim\left\|\frac{\chi(p)-1}{p}\right\|_{L^\infty_p}\lesssim1.
$$
In the region $|\xi-p|\geq1$, integration by parts in $p$ yields
\begin{align*}
&\left|\int_{|\xi-p|\geq1}e^{it(2\xi p-p^2)}\frac{\chi(p)-1}{p}{\rm d}p\right|\\
\lesssim&\frac{1}{t}\left(\left\|\frac{\chi(p)-1}{p}\right\|_{L^\infty_p}+\int_{|\xi-p|\geq1}\frac{1}{|\xi-p|}\left|\partial_p\left(\frac{\chi(p)-1}{p}\right)\right|{\rm d}p+\int_{|\xi-p|\geq1}\frac{1}{|\xi-p|^2}\left|\frac{\chi(p)-1}{p}\right|{\rm d}p\right)\\
\lesssim&\frac{1}{t}\left(\left\|\frac{\chi(p)-1}{p}\right\|_{L^\infty_p}+\left\|\partial_p\left(\frac{\chi(p)-1}{p}\right)\right\|_{L^1_p}+\left\|\frac{\chi(p)-1}{p}\right\|_{L^\infty_p}\int_{|\xi-p|\geq1}\frac{{\rm d}p}{|\xi-p|^2}\right)\\
\lesssim&1.
\end{align*}
Therefore,  the proposition is proved.
\end{proof}
\begin{prop}

Let $\rho$ be the constants fixed in (\ref{Nota}). Define
\begin{align*}
\mathcal{G}_{\rm mod}^{\rm nf}(t,\xi)=-&\frac12\int_{\mathbb R^2}e^{it\Phi_{++}(\xi,a,b)}m_{++}(\xi,a,b)\mathcal{G}_{\rm mod}(t,a)\widetilde g(t,b){\rm d}a{\rm d}b\\
-&\frac12\int_{\mathbb R^2}e^{it\Phi_{++}(\xi,a,b)}m_{++}(\xi,a,b)\widetilde g(t,a)\mathcal{G}_{\rm mod}(t,b){\rm d}a{\rm d}b\\
-&\frac12\int_{\mathbb R^2}e^{it\Phi_{+-}(\xi,a,b)}m_{+-}(\xi,a,b)\mathcal{G}_{\rm mod}(t,a)\overline{\widetilde g(t,b)}{\rm d}a{\rm d}b\\
-&\frac12\int_{\mathbb R^2}e^{it\Phi_{+-}(\xi,a,b)}m_{+-}(\xi,a,b)\widetilde g(t,a)\overline{\mathcal{G}_{\rm mod}(t,b)}{\rm d}a{\rm d}b,
\end{align*}
where $\Phi_{++}, \Phi_{+-}$ are defined in (\ref{mpj}), and $m_{++}$, $m_{+-}$ are defined in (\ref{qm}) respectively. 
Then
$$
\|\mathcal{G}_{\rm mod}^{\rm nf}\|_{L^\infty_\xi}\lesssim\varepsilon_1^4\langle t\rangle^{-1-\rho}.
$$
\end{prop}

\begin{proof}
The four integrals are estimated in the same way. We only consider the first one.
Set
$$
G_{\rm mod}=\dot y\left(\partial_x z-\frac{\mu}{4}S_\eta z^2\right)-i(\dot\phi+\eta\dot y)\left(T_\eta z-\frac14S_\eta z^2\right).
$$
Then $\mathcal{G}_{\rm mod}(t,a)=e^{it\langle a\rangle_\mu^2}(\widetilde{\mathcal F}G_{\rm mod})(t,a).$
Using the definition of $m_{++}$, Corollary~\ref{cdfi}, and \eqref{Fotrseta}, we have
\begin{align*}
&\int_{\mathbb R^2}e^{it\Phi_{++}(\xi,a,b)}m_{++}(\xi,a,b)(\partial_t\widetilde g)_{\rm mod}(t,a)\widetilde g(t,b){\rm d}a{\rm d}b\\
=&\frac{e^{it\langle\xi\rangle_\mu^2}}{4\sqrt{2\pi}}\langle\xi\rangle_\mu\int_{\mathbb R^2}\operatorname{sech}\left(\frac{\pi(\xi-a-b)}{2\mu}\right)\frac{(\widetilde{\mathcal F}G_{\rm mod})(t,a)}{\langle a\rangle_\mu}\frac{(\widetilde{\mathcal F}z)(t,b)}{\langle b\rangle_\mu}{\rm d}a{\rm d}b\\
=&\frac{\mu}{2}e^{it\langle\xi\rangle_\mu^2}\langle\xi\rangle_\mu\mathcal F\left((S_\eta D_\eta^{-1}P_eG_{\rm mod})(D_\eta^{-1}z)\right)(t,\xi).
\end{align*}
Then Lemma~\ref{Deipd}   yields
\begin{align*}
&\left\|\int_{\mathbb R^2}e^{it\Phi_{++}(\xi,a,b)}m_{++}(\xi,a,b)(\partial_t\widetilde g)_{\rm mod}(t,a)\widetilde g(t,b){\rm d}a{\rm d}b\right\|_{L^\infty_\xi}\\
\lesssim&\left\|S_\eta D_\eta^{-1}P_eG_{\rm mod}D_\eta^{-1}z\right\|_{W^{1,1}_x}\\
\lesssim&\left\|D_\eta^{-1}P_eG_{\rm mod}\right\|_{W^{1,\infty}_x}\left\|D_\eta^{-1}z\right\|_{W^{1,\infty}_x}\\
\lesssim&\|G_{\rm mod}\|_{L^\infty_x}\|z\|_{L^\infty_x}\\
\lesssim&\bigl(|\dot y|+|\dot\phi|\bigr)\left(\|\partial_xz\|_{L^\infty_x}+\|z\|_{L^\infty_x}+\|z\|_{L^\infty_x}^2\right)\|z\|_{L^\infty_x}\\
\lesssim&\varepsilon_1^4\langle t\rangle^{-1-\rho}.
\end{align*}
The remaining three integrals satisfy the same bound. This proves the proposition.
\end{proof}
Recall the definition of $\mathcal{G}_{\ge 4}$ in (\ref{hc}) and (\ref{nonlinear}). By (\ref{001}), we get
$$\|\mathcal{G}_{\ge4}(t)\|_{L^\infty_\xi}\lesssim\|\mathcal N_{\ge4}(t)\|_{L^1_x}.$$
Moreover, using (\ref{v12}), Proposition~\ref{z-decay} and Proposition~\ref{Sobes}, we obtain
$$
\|\mathcal N_{\ge4}(t)\|_{L^1_x}\lesssim\varepsilon_1^4\langle t\rangle^{-3/2}.
$$

\begin{prop}
Define
\begin{align*}
\mathcal{G}_{\ge4}^{\rm nf}(t,\xi)=&\frac{i}{2}\int_{\mathbb R^2}e^{it\Phi_{++}(\xi,a,b)}m_{++}(\xi,a,b)e^{it\langle a\rangle_\mu^2}\widetilde{\mathcal F}\left(\mathcal N_3+\mathcal N_{\ge4}\right)(t,a)\widetilde g(t,b){\rm d}a{\rm d}b\\
+&\frac{i}{2}\int_{\mathbb R^2}e^{it\Phi_{++}(\xi,a,b)}m_{++}(\xi,a,b)\widetilde g(t,a)e^{it\langle b\rangle_\mu^2}\widetilde{\mathcal F}\left(\mathcal N_3+\mathcal N_{\ge4}\right)(t,b){\rm d}a{\rm d}b\\
+&\frac{i}{2}\int_{\mathbb R^2}e^{it\Phi_{+-}(\xi,a,b)}m_{+-}(\xi,a,b)e^{it\langle a\rangle_\mu^2}\widetilde{\mathcal F}\left(\mathcal N_3+\mathcal N_{\ge4}\right)(t,a)\overline{\widetilde g(t,b)}{\rm d}a{\rm d}b\\
-&\frac{i}{2}\int_{\mathbb R^2}e^{it\Phi_{+-}(\xi,a,b)}m_{+-}(\xi,a,b)\widetilde g(t,a)e^{-it\langle b\rangle_\mu^2}\overline{\widetilde{\mathcal F}\left(\mathcal N_3+\mathcal N_{\ge4}\right)(t,b)}{\rm d}a{\rm d}b,
\end{align*}
where $\Phi_{++},\Phi_{+-}$ and $m_{++}$, $m_{+-}$ are defined in (\ref{mpj}) and (\ref{qm}) respectively, and $\mathcal{N}_3,\mathcal{N}_{\ge 4}$ are defined in (\ref{nonlinear}). 
Then
$$
\|\mathcal{G}_{\ge4}^{\rm nf}\|_{L^\infty_\xi}\lesssim\varepsilon_1^4\langle t\rangle^{-2}.
$$
\end{prop}

\begin{proof}
The four integrals are estimated in the same way, so it suffices to consider the first one.
Note that
\begin{align*}
&\left\|\int_{\mathbb R^2}e^{it\Phi_{++}(\xi,a,b)}m_{++}(\xi,a,b)e^{it\langle a\rangle_\mu^2}\widetilde{\mathcal F}\left(\mathcal N_3+\mathcal N_{\ge4}\right)(t,a)\widetilde g(t,b)\,{\rm d}a\,{\rm d}b\right\|_{L^\infty_\xi}\\
\lesssim&\left\|\left(S_\eta D_\eta^{-1}P_e\left(\mathcal N_3+\mathcal N_{\ge4}\right)\right)\left(D_\eta^{-1}z\right)\right\|_{W^{1,1}_x}\\
\lesssim&\left\|D_\eta^{-1}P_e\left(\mathcal N_3+\mathcal N_{\ge4}\right)\right\|_{W^{1,\infty}_x}\left\|D_\eta^{-1}z\right\|_{W^{1,\infty}_x}\\
\lesssim&\left\|\mathcal N_3+\mathcal N_{\ge4}\right\|_{L^\infty_x}\|z\|_{L^\infty_x}\\
\lesssim&\varepsilon_1^4\langle t\rangle^{-2}.
\end{align*}
The remaining three integrals satisfy the same bound. This proves the proposition.
\end{proof}

Recall $\widetilde{h}$ introduced in (\ref{nfpr}) and its evolution equation (\ref{nfpe}). Combining the preceding estimates in this subsection, we obtain
\begin{equation}\label{odef}
\partial_t\widetilde h(t,\xi) = \frac{i}{4t}\langle\xi\rangle_\mu^2|\widetilde g(t,\xi)|^2\widetilde g(t,\xi)+\mathcal R(t,\xi).
\end{equation}
where
$$
\|\mathcal R(t)\|_{L^\infty_\xi}\lesssim_N\varepsilon_1^3 t^{-1-\rho},\quad t\ge1.
$$
We now turn to prove the modified scattering for the radiation $u$.

\subsection{Modified scattering}\label{submodi}

The aim of this subsection is to prove \eqref{L2aspthm}.  Note that
\begin{align*}
\|u(t)-\mathbf W(t)\|_{L^2_x}
\le&\|u(t)-z(t)\|_{L^2_x}\\
+&\left\|z(t,x)-\frac{e^{ix^2/(4t)-i\mu^2t-i\pi/4}}{\sqrt{2t}}\frac{i\zeta-\mu T_\eta(x)}{\langle\zeta\rangle_\mu}\widetilde g(t,\zeta)\right\|_{L^2_x}\\
+&\left\|\frac{e^{ix^2/(4t)-i\mu^2t-i\pi/4}}{\sqrt{2t}}\frac{i\zeta-\mu T_\eta(x)}{\langle\zeta\rangle_\mu}
\left(\widetilde g(t,\zeta)-e^{\frac{i}{4}\langle\zeta\rangle_\mu^2|W_\infty(\zeta)|^2\log t}W_\infty(\zeta)\right)\right\|_{L^2_x}.
\end{align*}
The change of variables \eqref{u-z} and the estimates \eqref{zinfdecay} and \eqref{Hmze} give
\begin{equation*}
\|u(t)-z(t)\|_{L^2_x}\lesssim\|z(t)\|_{L^\infty_x}^2\|z(t)\|_{L^2_x}\lesssim\varepsilon_1^3\langle t\rangle^{-1}.
\end{equation*}
Combining the representation \eqref{profile-inv} with the formula \eqref{gee}  yields
$$
z(t,x)=\frac{e^{-i\mu^2t}}{\sqrt{2\pi}}\int_{\mathbb R}e^{ix\xi-it\xi^2}\frac{i\xi-\mu T_\eta(x)}{\langle\xi\rangle_\mu}\widetilde g(t,\xi){\rm d}\xi,
$$
which can be rewritten as
$$
z(t,x)=e^{-i\mu^2t}e^{it\partial_x^2}\mathcal F^{-1}\left(\frac{i\xi\widetilde g(t,\xi)}{\langle\xi\rangle_\mu}\right)(x)-\mu T_\eta(x)e^{-i\mu^2t}e^{it\partial_x^2}\mathcal F^{-1}\left(\frac{\widetilde g(t,\xi)}{\langle\xi\rangle_\mu}\right)(x).
$$
By (\ref{freesch}), for $t\ge 1$, we get
$$
\left\|e^{it\partial_x^2}\mathcal F^{-1}\left(\frac{i\xi\widetilde g(t,\xi)}{\langle\xi\rangle_\mu}\right)-\frac{e^{ix^2/(4t)-i\pi/4}}{\sqrt{2t}}\frac{i\zeta\widetilde g(t,\zeta)}{\langle\zeta\rangle_\mu}\right\|_{L^2_x}
\lesssim t^{-1/2}\left\|\partial_\xi\left(\frac{i\xi\widetilde g(t,\xi)}{\langle\xi\rangle_\mu}\right)\right\|_{L^2_\xi},
$$
and
$$
\left\|e^{it\partial_x^2}\mathcal F^{-1}\left(\frac{\widetilde g(t,\xi)}{\langle\xi\rangle_\mu}\right)-\frac{e^{ix^2/(4t)-i\pi/4}}{\sqrt{2t}}\frac{\widetilde g(t,\zeta)}{\langle\zeta\rangle_\mu}\right\|_{L^2_x}
\lesssim t^{-1/2}\left\|\partial_\xi\left(\frac{\widetilde g(t,\xi)}{\langle\xi\rangle_\mu}\right)\right\|_{L^2_\xi}.
$$
Therefore, it suffices to show that
\begin{equation}\label{gmodsca}
\left\|e^{-\frac{i}{4}\langle\xi\rangle_\mu^2|W_\infty(\xi)|^2\log t}\widetilde g(t,\xi)-W_\infty(\xi)\right\|_{L^2_\xi}\longrightarrow0, \ as \ t\to\infty.
\end{equation}
Indeed, 
\begin{align*}
&\Big\|m(t)-R_{\phi_\infty}\Big(\omega_*+\operatorname{Re}\mathbf W(t)n_*+\operatorname{Im}\mathbf W(t)p_*\Big)(\cdot-y_\infty)\Big\|_{L^2_x}\\
\lesssim&
\|u(t)-\mathbf W(t)\|_{L^2_x}
+\|\alpha(t)\|_{L^2_x}
+\bigl(|y(t)-y_\infty|+|\phi(t)-\phi_\infty|\bigr)\bigl(1+\|u(t)\|_{H^1_x}\bigr).
\end{align*}

For $t\geq1$, we define
$$
\Theta(t,\xi)=\frac14\langle\xi\rangle_\mu^2\int_1^t\frac{|\widetilde g(s,\xi)|^2}{s}\,\mathrm ds.
$$
Therefore \eqref{odef} yields
\begin{align*}
\partial_t\left(e^{-i\Theta(t,\xi)}\widetilde h(t,\xi)\right)
&=e^{-i\Theta(t,\xi)}\left(\mathcal R(t,\xi)+\frac{i}{4t}\langle\xi\rangle_\mu^2|\widetilde g(t,\xi)|^2\bigl(\widetilde g(t,\xi)-\widetilde h(t,\xi)\bigr)\right).
\end{align*}
By \eqref{nfpr}, \eqref{qm}, Corollary~\ref{cdfi}, and \eqref{Fotrseta}, we obtain
$$
-\frac12\int_{\mathbb R^2}e^{it\Phi_{++}(\xi,a,b)}m_{++}(\xi,a,b)\widetilde g(t,a)\widetilde g(t,b)\,{\rm d}a\,{\rm d}b=-\frac{\mu}{4}e^{it\langle\xi\rangle_\mu^2}\langle\xi\rangle_\mu\mathcal F\left[S_\eta\bigl(D_\eta^{-1}z(t)\bigr)^2\right](\xi).
$$
The other quadratic contribution satisfies
$$
-\frac12\int_{\mathbb R^2}e^{it\Phi_{+-}(\xi,a,b)}m_{+-}(\xi,a,b)\widetilde g(t,a)\overline{\widetilde g(t,b)}\,{\rm d}a\,{\rm d}b=\frac{\mu}{2}e^{it\langle\xi\rangle_\mu^2}\langle\xi\rangle_\mu^{-1}\mathcal F\left[S_\eta(D_\eta^*z(t))\overline{D_\eta^{-1}z(t)}\right](\xi).
$$
Therefore (\ref{nfpr}) gives 
\begin{align*}
\widetilde h(t,\xi)-\widetilde g(t,\xi)
=&-\frac{\mu}{4}e^{it\langle\xi\rangle_\mu^2}\langle\xi\rangle_\mu\mathcal F\left(S_\eta\bigl(D_\eta^{-1}z(t)\bigr)^2\right)(\xi)\\
&+\frac{\mu}{2}e^{it\langle\xi\rangle_\mu^2}\langle\xi\rangle_\mu^{-1}\mathcal F\left(S_\eta(D_\eta^*z(t))\overline{D_\eta^{-1}z(t)}\right)(\xi).
\end{align*}
Using Proposition~\ref{Deipd} and \eqref{Linw}, we have
\begin{align*}
\|\langle\xi\rangle_\mu^2(\widetilde h(t)-\widetilde g(t))\|_{L^\infty_\xi}
&\lesssim\sum_{k=0}^3\left\|\partial_x^k\left[S_\eta\bigl(D_\eta^{-1}z(t)\bigr)^2\right]\right\|_{L^1_x}+\sum_{k=0}^1\left\|\partial_x^k\left[S_\eta(D_\eta^*z(t))\overline{D_\eta^{-1}z(t)}\right]\right\|_{L^1_x}\\
&\lesssim\|D_\eta^{-1}z(t)\|_{W^{3,\infty}_x}^2+\|D_\eta^*z(t)\|_{W^{1,\infty}_x}\|D_\eta^{-1}z(t)\|_{W^{1,\infty}_x}\\
&\lesssim\varepsilon_1^2\langle t\rangle^{-1}.
\end{align*}
Hence  for any $t\geq1$ 
\begin{align*}
\left\|\partial_t\left(e^{-i\Theta(t)}\widetilde h(t)\right)\right\|_{L^\infty_\xi}
&\lesssim\varepsilon_1^3t^{-1-\rho}+\frac1t\|\widetilde g(t)\|_{L^\infty_\xi}^2\|\langle\xi\rangle_\mu^2(\widetilde g(t)-\widetilde h(t))\|_{L^\infty_\xi}\lesssim\varepsilon_1^3t^{-1-\rho}.
\end{align*}
Letting
$$
w_\infty(\xi)=\widetilde h(1,\xi)+\int_1^\infty\partial_s\left(e^{-i\Theta(s,\xi)}\widetilde h(s,\xi)\right){\rm d}s,
$$
we find
\begin{equation}\label{ghwinf}
\begin{aligned}
\left\|e^{-i\Theta(t)}\widetilde g(t)-w_\infty\right\|_{L^\infty_\xi}
&\le\left\|\widetilde g(t)-\widetilde h(t)\right\|_{L^\infty_\xi} + \left\|e^{-i\Theta(t)}\widetilde h(t)-w_\infty\right\|_{L^\infty_\xi}\\
&\lesssim\varepsilon_1^2t^{-\rho}.
\end{aligned}
\end{equation}
The uniform bound \eqref{gLlt} gives
\begin{equation}
\label{ginfde}
|\widetilde g(t,\xi)|+|w_\infty(\xi)|\lesssim\varepsilon_1\langle\xi\rangle^{-(N-2)},
\end{equation}
which together with \eqref {ghwinf}  implies
\begin{equation}
\label{igwest}
\begin{aligned}
&\left\|\int_1^\infty\frac1s\langle\xi\rangle_\mu^2\left(|\widetilde g(s)|^2-|w_\infty|^2\right)\,{\rm d}s\right\|_{L^\infty_\xi}\\
\le{}&\int_1^\infty\frac1s\left\|\langle\xi\rangle_\mu^2\left(|\widetilde g(s)|+|w_\infty|\right)\right\|_{L^\infty_\xi}
\left\||\widetilde g(s)|-|w_\infty|\right\|_{L^\infty_\xi}\,{\rm d}s\\
\le{}&\int_1^\infty\frac1s\left\|\langle\xi\rangle_\mu^2\left(|\widetilde g(s)|+|w_\infty|\right)\right\|_{L^\infty_\xi}
\left\|e^{-i\Theta(s)}\widetilde g(s)-w_\infty\right\|_{L^\infty_\xi}\,{\rm d}s\\
\lesssim{}&\ \varepsilon_1^3,
\end{aligned}
\end{equation}
since $N\ge 6$.
We may thus define
$$
\Gamma_\infty(\xi)=\frac14\langle\xi\rangle_\mu^2\int_1^\infty\frac{|\widetilde g(s,\xi)|^2-|w_\infty(\xi)|^2}{s}{\rm d}s.
$$
Then
$$\Theta(t,\xi)=\frac14\langle\xi\rangle_\mu^2|w_\infty(\xi)|^2\log t+\Gamma_\infty(\xi)-r(t,\xi),$$
where
$$
r(t,\xi) = \frac14\langle\xi\rangle_\mu^2\int_t^\infty\frac{|\widetilde g(s,\xi)|^2-|w_\infty(\xi)|^2}{s}{\rm d}s.
$$
The argument used for \eqref{igwest} also gives
\begin{equation}\label{rees}
\begin{aligned}
\|r(t)\|_{L^\infty_\xi}
&\le\frac14\int_t^\infty\frac1s\left\|\langle\xi\rangle_\mu^2\left(|\widetilde g(s)|^2-|w_\infty|^2\right)\right\|_{L^\infty_\xi}{\rm d}s\\
&\lesssim\int_t^\infty\left(\varepsilon_1^3s^{-2}+\varepsilon_1^4s^{-1-\rho}\right){\rm d}s\\
&\lesssim\varepsilon_1^3t^{-\rho},\quad N\ge 6.
\end{aligned}
\end{equation}
Let $W_\infty(\xi) = e^{i\Gamma_\infty(\xi)}w_\infty(\xi),$ note that 
\begin{align*}
e^{-\frac{i}{4}\langle\xi\rangle_\mu^2|W_\infty(\xi)|^2\log t}\widetilde g(t,\xi)=e^{i\Gamma_\infty(\xi)}e^{-ir(t,\xi)}e^{-i\Theta(t,\xi)}\widetilde g(t,\xi).
\end{align*}
By 
\eqref{ghwinf} and \eqref{rees}, we obtain
\begin{align*}
&\bigl\|e^{-\frac{i}{4}\langle\xi\rangle_\mu^2|W_\infty|^2\log t}\widetilde g(t)-W_\infty\bigr\|_{L^\infty_\xi}\\
\lesssim&\bigl\|r(t)\|_{L^\infty_\xi}\|\widetilde g(t)\bigr\|_{L^\infty_\xi} + \bigr\| e^{-i\Theta(t)}\widetilde g(t)-w_\infty\bigr\|_{L^\infty_\xi}\\
\lesssim&\varepsilon_1^2t^{-\rho}.
\end{align*}
The bound \eqref{ginfde} gives
$$
\left|e^{-\frac{i}{4}\langle\xi\rangle_\mu^2|W_\infty(\xi)|^2\log t}\widetilde g(t,\xi)-W_\infty(\xi)\right|\lesssim\varepsilon_1\langle\xi\rangle^{-(N-2)}.
$$
The dominated convergence theorem therefore yields \eqref{gmodsca}.

\appendix
\section{Conservation laws}
\label{appcsv}

In this appendix, we verify directly the conservation of the two quantities \eqref{H2eta} and \eqref{H4eta}. Recall that, under the rotation $\mathbf M=R_{\eta x}m$, the quantities \eqref{H2eta} and \eqref{H4eta} become
\begin{align*}
\mathcal H_\eta^1(m)&=\int_{\mathbb R}\left(|\mathbf M_x|^2+\mu^2(1-\mathbf M_1^2)\right){\rm d}x,\\
\mathcal H_\eta^2(m)&=\int_{\mathbb R}\left(|\mathbf M_{xx}|^2-\frac54|\mathbf M_x|^4+\mu^2\left(\frac32\mathbf M_1^2|\mathbf M_x|^2-|\partial_x\mathbf M_1|^2\right)+\frac{\mu^4}{4}(1-\mathbf M_1^4)\right){\rm d}x.
\end{align*}
Using \eqref{LLM} and integrating by parts, we obtain
\begin{align*}
\frac{{\rm d}}{{\rm d}t}\mathcal H_\eta^1(m(t))&=-2\int_{\mathbb R}(\mathbf M_{xx}+\mu^2\mathbf M_1e_1)\cdot\mathbf M_t{\rm d}x\\
&=-2\int_{\mathbb R}(\mathbf M_{xx}+\mu^2\mathbf M_1e_1)\cdot\left[\mathbf M\wedge(\mathbf M_{xx}+\mu^2\mathbf M_1e_1)\right]{\rm d}x=0.
\end{align*}
For the second quantity, integrating by parts in $x$ gives
\begin{align*}
\frac{{\rm d}}{{\rm d}t}\mathcal H_\eta^2(m(t))={}&\int_{\mathbb R}\mathbf M_t\cdot\Big[2\mathbf M_{xxxx}+5\partial_x(|\mathbf M_x|^2\mathbf M_x)\\
&\qquad+\mu^2\big(3\mathbf M_1|\mathbf M_x|^2e_1-3\partial_x(\mathbf M_1^2\mathbf M_x)+2(\partial_x^2\mathbf M_1)e_1\big)-\mu^4\mathbf M_1^3e_1\Big]{\rm d}x.
\end{align*}
Using \eqref{LLM}, $|\mathbf M|=1$, and integrating by parts in $x$, we obtain
\begin{align*}
2\int_{\mathbb R}\mathbf M_t\cdot\mathbf M_{xxxx}{\rm d}x
=&2\int_{\mathbb R}\mathbf M_{xxxx}\cdot(\mathbf M\wedge\mathbf M_{xx}){\rm d}x
+2\mu^2\int_{\mathbb R}\mathbf M_1\mathbf M_{xxxx}\cdot(\mathbf M\wedge e_1){\rm d}x\\
=&-2\int_{\mathbb R}\mathbf M_{xxx}\cdot(\mathbf M_x\wedge\mathbf M_{xx}){\rm d}x
+2\mu^2\int_{\mathbb R}\mathbf M_{xx}\cdot\partial_x^2\bigl(\mathbf M_1(\mathbf M\wedge e_1)\bigr){\rm d}x\\
=&5\int_{\mathbb R}(\partial_x|\mathbf M_x|^2)\mathbf M\cdot(\mathbf M_x\wedge\mathbf M_{xx}){\rm d}x
-2\mu^2\int_{\mathbb R}(\partial_x^2\mathbf M_1)\partial_x\bigl[\mathbf M\cdot(\mathbf M_x\wedge e_1)\bigr]{\rm d}x\\
&-4\mu^2\int_{\mathbb R}(\partial_x\mathbf M_1)\Bigl[\mathbf M_1\mathbf M\cdot(\mathbf M_x\wedge\mathbf M_{xx})
+|\mathbf M_x|^2\mathbf M\cdot(\mathbf M_x\wedge e_1)\Bigr]{\rm d}x.
\end{align*}
Next,
\begin{align*}
5\int_{\mathbb R}\mathbf M_t\cdot\partial_x(|\mathbf M_x|^2\mathbf M_x){\rm d}x
={}&-5\int_{\mathbb R}(\partial_x|\mathbf M_x|^2)\mathbf M\cdot(\mathbf M_x\wedge\mathbf M_{xx}){\rm d}x\\
-5\mu^2\int_{\mathbb R}\mathbf M_1&\Bigl[(\partial_x|\mathbf M_x|^2)\mathbf M\cdot(\mathbf M_x\wedge e_1)+|\mathbf M_x|^2\partial_x\bigl[\mathbf M\cdot(\mathbf M_x\wedge e_1)\bigr]\Bigr]{\rm d}x.
\end{align*}
Furthermore, 
\begin{align*}
3\mu^2\int_{\mathbb R}\mathbf M_1|\mathbf M_x|^2\mathbf M_t\cdot e_1{\rm d}x
=3\mu^2\int_{\mathbb R}\mathbf M_1|\mathbf M_x|^2\partial_x\bigl[\mathbf M\cdot(\mathbf M_x\wedge e_1)\bigr]{\rm d}x.
\end{align*}
Similarly,
\begin{align*}
-3\mu^2\int_{\mathbb R}\mathbf M_t\cdot\partial_x(\mathbf M_1^2\mathbf M_x){\rm d}x
={}&6\mu^2\int_{\mathbb R}\mathbf M_1(\partial_x\mathbf M_1)\mathbf M\cdot(\mathbf M_x\wedge\mathbf M_{xx}){\rm d}x\\
&+6\mu^4\int_{\mathbb R}\mathbf M_1^2(\partial_x\mathbf M_1)\mathbf M\cdot(\mathbf M_x\wedge e_1){\rm d}x\\
&+3\mu^4\int_{\mathbb R}\mathbf M_1^3\partial_x\bigl[\mathbf M\cdot(\mathbf M_x\wedge e_1)\bigr]{\rm d}x.
\end{align*}
Moreover,
$$
2\mu^2\int_{\mathbb R}(\partial_x^2\mathbf M_1)\mathbf M_t\cdot e_1{\rm d}x
=2\mu^2\int_{\mathbb R}(\partial_x^2\mathbf M_1)\partial_x\bigl[\mathbf M\cdot(\mathbf M_x\wedge e_1)\bigr]{\rm d}x.
$$
Finally,
$$
-\mu^4\int_{\mathbb R}\mathbf M_1^3\mathbf M_t\cdot e_1{\rm d}x
=-\mu^4\int_{\mathbb R}\mathbf M_1^3\partial_x\bigl[\mathbf M\cdot(\mathbf M_x\wedge e_1)\bigr]{\rm d}x,
$$
Combining the preceding identities and using
$$
|\mathbf M_x|^2\partial_x\bigl[\mathbf M\cdot(\mathbf M_x\wedge e_1)\bigr]
=\frac12(\partial_x|\mathbf M_x|^2)\mathbf M\cdot(\mathbf M_x\wedge e_1)
-(\partial_x\mathbf M_1)\mathbf M\cdot(\mathbf M_x\wedge\mathbf M_{xx}),
$$
we obtain
\begin{align*}
\frac{{\rm d}}{{\rm d}t}\mathcal H_\eta^2(m(t))
=&-4\mu^2\int_{\mathbb R}\partial_x\Bigl(\mathbf M_1|\mathbf M_x|^2\mathbf M\cdot(\mathbf M_x\wedge e_1)\Bigr){\rm d}x\\
&+2\mu^4\int_{\mathbb R}\partial_x\Bigl(\mathbf M_1^3\mathbf M\cdot(\mathbf M_x\wedge e_1)\Bigr){\rm d}x=0.
\end{align*}

\section{Local well-posedness}
\label{applwp}

For an integer $k\geq1$, and $i<j$,  we define $\nabla_\eta = \partial_x+\eta e_1\wedge $ and
\begin{align*}
\mathcal E_\eta^k(\mathbb R) &=\left\{m\in L^1_{\mathrm{loc}}(\mathbb R;\mathbb S^2): (m_2,m_3)\in L^2(\mathbb R)\times L^2(\mathbb R),\ 
\nabla_\eta m\in H^{k-1}(\mathbb R)\right\},\\
\mathcal E^k_{i,j}(\mathbb R)&=\left\{m\in L^1_{\mathrm{loc}}(\mathbb R;\mathbb S^2):(m_i,m_j)\in L^2(\mathbb R)\times L^2(\mathbb R),\ 
\partial_x m\in H^{k-1}(\mathbb R)\right\}.
\end{align*}
Following the energy space metric used in \cite{LG2018}, and after a permutation of the coordinate components, for $\{i,j,\ell\}=\{1,2,3\}$, we equip $\mathcal E_{i,j}^k(\mathbb R)$ with the metric induced by the norm
$$\|v\|_{\mathcal E_{i,j}^k} = \left( \|v_i\|_{H^k}^2 +\|v_j\|_{H^k}^2 +\|v_\ell\|_{L^\infty}^2+\|\partial_xv_\ell\|_{H^{k-1}}^2\right)^{1/2},$$
namely,
$$d_{\mathcal E_{i,j}^k}(m,\widetilde m)=\|m-\widetilde m\|_{\mathcal E_{i,j}^k}.$$
We equip $\mathcal E_\eta^k(\mathbb R)$ with the pullback metric
$$d_{\mathcal E_\eta^k}(m,\widetilde m)=d_{\mathcal E_{2,3}^k}\bigl(R_{\eta x}m,R_{\eta x}\widetilde m\bigr).$$
\begin{thm}
\label{LWPth}
Let $|\eta|<1$ and let $N\ge6$ be an integer. Assume that
$m_0:\mathbb R\to\mathbb S^2, m_0-\omega_*\in H^N(\mathbb R).$
Then there exist a maximal time $T_{\max}>0$ and a unique solution
$m:\mathbb R\times[0,T_{\max})\to\mathbb S^2$
to
$$
\partial_t m
=
m\wedge
\left(
\partial_x^2m
+2\eta e_1\wedge\partial_xm
+m_1e_1
\right),
\qquad
m(0)=m_0.
$$
Moreover, the following statements hold.

\begin{enumerate}
\item[(i)]
For every $0<T<T_{\max}$, $m\in L^\infty\bigl([0,T];\mathcal E_\eta^N(\mathbb R)\bigr), \partial_t m
\in L^\infty\bigl([0,T];H^{N-2}(\mathbb R)\bigr).$

\item[(ii)]
If $T_{\max}<\infty$, then
$$
\int_0^{T_{\max}}
\|\nabla_\eta m(t)\|_{L^\infty(\mathbb R)}^2\,dt =
\infty.
$$

\item[(iii)]
For every $0<T<T_{\max}$, the flow map $m_0\longmapsto m$
is well defined and locally Lipschitz continuous from $\mathcal E_\eta^N(\mathbb R)$
to $C\bigl([0,T];\mathcal E_\eta^{N-1}(\mathbb R)\bigr).$

\item[(iv)] If $\ell>N$ is an integer and $m_0-\omega_*\in H^\ell(\mathbb R)$, then, for every $0<T<T_{\max}$, $m\in L^\infty\bigl([0,T];\mathcal E_\eta^\ell(\mathbb R)\bigr)$, $\partial_t m
\in L^\infty\bigl([0,T];H^{\ell-2}(\mathbb R)\bigr).$
\end{enumerate}
\end{thm}

\begin{proof}
Let $n=P\mathbf M,$ where
$$P=
\begin{pmatrix}
0&1&0\\
1&0&0\\
0&0&1
\end{pmatrix}.$$

A direct computation gives
$$\partial_tn
+
n\wedge\left(\partial_x^2n-\mu^2n_1e_1-\mu^2n_3e_3\right)=0.$$
Thus the equation is exactly the anisotropic Landau--Lifshitz equation considered in Theorem~5 of \cite{LG2018}, with $\lambda_1=\lambda_3=\mu^2>0.$

Since $n_0-n_*=PR_{\eta x}(m_0-\omega_*)\in H^N(\mathbb R)$ and $n_*=(S_\eta,T_\eta,0)\in\mathcal E^N_{1,3}(\mathbb R)$, we have $n_0\in\mathcal E^N_{1,3}(\mathbb R)$.
Therefore, Theorem~5 of \cite{LG2018} applies to $n$ and yields all the conclusions above.
\end{proof}

The same argument yields the following local well-posedness result for $\mathbf M$.
\begin{thm}
\label{thmMlw}
Let $N\ge6$ be an integer. Assume that
$\mathbf M_0:\mathbb R\to\mathbb S^2$ and $\mathbf M_0-\Omega_*\in H^N(\mathbb R)$.
Then there exist a maximal time $T_{\max}>0$ and a unique solution
$\mathbf M:\mathbb R\times[0,T_{\max})\to\mathbb S^2$
to
$$\partial_t \mathbf M=\mathbf M\wedge\left(\partial_x^2\mathbf M+\mu^2\mathbf M_1e_1\right),\quad\mathbf M(0)=\mathbf M_0.$$
Moreover, the following statements hold.

\begin{enumerate}
\item[(i)]
For every $0<T<T_{\max}$,
$\mathbf M\in L^\infty\bigl([0,T];\mathcal E^N_{2,3}(\mathbb R)\bigr)$ and
$\partial_t \mathbf M\in L^\infty\bigl([0,T];H^{N-2}(\mathbb R)\bigr)$.

\item[(ii)]
If $T_{\max}<\infty$, then
$$\int_0^{T_{\max}}\|\partial_x\mathbf M(t)\|_{L^\infty(\mathbb R)}^2{\rm d}t=\infty.$$

\item[(iii)]
For every $0<T<T_{\max}$, the flow map $\mathbf M_0\longmapsto \mathbf M$
is well defined and locally Lipschitz continuous from $\mathcal E^N_{2,3}(\mathbb R)$ to $C\bigl([0,T];\mathcal E^{N-1}_{2,3}(\mathbb R)\bigr)$.

\item[(iv)]
If $\ell>N$ is an integer and $\mathbf M_0-\Omega_*\in H^\ell(\mathbb R)$,
then, for every $0<T<T_{\max}$, $\mathbf M\in L^\infty\bigl([0,T];\mathcal E^\ell_{2,3}(\mathbb R)\bigr)$ and $\partial_t \mathbf M\in L^\infty\bigl([0,T];H^{\ell-2}(\mathbb R)\bigr)$.
\end{enumerate}
\end{thm}

\section{Global well-posedness}
\label{appgwp}

\begin{thm}
Let $|\eta|<1$ and let $N\ge6$ be an integer. Assume that
$m_0:\mathbb R\to\mathbb S^2$ and $m_0-\omega_*\in H^N(\mathbb R)$.
Then the solution given by Theorem~\ref{LWPth} is global, namely, $T_{\max}=\infty$.
Moreover, for every $T>0$, $m\in L^\infty\bigl([0,T];\mathcal E_\eta^N(\mathbb R)\bigr),\partial_t m\in L^\infty\bigl([0,T];H^{N-2}(\mathbb R)\bigr).$
\end{thm}

\begin{proof}
It is enough to prove that
$$
\sup_{0\leq t<T_{\max}}\|\mathbf M_x(t)\|_{L^\infty}<\infty.
$$
By the conservation of \eqref{H2eta}, together with
$\mathbf M_x=R_{\eta x}\nabla_\eta m$ and $\mathbf M_1=m_1$, we obtain
$$
\|\mathbf M_x(t)\|_{L^2}^2
\le
\mathcal H_\eta^1(m(t))
=
\mathcal H_\eta^1(m_0),
\quad
0\leq t<T_{\max}.
$$
Moreover, $\mathbf M_{xx}=R_{\eta x}\nabla_\eta^2m, \partial_x\mathbf M_1=\partial_xm_1,$
so the conservation of \eqref{H4eta} gives
$$\mathcal H_\eta^2(m_0)
=\|\mathbf M_{xx}(t)\|_{L^2}^2-\frac54\|\mathbf M_x(t)\|_{L^4}^4 +\mu^2\int
\left(\frac32\mathbf M_1^2|\mathbf M_x|^2-|\partial_x\mathbf M_1|^2\right)
+\frac{\mu^4}{4}\int(1-\mathbf M_1^4).$$
Therefore,
$$
\|\mathbf M_{xx}(t)\|_{L^2}^2
\lesssim
|\mathcal H_\eta^2(m_0)|+\|\mathbf M_x(t)\|_{L^4}^4+\|\mathbf M_x(t)\|_{L^2}^2.
$$
The one dimensional Gagliardo--Nirenberg inequality gives
$$
\|\mathbf M_x(t)\|_{L^4}^4\lesssim\|\mathbf M_x(t)\|_{L^2}^3\|\mathbf M_{xx}(t)\|_{L^2}.
$$
Using
$$
\|\mathbf M_x(t)\|_{L^2}^3\|\mathbf M_{xx}(t)\|_{L^2}\lesssim\|\mathbf M_{xx}(t)\|_{L^2}^2 + \|\mathbf M_x(t)\|_{L^2}^6,
$$
we obtain
$$
\|\mathbf M_{xx}(t)\|_{L^2}^2\lesssim |\mathcal H_\eta^2(m_0)| + \|\mathbf M_x(t)\|_{L^2}^6 + \|\mathbf M_x(t)\|_{L^2}^2.
$$
Therefore one has 
$$
\|\mathbf M_{xx}(t)\|_{L^2}^2
\lesssim
|\mathcal H_\eta^2(m_0)|
+\bigl(\mathcal H_\eta^1(m_0)\bigr)^3
+\mathcal H_\eta^1(m_0).
$$
Applying the one-dimensional Gagliardo--Nirenberg inequality once again to $\mathbf M_x$, we obtain
$$
\|\mathbf M_x(t)\|_{L^\infty}^2
\lesssim\|\mathbf M_x(t)\|_{L^2}\|\mathbf M_{xx}(t)\|_{L^2}
<\infty.
$$
\end{proof}

The same argument yields the following global well-posedness result for $\mathbf M$.
\begin{thm}
\label{gwpM}
Let $N\ge6$ be an integer. Assume that
$\mathbf M_0:\mathbb R\to\mathbb S^2$ and $\mathbf M_0-\Omega_*\in H^N(\mathbb R)$.
Then the solution given by Theorem~\ref{thmMlw} is global,
namely, $T_{\max}=\infty$.
Moreover, for every $T>0$,
$\mathbf M\in L^\infty\bigl([0,T];\mathcal E^N_{2,3}(\mathbb R)\bigr)$ and
$\partial_t \mathbf M\in L^\infty\bigl([0,T];H^{N-2}(\mathbb R)\bigr)$.
\end{thm}

\section{Basic identities and Fourier transform formulas}
We collect here several elementary identities and Fourier transform formulas used throughout the paper.
\begin{align}
S_\eta^2+T_\eta^2&=1,
\label{STide}\\
\partial_xS_\eta&=-\mu S_\eta T_\eta,
\label{bSederi}\\
\partial_xT_\eta&=\mu S_\eta^2.
\label{Tederi}
\end{align}
The standard Fresnel integrals are
\begin{equation*}
\int_{\mathbb R}e^{-ix^2}{\rm d}x=\sqrt{\pi}e^{-i\pi/4},
\end{equation*}
and
\begin{equation*}
\int_{\mathbb R}e^{ix^2}{\rm d}x=\sqrt{\pi}e^{i\pi/4}.
\end{equation*}
Moreover,
\begin{equation}
\label{bfpi}
\operatorname{p.v.}\int_{\mathbb R}e^{i(\lambda x-x^2)}\frac{{\rm d}x}{x}
=i\sqrt{\pi}e^{-i\pi/4}\int_0^\lambda e^{is^2/4}{\rm d}s.
\end{equation}
With the definition of the Fourier transform given in Section~\ref{2}, we have
\begin{equation}
\label{Fotrseta}
\widehat{S_\eta}(p)
=
\sqrt{\frac{\pi}{2}}\frac{1}{\mu}
\operatorname{sech}\left(\frac{\pi p}{2\mu}\right),
\qquad p\in\mathbb R,
\end{equation}
and
\begin{equation}
\label{Ftsq}
\widehat{S_\eta^2}(p)
=
\sqrt{\frac{\pi}{2}}\frac{p}{\mu^2}
\operatorname{csch}\left(\frac{\pi p}{2\mu}\right),
\qquad p\in\mathbb R.
\end{equation}
Furthermore,
\begin{equation}
\label{Ftone}
\widehat{1}(p)=\sqrt{2\pi}\,\delta(p)
\qquad\text{in }\mathcal S'(\mathbb R),
\end{equation}
and
\begin{equation}
\label{Ftt}
\widehat{T_\eta}(p)
=-i\sqrt{\frac{\pi}{2}}\frac{1}{\mu}
\operatorname{p.v.}\operatorname{csch}\left(\frac{\pi p}{2\mu}\right)
\quad\text{in }\mathcal S'(\mathbb R).
\end{equation}
Finally,
\begin{equation}
\label{Teseta}
\widehat{T_\eta S_\eta^2}(p)
= -\frac{i}{2}\sqrt{\frac{\pi}{2}}\frac{p^2}{\mu^3}
\operatorname{csch}\left(\frac{\pi p}{2\mu}\right),
\quad p\in\mathbb R.
\end{equation}

\end{document}